%% file: main.tex
\documentclass[a4paper,11pt]{amsart}

\input{mycommands}

\makeatletter

\renewcommand{\l@section}{%
  \@tocline{1}{6pt}{0pt}{}{\bfseries}%
}

\renewcommand{\l@subsection}{%
  \@tocline{2}{0pt}{1.5em}{}{\normalfont}%
}

\makeatother

\hypersetup{hidelinks}

\title[Hölder Regularity of the Top Lyapunov Exponent]{Hölder regularity of the top Lyapunov exponent for random matrix products with low-dimensional equator}

\author[A. C. Barreto]{Ana Cristina Barreto}
\address{Universidade Federal do Ceará. Ceará, Brazil.}
\email{anabarretoaraujo@mat.ufc.br}

\author[A. Sánchez]{Adriana Sánchez}
\address{Centro de investigación de Matemática Pura y Aplicada, Universidad de Costa Rica. San José, Costa Rica.}
\email{adriana.sanchez\_c@ucr.ac.cr}

\author[E. Y. Tall]{El Hadji Yaya Tall}
\address{Centro de investigación de Matemática Pura y Aplicada, Universidad de Costa Rica. San José, Costa Rica.}
\email{el.tallL@ucr.ac.cr}

\author[M. Viana]{Marcelo Viana}
\address{Instituto de Matematica Pura e Aplicada. Rio de Janeiro, Brazil.}
\email{viana@impa.br}

 
\begin{document}

\begin{abstract}
We investigate the regularity of Lyapunov exponents for products of random invertible matrices of arbitrary dimension $d\ge 2$. We prove that the top Lyapunov exponent is pointwise Hölder continuous with respect to the Wasserstein-Hausdorff distance in the space of probability measures with compact support, in the case where this exponent is simple and the equator has dimension equal to or less than one.
\end{abstract}

\maketitle
\tableofcontents

\input{body/1Introduction}

\section{Preliminaries}\label{s:preliminaries}

\input{body/2Preliminaries}

\section{Proof of Theorem A assuming Theorem B}\label{s:thmA}
\input{body/3ProvaThmA}

\section{Proof of Theorem B}\label{s:thmB}
\input{body/4ProvaThmB}

\bibliographystyle{alpha}
\bibliography{body/bibliography}

\end{document}

%% file: mycommands.tex
\usepackage{amsmath}
\usepackage{amsthm}
\usepackage{amsfonts}
\usepackage{amssymb}
\usepackage{amscd}
\usepackage{latexsym}
\usepackage[utf8]{inputenc}
\usepackage[dvipsnames]{xcolor}
\usepackage{hyperref}
\usepackage{caption}
\usepackage{subcaption}
\usepackage{float}
\usepackage{wrapfig}
\usepackage{amsrefs}
\usepackage{geometry}
\usepackage{marginnote}
\usepackage{dsfont}
\usepackage{mathtools}
    \mathtoolsset{showonlyrefs=true}
    \allowdisplaybreaks
    \newcommand{\proofpart}[1]{\vspace{1em}\noindent\textbf{#1.}}
\usepackage{tcolorbox}
\tcbuselibrary{breakable, skins}

\newtheorem{theorem}{Theorem}

\newtheorem{proposition}[theorem]{Proposition}
\newtheorem{lemma}[theorem]{Lemma}
\newtheorem{corollary}[theorem]{Corollary}
\newtheorem{definition}[theorem]{Definition}
\newtheorem{remark}[theorem]{Remark}
\newtheorem{example}[theorem]{Example}
\newtheorem{claim}[theorem]{Claim}
\numberwithin{theorem}{section}
\numberwithin{equation}{section}
\numberwithin{figure}{section}
\newtheorem{thmx}{Theorem}[section]

\def\Gr{{\operatorname{Gr}}}

\def\NN{{\mathbb N}}
\def\PP{{\mathbb P}}

\def\RR{{\mathbb R}}

\def\cA{{\mathcal A}}
\def\cB{{\mathcal B}}

\def\cE{{\mathcal E}}
\def\cF{{\mathcal F}}

\def\cL{{\mathcal L}}

\def\cN{{\mathcal N}}

\def\cP{{\mathcal P}}
\def\cQ{{\mathcal Q}}

\def\cT{{\mathcal T}}

\def\cX{{\mathcal X}}
\def\cY{{\mathcal Y}}

\def\fT{{\mathfrak{T}}}
\def\fg{{\mathfrak{g}}}

\def\ddx{{\operatorname{d}_X}}
\def\ddp{{\operatorname{d}_P}}
\def\dtheta{{\operatorname{d}_\theta}}

\def\GL{\operatorname{GL}}

\def\supp{\operatorname{supp}}

\def\Stat{\operatorname{Stat}}

\def\dim{\operatorname{dim}}

\def\Prob{\operatorname{Prob}}

\DeclareMathOperator{\VA}{VA}

\newcommand{\paren}[1]{\left( #1 \right)}
\newcommand{\norm}[1]{\left\Vert #1 \right\Vert}

\newcommand{\set}[1]{\left\{ #1 \right\}}
\newcommand{\dd}{\,\mathrm{d}} 

\usepackage{scalerel}[2016/12/29]

\usepackage{anysize}
\marginsize{2.5cm}{2.5cm}{1.3cm}{1.3cm}

\usepackage{bclogo}
\usepackage{amsrefs}

%% file: body/1Introduction.tex
\section{Introduction}\label{s:intro}

In this paper, we consider products of independent,
identically distributed (i.i.d.) matrices with common probability law
$\nu\in\operatorname{Prob}_c(G)$, where $G=\GL(d,\RR)$,
$d\geq 2$, and $\operatorname{Prob}_c(G)$ denotes the space
of compactly supported probability measures on $G$.
These products can be represented as a linear cocycle over
the Bernoulli shift on $(G^{\NN},\nu^{\NN})$,
with $A((g_j)_{j\in\NN})=g_0$.

The compact support assumption ensures the integrability
required by the Furstenberg-Kesten theorem \cite{FK60},
which guarantees that the limit
$$
\lambda_1(\nu)
=
\lim_{n\to\infty}\frac{1}{n}
\log\|g_{n-1}\cdots g_0\|
$$
exists for $\nu^{\NN}$-almost every sequence. By ergodicity of the Bernoulli shift, this limit is almost surely constant. Its value is called the \textit{top Lyapunov exponent of $\nu$}. Thus, the map $\nu\mapsto\lambda_1(\nu)$ defines a real-valued function on $\operatorname{Prob}_c(G)$.

Oseledets’ multiplicative ergodic theorem \cite{Oseledets} gives a more complete description of asymptotic growth, providing the full Lyapunov spectrum and an associated measurable filtration. A natural question is how these asymptotic growth rates, the Lyapunov exponents, vary under perturbations of the probability law.

For products of i.i.d. invertible matrices, Kifer \cite{Ki82b} established continuity of the top Lyapunov exponent at irreducible probability measures (that is, measures whose support has no common nonzero proper invariant subspace), under weak$^*$ convergence with suitable uniform integrability assumptions. Furstenberg and Kifer \cite{FK83} proved a  related continuity theorem under a weaker irreducibility assumption: the matrices in the support have at most one common nonzero proper invariant subspace. Their result holds in arbitrary dimension.

Still in the i.i.d. setting, Böcker and Viana \cite{BV2017} proved continuity of the Lyapunov exponents for compactly supported probability
measures in dimension two, without any irreducibility assumption. Avila, Eskin and Viana \cite{AEV2023} extended this result
to products of random matrices of arbitrary dimension.

Regarding the quantitative regularity of the top Lyapunov exponent in the i.i.d. case, we first highlight the work of Le Page \cite{LP89}. He established local Hölder continuity in arbitrary dimensions for probability measures satisfying strong irreducibility and a contraction property. Notably, these assumptions imply the simplicity of the top Lyapunov exponent, as shown by Guivarc’h and Raugi \cite{GR1986} and Gol’dsheid and Margulis \cite{GM1989}. This naturally leads to the question: is simplicity of the top Lyapunov exponent sufficient to guarantee Hölder continuity, at least in the pointwise sense?

In \cite{TV2020}, Tall and Viana answered this question affirmatively in dimension $d=2$. For compactly supported probability measures, they proved
pointwise Hölder continuity of the Lyapunov exponents
with respect to the Wasserstein-Hausdorff distance
at every measure with simple spectrum.
They also established pointwise log-Hölder continuity
without the simplicity assumption.

Motivated by Tall and Viana \cite{TV2020}, we extend their pointwise Hölder continuity result to higher-dimensional
systems with a simple top Lyapunov exponent, under
the additional assumption that the equator has
dimension at most one.

The results in \cite{BV2017}, \cite{AEV2023} and  \cite{TV2020} share a common topological setting: the space $\operatorname{Prob}_c(G)$ equipped with the topology $\fT$ combining weak convergence of measures and Hausdorff convergence of their supports. For their quantitative estimates, the authors in \cite{TV2020}  used the Wasserstein-Hausdorff distance $\delta_{\fT,\theta}$, which metrizes $\fT$. We use the same distance in our main result. See Section \ref{s:preliminaries} and \eqref{eq:definitionmetric} for the precise definitions.

For further continuity and regularity results in dimension two, including Markov cocycles and cocycles with invariant holonomies, see \cites{MV2015,BBB2018,DK2019b}. Results in arbitrary dimension and related settings can be found in \cites{Amorim2025,LV2026,CDKM2022,DG2025,DKP2022}. For a broader perspective on continuity and discontinuity of Lyapunov exponents, see Viana’s survey \cite{V2020Dis} and the references therein.

\subsection{Main results}
Before stating our main result, we introduce the concept of the equator. A subspace $\cL$ of $\RR^d$ is said to be $\nu$-invariant if $g\cL=\cL$ for all $g\in \supp\nu$. The $\nu$-equator, denoted by $E_\nu$, is the largest $\nu$-invariant subspace whose restricted top Lyapunov exponent is strictly smaller than $\lambda_1(\nu)$, see Definition \ref{definition.equator}. We set $E_\nu=\{0\}$ when no nonzero invariant subspace has this property. The main result of this paper is the following.

\begin{thmx}\label{theorem.HolderCont}
Let $\nu\in \operatorname{Prob}_c(G)$ be such that $\lambda_1(\nu)$ is simple and the $\nu$-equator $E_\nu$ satisfies $\dim E_\nu \in\{0,1\}$. There exists a constant $\theta_0=\theta_0(\nu)>0$ such that for every $\theta\in (0,\theta_0]$ there exists a neighborhood $\cN_0=\cN_0(\nu, \theta)$ of $\nu$ in $\operatorname{Prob}_c(G)$ and there are constants $C_0=C_0(\nu,\theta)>0$ and $\beta_0=\beta_0(\nu,\theta)\in (0,1]$ satisfying
$$
|\lambda_1(\nu)-\lambda_1(\nu')| \le C_0\, \delta_{\fT,\theta}(\nu,\nu')^{\beta_0}
$$
for every $\nu' \in \cN_0$.
\end{thmx}

Observe that the condition $\dim E_\nu\leq 1$ is
automatically satisfied when $d=2$.

The proof of Theorem \ref{theorem.HolderCont} treats separately the cases $\dim E_\nu=0$ and $\dim E_\nu=1$.

In the zero-dimensional case, we adapt an argument of Le Page \cite{LP89} to prove a contraction property of the dual Markov operator, which yields pointwise Hölder continuity of the top Lyapunov exponent. When the equator is trivial and the top Lyapunov exponent is simple, Durães \cite{Duraes2025} established local Hölder continuity using different techniques.

In the one-dimensional case, we follow the strategy of Tall and Viana \cite{TV2020}. To implement it, we need a higher-dimensional version of the stationary measure estimate in \cite{V2020Dis}*{Theorem 3.6}. We state this higher-dimensional estimate below, after introducing the necessary notation and terminology.

We now turn to our second result.
Let $P=\mathbb{P}\RR^d$.
For $r>0$, let $E_\nu(r)\subset P$ denote the closed
$r$-neighborhood of $E_\nu$. A probability measure on $G$ is called generic if it gives zero mass to every proper algebraic subvariety of $G$. Such a measure has a unique stationary probability measure on $P$. For more details, see Section \ref{s:preliminaries}.

\begin{thmx}\label{theorem.estimateadriyaya}
Let $\nu\in\Prob_c(G)$ and assume that $\dim E_\nu=1$.
There exists a neighborhood of $\nu$, $\cN_1= \cN_1(\nu) \subset \Prob_c(G)$,
and there are constants $\varepsilon_0=\varepsilon_0(\nu)>0$, $q=q(\nu)\in(0,1)$ and $s=s(\nu)>0$
such that for every $\varepsilon\in(0,\varepsilon_0]$ and every generic
$\nu'\in\cN_1$ satisfying $\supp\nu'\subset B_{s\varepsilon}(\supp\nu),$ the unique $\nu'$-stationary measure $\eta'$ satisfies
\begin{equation}\label{eq:balance}
    \eta'(E_\nu(q\varepsilon))
    \leq \frac23\,\eta'(E_\nu(\varepsilon)).
\end{equation}
\end{thmx}

Iterating this estimate bounds the mass of the stationary measure near the equator by a positive power of the radius, at scales for which the support condition remains valid.  Combined with estimates away from the equator and an approximation argument for general perturbations, it provides the stationary-measure bounds used to prove Theorem \ref{theorem.HolderCont}.

The proof of Theorem \ref{theorem.estimateadriyaya} is given in Section \ref{s:thmB}. The argument adapts the additive Margulis functions and recoupling technique developed by Avila, Eskin and Viana \cite{AEV2023} to obtain the required
quantitative estimate.

The restriction to a one-dimensional equator enters through Theorem \ref{theorem.estimateadriyaya}. An analogue of this estimate for higher-dimensional equators would yield pointwise Hölder continuity whenever the top Lyapunov exponent is simple. The main difficulty is that the higher-dimensional argument of \cite{AEV2023} proceeds through Grassmannians and flag spaces, whereas our proof requires a quantitative estimate for stationary measures on the projective space. It remains unclear how to derive the required projective estimate from the control obtained in those spaces.

\subsection*{Acknowledgements} This work corresponds to A.C.B.'s doctoral thesis, defended at IMPA on January 12, 2026. A.C.B. was supported by CAPES during her doctoral studies and has been supported by FUNCAP since February 2026.  A.S.C. and E.H.Y.T. were supported by Centro de Matemática Pura y Aplicada and Escuela de Matemática at University of Costa Rica. M.V. was partially supported by CNPq, FAPERJ and Institut de France/Prize Louis D. The first author would like to thank Professor Mauricio Poletti for many insightful discussions related to this work.

%% file: body/2Preliminaries.tex
In this section we introduce some key concepts and tools that are essential to this work. The results presented here are valid for any dimension $d\ge2$ and, unless stated otherwise, do not require the assumption of simplicity for the first exponent.

\subsection{Random matrix products}

Let $G = \mathrm{GL}(d, \mathbb{R})$ ($d \ge 2$) be endowed with the metric $\delta(g,g') = \|g-g'\| + \|g^{-1}-g'^{-1}\|$, where $\|\cdot\|$ is the operator norm. For $\theta \in (0,1]$, $\delta^\theta$ is also a metric on $G$. Let $\mathrm{Prob}_c(G)$ be the space of compactly supported probability measures on $G$ and take $\nu \in \mathrm{Prob}_c(G)$. We consider the sequence space $G^{\mathbb{N}}$ endowed with the product measure $\nu^{\mathbb{N}}$. 

We define the shift map $\sigma: G^{\mathbb{N}} \to G^{\mathbb{N}}$ by $\sigma((g_j)_{j \geq 0}) = (g_{j+1})_{j \geq 0}$. For a sequence $(g_j)_{j \geq 0} \in G^{\NN}$, we use, when convenient, the shorthand $\fg = (g_j)_{j \geq 0}$. Define also the projection onto the zeroth coordinate $A: G^\mathbb{N} \to G$ by $A(\fg) = g_0$, where $\fg = (g_j)_{j \geq 0}$. The linear cocycle $\mathcal{F}: G^{\mathbb{N}} \times \mathbb{R}^d \to G^{\mathbb{N}} \times \mathbb{R}^d$ generated by $A$ over $\sigma$ is defined by $$
\mathcal{F}(\fg, v) = (\sigma(\fg), A(\fg) v).$$ The $n$-th iterate satisfies $\mathcal{F}^n(\fg, v) = (\sigma^n(\fg), A^n(\fg) v)$, where $A^n(\fg) = A(\sigma^{n-1}(\fg))\ldots A(\sigma(\fg))A(\fg)$, which is exactly the product $g_{n-1} \cdots g_1 g_0$.

In this setting, the Furstenberg-Kesten Theorem (\cite{LLE}*{Theorem 3.12}) guarantees the existence, $\nu^\NN$-almost everywhere, of the limits
$$
    \lambda_+(\fg) = \lim_{n\to\infty} \frac{1}{n} \log\|A^n(\fg)\| \quad \text{and} \quad \lambda_-(\fg) = \lim_{n\to\infty} \frac{1}{n} \log\|(A^n(\fg))^{-1}\|^{-1},
$$ which are called the \textit{extremal Lyapunov exponents}. By ergodicity of the base dynamics $(\sigma, \nu^\NN)$, these limits are $\nu^{\mathbb{N}}$-almost everywhere constant. In this work, we are particularly interested in the top Lyapunov exponent $\lambda_+(\fg)$, which we denote from now on by $\lambda_1(\nu)$.

According to the Oseledets Multiplicative Ergodic Theorem (see, e.g., \cite{LLE}*{Theorem 4.1}), the compact support of $\nu$ ensures that for $\nu^\NN$-almost every sequence $\fg \in G^\NN$, the space $\RR^d$ admits a measurable and $\mathcal{F}$-invariant filtration called the Oseledets flag $\RR^d = V_\fg^1 \supset V_\fg^2 \supset \cdots \supset V_\fg^k \supset \{0\}$,
associated with the Lyapunov exponents $\lambda_1(\nu) > \lambda_2(\nu) > \dots > \lambda_k(\nu)$. These subspaces satisfy the condition
$$
\lim_{n \to \infty} \frac{1}{n} \log \|A^n(\fg)v\| = \lambda_i(\nu)
$$
for every vector $v \in V_\fg^i \setminus V_\fg^{i+1}$ and each $1 \le i \le k$ (with $V_\fg^{k+1} = \{0\}$). In this context, the multiplicity of each exponent $\lambda_i(\nu)$ is defined as the drop in dimension between consecutive elements of the filtration, namely $\dim V_\fg^i - \dim V_\fg^{i+1}$. To say that the top Lyapunov exponent $\lambda_1(\nu)$ is \textit{simple} means that its multiplicity is exactly one.

We also write the Lyapunov exponents counted with
multiplicity as
$$
    \chi_1(\nu)\geq\chi_2(\nu)\geq\cdots\geq\chi_d(\nu),
$$
where each distinct exponent $\lambda_i(\nu)$ is
repeated according to its multiplicity. In particular,
if $\lambda_1(\nu)$ is simple, then
$\chi_1(\nu)=\lambda_1(\nu)$ and
$\chi_2(\nu)=\lambda_2(\nu)$.

\begin{remark}[Second exterior power] 
\label{rem:secondexteriorpower}
We will also consider the action of random matrix
products on the second exterior power
$\bigwedge^2\RR^d$. Every $g\in G$ induces an invertible
linear map $\wedge^2g$, characterized by
$$
    (\wedge^2g)(u\wedge v)=gu\wedge gv.
$$
Since
$$
    \wedge^2A^n(\fg)
    =(\wedge^2g_{n-1})\cdots(\wedge^2g_0),
$$
this is again a random matrix product, with distribution
$$
    \tilde\nu=(\pi_{\wedge 2})_*\nu,
$$ with $\pi_{\wedge 2}(g)=\wedge^2g$. The continuity of $\pi_{\wedge 2}$ ensures that
$\tilde\nu$ has compact support. Its Lyapunov exponents,
counted with multiplicity, are the sums
$\chi_i(\nu)+\chi_j(\nu)$, with $1\leq i<j\leq d$
(see \cite{LLE}*{Proposition 4.17}). In particular,
$$
    \lambda_1(\tilde\nu)=\chi_1(\nu)+\chi_2(\nu).
$$
\end{remark}

 We equip $\mathrm{Prob}_c(G)$ with the topology $\mathfrak{T}$, defined as the smallest topology such that
\begin{enumerate}
    \item [(1)] The map $\nu\mapsto \int_G\psi d\nu$ is continuous for every continuous function $\psi:G\rightarrow \RR$, and
    \item [(2)] The map $\nu\mapsto \supp\nu$ is continuous relative to the Hausdorff topology in the space of compact subsets of $G$.
\end{enumerate} 

 This topology is metrized by the Wasserstein-Hausdorff distance
\begin{equation}\label{eq:definitionmetric}
\delta_{\mathfrak{T},\theta}(\nu,\nu') = \delta_\theta(\nu,\nu') + \delta_H(\mathrm{supp}\,\nu, \mathrm{supp}\,\nu'),
\end{equation}
where $\delta_\theta(\nu,\nu') = \sup_{\|\psi\|_{\mathrm{Lip}} \le 1} \left| \int_G \psi \, d(\nu-\nu') \right|$ is the $\theta$-Wasserstein distance and the Lipschitz constant $||\cdot||_{\mathrm{Lip}}$ is computed with respect to $\delta^\theta$. Moreover, $\delta_H$ is the Hausdorff metric defined for compact sets $K, K' \subset G$ by
$$
\delta_H(K,K') = \inf\{r > 0 : K \subset K'(r) \text{ and } K' \subset K(r)\},
$$
with $K(r) := \{g \in G : \delta(g,K) < r\}$.

\subsection{Markov operators}
\label{Markovoperatorsingeneral}

Let $(X, \ddx)$ be a metric space. Consider the Markov operator on the space of bounded Borel functions $\mathcal{T}:\mathfrak{B}(X)\rightarrow \mathfrak{B}(X)$ defined by
$$
    \mathcal{T}\varphi(x):=\int_X \varphi(y) d\sigma_x(y),
$$where $\{\sigma_x : x\in X\}$ is a family of probability measures on $X$. We say that $\mathcal{T}$ is \textit{continuous} when the map $x\mapsto \sigma_x$ is continuous with respect to the weak*-topology. Notice that $\cT$ is a bounded operator with $\|\cT\|=1$. The dual map $\mathcal{T}^*$ of $\mathcal{T}$ is defined on the set of bounded finitely additive signed measures $\eta$ on $X$ with the total variation norm (see \cite{DSI}*{IV.4.5}) by
$$
    \mathcal{T}^*\eta  := \int_X \sigma_y \,d\eta(y).
$$ We say that $\eta$ is $\cT$\textit{-invariant} if $\cT^*\eta=\eta$. Note that $\cT^*$ leaves the set of probability measures on $X$ invariant. These two operators are related by
$$
\int_X \mathcal{T}\varphi (y) d\eta(y) = \int_X \varphi(y) d(\mathcal{T}^*\eta)(y).
$$

Let $X$ and $X'$ be measurable spaces. Consider $\eta$ a measure on $X$ and $\eta'$ a measure on $X'$ with $\eta(X)=\eta'(X')$. A \textit{coupling} of $\eta$ and $\eta'$ is a measure $\tilde{\eta}$ on $X\times X'$ such that $\tilde{\eta}$ projects to $\eta$ on the first coordinate and to $\eta'$ on the second. In other words, if we define the coordinate maps
\begin{align*}
    \overline{\pi}_1: X\times X' &\rightarrow X  &\overline{\pi}_2: X\times X' &\rightarrow X' \\
    (x,x') & \mapsto x     &(x,x') & \mapsto x'
\end{align*}
then, we have that $\eta = {\overline{\pi}_1}_* \tilde{\eta}$ and $\eta' = {\overline{\pi}_2}_* \tilde{\eta}$. When $\eta=\eta'$ we say that $\tilde{\eta}$ is a \textit{self-coupling} of $\eta$. A self-coupling $\tilde{\eta}$ is \textit{symmetric} if it is invariant under the involution map $\iota:(x,x')\mapsto (x',x)$.

Now, consider the product space $X\times X$ with the product sigma-algebra. Let $\mathfrak{B}(X\times X)$ be the space of bounded functions on $X\times X$. A \textit{self-coupling} $\Tilde{\cT}$ of $\cT$ is an operator $\Tilde{\cT}:\mathfrak{B}(X\times X)\rightarrow \mathfrak{B}(X\times X)$ defined as
$$
    \Tilde{\cT}\psi(x,x'):=\int_{X\times X} \psi(y,y') d\sigma_{x,x'}(y,y'),
$$ where $\sigma_{x,x'}$ is a coupling of $\sigma_x$ and $\sigma_{x'}$. The self-coupling is \textit{continuous} if the map $(x,x')\mapsto \sigma_{x,x'}$ is continuous on $X\times X$. We can also define the dual operator $\Tilde{\cT}^*$ by
$$
    \Tilde{\mathcal{T}}^*\Tilde{\eta}  := \int_{X\times X} \sigma_{y,y'}\, d\Tilde{\eta}(y,y').
$$ We say that $\Tilde{\eta}$ is $\Tilde{\cT}$\textit{-invariant} if $\Tilde{\cT}^*\Tilde{\eta}=\Tilde{\eta}$.


Let $P=\mathbb{P}\RR^d=\Gr(1,d)$ be the space of all lines through the origin in $\RR^d$, endowed with the distance $\ddp(u,u') = |\sin\angle(u,u')|$. To simplify notation, we use the same symbol, for example $u$, to refer both to a nonzero vector in $\RR^d$ and to its corresponding point in the projective space $P$. Throughout this text, we will encounter projective spaces of the form $\mathbb{P}\mathcal{L}$, where $\mathcal{L}$ is a subspace of $\mathbb{R}^d$. We will also denote the distance on $\mathbb{P}\mathcal{L}$ by $\ddp$, and its definition remains the same.

The group $G$ acts on the projective space $P$ via the $G$-action $(g,u)\mapsto gu$, which is \textit{transitive}, meaning that for any fixed $u\in P$, the map $g\mapsto gu$ is surjective. This action describes the dynamics on the fibers of the projective cocycle $\PP\cF$, which is defined as follows
\begin{alignat}{2}
\label{eq.projectivecocycledefinition}
    \PP\cF: G^\NN\times P&\rightarrow G^\NN\times P \nonumber\\
    (\fg,u)&\mapsto (\sigma (\fg), A(\fg) u),
\end{alignat} where $\sigma$ is the shift-map and $A$ is the map $A:G^\NN\to G, (g_k)_{k\ge0}\mapsto g_0$, as before.

Consider the space $\operatorname{Prob}(P)$ of all probability measures on $P$ endowed with the Wasserstein metric. The $\theta$-\textit{Wasserstein distance} $\dtheta$ of two probability measures $\eta,\eta'\in \operatorname{Prob}(P)$ is defined as
\begin{equation*}
\dtheta(\eta,\eta') = \sup_{\|\phi\|_{\mathrm{Lip}} \le 1} \left| \int_P \phi \, d(\eta-\eta') \right|= \inf\left\{\int_{P\times P} \ddp(u,u')^\theta d\xi(u,u'): \xi \text{ is a coupling of } \eta \text{ and } \eta'\right\},
\end{equation*} where the Lipschitz constant $||\cdot||_{\mathrm{Lip}}$ is computed with respect to $\ddp^\theta$. The weak*-topology on $\operatorname{Prob}(P)$ is metrized by $\dtheta$, for every $\theta\in(0,1]$.

 Consider the family $\{\sigma_u:u\in P\}$ of probability measures on $P$, where $\sigma_u$ is the push-forward of $\nu$ under the map $g\mapsto gu$. These probabilities define a Markov operator, denoted by $\cP_\nu:\mathfrak{B}(P)\rightarrow\mathfrak{B}(P)$ and given by
\begin{equation}\label{eq.definitionMarkovoperator}
    \cP_\nu\varphi(u) := \int_P \varphi(v)\; d\sigma_u(v) = \int_G \varphi(gu) \;d\nu(g).
\end{equation}
We can also define the dual operator $\cP_\nu^*$, given by
\begin{equation}
    \label{eq.defmarkovdual}
    \cP_\nu^*\eta := \int_P \sigma_u\;d\eta(u) = \int_G g_*\eta\; d\nu(g),\qquad \eta\in \operatorname{Prob}(P).
\end{equation} In this definition, $g_*\eta$ denotes the push-forward of the measure $\eta$ under the map $u\mapsto gu$. Note that $\cP_\nu$ is continuous if the $G$-action is continuous.

Denote the product measure on $G^n$ by $\nu^n$, $n\ge1$.

\begin{definition}\label{definition.convolution}
The \emph{$n$-convolution} $\nu^{(n)}\in \operatorname{Prob}_c(G)$ of
a measure $\nu\in \operatorname{Prob}_c(G)$ is defined by
$$
\int_G \varphi(g) \, d\nu^{(n)}(g)
= \int_{G^n} \varphi(g_{n-1} \cdots g_0) \, d\nu^n(g_0,\ldots,g_{n-1}),
$$ for every $\varphi:G\rightarrow \RR$ bounded. Equivalently, $\nu^{(n)}$ is the push-forward of $\nu^n$ under the map $\paren{g_0,\ldots,g_{n-1}}\mapsto g_{n-1}\ldots g_0$. The map $\nu\mapsto\nu^{(n)}$ is continuous relative to the weak*-topology. 
\end{definition}

By a standard induction argument, the $n$-th iterate of the dual Markov operator $\cP_\nu^{*n}$ satisfies, for any $\eta\in \text{Prob}(P)$,
$$
    \cP_\nu^{*n} \eta  = \int_G  g_*\eta\, d\nu^{(n)}(g) = \cP_{\nu^{(n)}}^{*} \eta.
$$Moreover, since the base dynamics is i.i.d., the product measure $\nu^n$ is invariant under reversing the order of the coordinates. Hence, for every $\varphi\in \mathfrak{B}(P)$ and $u\in P$,
$$
    \cP_\nu^n\varphi(u)=\int_G\varphi(gu)\;d\nu^{(n)}(g) = \cP_{\nu^{(n)}}\varphi(u).
$$

We can define a $G$-action on the product space $P\times P$ by  $(g,u,u')\mapsto (gu,gu')$. Let $\{\sigma_{u,u'}: (u,u')\in P\times P\}$ be a family of probability measures on $P\times P$ such that $\sigma_{u,u'}$ is the push-forward of $\nu$ under the map $g\mapsto (gu,gu')$. This family defines a Markov operator $\tilde{\cP}_{\nu}:\mathfrak{B}(P\times P)\rightarrow \mathfrak{B}(P\times P)$ in the product space, as follows
\begin{equation}
    \tilde{\cP}_{\nu} \varphi(u,u') := \int_{P\times P} \varphi(v,v')\,d\sigma_{u,u'}(v,v') = \int_G \varphi(gu,gu')\,d\nu(g).
    \label{eq.markovoperatorproductspacestationary}    
\end{equation} The dual operator $\tilde{\cP}_{\nu}^*$ is defined as
$$
    \tilde{\cP_\nu^*}\tilde{\eta} := \int_{P\times P} \sigma_{u,u'}\, d\tilde{\eta}(u,u') = \int_G g_*\tilde{\eta}\,d\nu(g),\quad\tilde{\eta}\in\operatorname{Prob}(P\times P).
$$

The iterates of these operators satisfy $\tilde{\cP}_{\nu}^{n}=\tilde{\cP}_{\nu^{(n)}}$ and $\tilde{\cP}_{\nu}^{*n}=\tilde{\cP}_{\nu^{(n)}}^*$. Hence, for every bounded function $\varphi \in \mathfrak{B}(P\times P)$, any probability measure $\tilde{\eta}\in \text{Prob}(P\times P)$ and every $n \in \mathbb{N}$, we have:
\begin{equation}\label{eq:ncouplingexpression}
    \int_{P\times P} \varphi(u,u')\, d(\tilde{\cP}_\nu^{*n}\tilde{\eta})(u,u') = \int_{G\times P\times P} \varphi(gu,gu')\, d\nu^{(n)}(g)d\tilde{\eta}(u,u').
\end{equation} We say that $\tilde\eta$ is $\tilde\cP_\nu$\textit{-invariant} if $\tilde\cP_\nu^*\tilde\eta=\tilde\eta$.

\begin{lemma}
\label{lemma.selfcouplingofthemarkovoperator}
   The operator $\tilde{\cP}_\nu^{*n}$ is a self-coupling of $\cP_\nu^{*n}$.
\end{lemma}

\begin{proof}
We know that $\cP^{*n}_\nu=\cP^*_{\nu^{(n)}}$, which is the Markov operator associated with the family of probability measures $\{\sigma_{u,(n)}: u \in P\}$, where $\sigma_{u,(n)}=$ push-forward of $\nu^{(n)}$ under the map $g\mapsto gu$. Analogously, we have $\tilde\cP^{*n}_\nu=\tilde\cP^*_{\nu^{(n)}}$, which is associated with the family $\{\sigma_{u,u',(n)}:(u,u')\in P\times P\}$, where $\sigma_{u, u',(n)}=$ push-forward of $ \nu^{(n)}$ under the map $g\mapsto (gu,gu')$. Let $\pi_1$ and $\pi_2$ be the coordinate maps
    \begin{align} \label{eq.definitionofprojectionspi}
         \pi_1: P\times P &\rightarrow P  &\pi_2: P\times P &\rightarrow P \\
        (u,u') & \mapsto u     &(u,u') & \mapsto u'.\nonumber
    \end{align}  

     We want to prove that $\sigma_{u, u',(n)}$ is a coupling of $\sigma_{u,(n)}$ and $\sigma_{u',(n)}$ for every $u,u'\in P$ and every $n\ge 1$. In other words, we want to prove that, for every $u,u'\in P$ and for all $n\ge1$, we have
    $$
            {\pi_1}_* \sigma_{u, u',(n)} = \sigma_{u, (n)}\quad\textnormal{and}\quad {\pi_2}_* \sigma_{u, u',(n)} = \sigma_{u',(n)}.
    $$

 Given $u, u'\in P$, we introduce the maps $f_u:G \rightarrow P$ such that $f_u(g)=gu$ and $f_{u,u'}:G \rightarrow P\times P$ such that $f_{u,u'}(g)=(gu,gu')$. Then, by definition, we have $\sigma_{u,(n)}:=(f_u)_*\nu^{(n)}$ and $\sigma_{u, u',(n)} := (f_{u,u'})_*\nu^{(n)}$. Let $\cE$ be a measurable subset of $P$. We have that
    \begin{equation*}
        {\pi_1}_*\sigma_{u, u',(n)}(\cE)
        = \int_G \mathds{1}_{\cE\times P}(f_{u,u'}(g))\; d\nu^{(n)}(g) = \int_G \mathds{1}_\cE(f_u(g))\;d\nu^{(n)}(g)= \sigma_{u,(n)}(\cE).
    \end{equation*} Thus, the equality ${\pi_1}_* \sigma_{u, u',(n)} = \sigma_{u,(n)}$ holds for every $(u,u')\in P\times P$ and for all $n\ge 1$. A similar argument guarantees that the second equality ${\pi_2}_* \sigma_{u, u',(n)} = \sigma_{u',(n)}$ is also true for every $(u,u')\in P\times P$ and for all $n\ge 1$. This concludes the proof.

\end{proof}

\begin{lemma}
\label{lemma:ndiagonalpushforward}
    If $\tilde{\eta}$ is a coupling of $\eta$ and $\eta'$, then $\tilde{\cP}_\nu^{*n}\tilde{\eta}$ is a coupling of $\cP^{*n}_\nu\eta$ and $\cP^{*n}_\nu\eta'$, for every $n\in\NN, n\ge1$.
\end{lemma}

\begin{proof}
    This is a direct consequence of  \cite{AEV2023}*{Lemma 6.18} and Lemma \ref{lemma.selfcouplingofthemarkovoperator}.
\end{proof}

\begin{definition}
\label{definition.stationarymeasure}
We say that $\eta \in \operatorname{Prob}(P)$ is $\nu$\textit{-stationary} if it is a fixed point of the dual Markov operator $\cP_\nu^*$, \textit{i.e.}, $\mathcal{P}_\nu^*\eta=\eta$.     
\end{definition}

The set of all stationary measures of $\nu$ is denoted by $\operatorname{Stat}(\nu)$. Notice that a $\nu$-stationary measure $\eta$ is a special case of an invariant measure for a Markov operator. In fact, $\eta$ is called $\nu$-stationary when it is invariant under the particular Markov operator $\cP_\nu$.

\begin{remark}\label{remark.aboutstationarymeasures}
    A probability measure $\eta$ is $\nu$-stationary if and only if the product measure $\nu^{\mathbb{N}}\times\eta$ is invariant under the projective cocycle $\mathbb{P}\mathcal{F}$ (see \cite{LLE}*{Proposition 5.5}). The compactness of $P$ ensures the existence of stationary measures (see \cite{LLE}*{Proposition 5.6}). Moreover, a $\nu$-stationary measure is \textit{ergodic} if every function $\varphi:P\rightarrow \RR$ with $\cP\varphi=\varphi$ is constant on a full $\eta$-measure set. Equivalently, a $\nu$-stationary measure $\eta$ is ergodic if and only if the system $(\PP\cF, \nu^\NN\times \eta)$ is ergodic (see \cite{LLE}*{Proposition 5.13}). 
\end{remark}

Consider the map
\begin{align} \label{def.Phi}
    \Phi: G^\NN\times P &\rightarrow \RR \\
    (\fg,u)&\mapsto \log\frac{\|g_0u\|}{\|u\|},
\end{align} where $\fg=(g_k)_{k\ge0}$. Since this function only depends on the first term of the sequence, we may identify $\Phi$ with the corresponding map on $G\times P \to \RR$.

By \cite{LLE}*{Proposition 6.7}), we have
\begin{equation}
\label{eq.lambdaonemaximumintegral}
     \lambda_1(\nu)
    = \max\left\{ \int_{G^\mathbb{N} \times P} \Phi \, d(\nu^\mathbb{N} \times \eta) :\; \eta \ \nu\text{-stationary} \right\}.
\end{equation}
Moreover, there exists a $\nu$-stationary measure $\eta$ attaining this maximum.  
In this case, we say that $\eta$ is \emph{maximal}, and we get
\begin{equation}
    \label{eq.lambadaoneintegralformula}
    \lambda_1(\nu)
    = \int_{G^\mathbb{N} \times P} \Phi \, d(\nu^\mathbb{N} \times \eta)
    = \int_{G \times P} \log\frac{\|g u\|}{\|u\|} \, d(\nu \times \eta).
\end{equation} 

Maximal stationary measures are essential to the development of this work exactly because they provide the integral formula \eqref{eq.lambadaoneintegralformula} for the system's top Lyapunov exponent. This representation is the key to obtaining an estimate for the difference between $\lambda_1(\nu)$ and $\lambda_1(\nu')$, with $\nu'$ near $\nu$.

\subsection{Generic measures}
\label{subsection:generic-measures}

We consider the real algebraic varieties
$P=\mathbb{P}\mathbb{R}^{d}$ and
$G=\operatorname{GL}(d,\mathbb{R})$.
Throughout this subsection, measures are finite positive
Borel measures. Let $X$ be one of the algebraic varieties above.

\begin{definition}
\label{def:zariski-dense}
A subset $S\subset X$ is \emph{Zariski-dense} if it is
not contained in any proper algebraic subvariety of $X$.
\end{definition}

\begin{definition}
\label{def:generic-measure}
A measure $\eta$ on $X$ is \emph{generic} if
$\eta(S)=0$ for every proper algebraic
subvariety $S\subset X$.
\end{definition}

In particular, every generic measure on $X=P$ or $X=G$ is non-atomic,
since each point is a proper algebraic subvariety of $X$. We apply this terminology both to probability measures
$\nu$ on $G$ and to stationary probability measures
$\eta$ on $P$. 

\begin{remark}[Properties of generic measures]
\label{rem:generic-measures}
\phantom{lalala}
\begin{enumerate}
    \item[(i)]
    Every generic probability measure on $X$ has
    Zariski-dense support.

    \item[(ii)]
    If $\nu\in\operatorname{Prob}_c(G)$ has Zariski-dense
    support, then it admits a unique stationary probability
    measure on $P$, and this measure is generic
    (see \cite{AEV2023}*{Proposition 5.2} and
    \cite{GM1989}*{Lemmas 4.2 and 4.5}).
    In particular, this holds whenever $\nu$ is generic.
    
    \item[(iii)]
    Generic probability measures are dense in
    $\operatorname{Prob}_c(G)$ with respect to the topology
    $\mathfrak T$: every $\nu\in\operatorname{Prob}_c(G)$
    is the limit, in this topology, of a sequence of generic
    probability measures
    (see \cite{AEV2023}*{Proposition 5.3}).   
\end{enumerate}
\end{remark}

\subsection{The \texorpdfstring{$\nu$}{nu}-equator}

Here, we introduce an invariant subspace of $\RR^d$, called the $\nu$\textit{-equator}, which will play a central role in this work.

\begin{definition}
    Let $\mathcal{L}$ be a linear subspace of $\RR^d$. We say that $\mathcal{L}$ is \textit{$\nu$-invariant} if $g\mathcal{L}=\mathcal{L}$ for $\nu$-almost every $g\in G$, or equivalently, for every $g\in\supp\nu$.
\end{definition}

 Notice that $\mathcal{L}=\RR^d$ is trivially $\nu$-invariant. Here, we use the same notation $\mathcal{L}$ for a linear subspace of $\RR^d$ and for its corresponding subset of $P$.

\begin{definition} \label{definition.equator}
    The $\nu$\textit{-equator} $E_\nu$ is the maximal 
    $\nu$-invariant subspace of $\RR^d$ such that
    \begin{equation}
    \label{definitionofequator}
        \lim_{n\rightarrow \infty} \frac{1}{n} \log\|g_{n-1}\ldots g_0 |_{E_\nu}\|< \lambda_1(\nu)
    \end{equation} for $\nu^\NN$-almost every $\fg=(g_k)_{k\ge 0} \in G^\NN$.
\end{definition} The motivation for this definition comes from \cite{FK83}*{Theorem 3.9}, which establishes a deterministic filtration of subspaces. Furthermore, when the measure $\nu$ is clear from the context, we will refer to $E_\nu$ simply as \emph{the equator}.

\begin{remark}\label{remark:equatorproperties}
    \begin{enumerate}

        \item The equator is unique. Indeed, this follows from the maximality in the definition and from the fact that, if $\mathcal{L}$ and $\mathcal{L'}$ are $\nu$-invariant subspaces satisfying \eqref{definitionofequator}, then their sum $\mathcal{L} + \mathcal{L'}$ also satisfies it. 
    
        \item Although $E_\nu$ is a subspace of $V^2_{\fg}$ for $\nu^{\mathbb{N}}$-almost every $\fg \in G^{\mathbb{N}}$, where $V^2_{\fg}$ denotes the second subspace in the Oseledets flag, it may not coincide with any of the Oseledets subspaces (see \cite{AEV2023}, Example 3.2).
 
    \end{enumerate}
\end{remark}

Another key result of Furstenberg and Kifer \cite{FK83}*{Theorem 3.5} establishes the following dichotomy. Let $u\in \RR^d\setminus\{0\}$, for $\nu^\NN$-almost every $\fg\in G^\NN$,
\begin{equation}
\label{eq:dichotomy}
     \lim\limits_{n\rightarrow \infty}\frac 1 n \log\|A^n(\fg)u\|\left\{\begin{matrix}
        =\lambda_1(\nu),\; \textnormal{if}\; u\notin E_\nu;
        \\
        <\lambda_1(\nu),\; \textnormal{if}\; u\in E_\nu.
    \end{matrix}\right.
\end{equation} Hence, when the $\nu$-equator has dimension zero, the limit above equals the first Lyapunov exponent $\lambda_1(\nu)$ for every $u\in \RR^d$ and $\nu^\NN$-almost every $\fg$.

The following proposition uses the $\nu$-equator to characterize maximal $\nu$-stationary measures: the maximal ones are exactly those that assign weight zero to $E_\nu$.

\begin{proposition}
\label{prop.groundediffmaximal}
    Let $\nu\in \operatorname{Prob}_c(G)$ and $\eta\in \operatorname{Stat}(\nu)$. Then, $\eta (E_\nu)=0$ if and only if $\eta$ is maximal.
\end{proposition}

\begin{proof}
Let $\nu\in \operatorname{Prob}_c(G)$ and $\eta\in \operatorname{Stat}(\nu)$. We have
    \begin{align*}
        \int_{G\times P} \log{\frac{\|gu\|}{\|u\|}} d\nu(g)d\eta(u) &= \int_{G^\NN\times P} \log{\frac{\|g_0 u\|}{\|u\|}} \; d(\nu^\NN \times\eta)(\fg,u) = \int_{G^\NN\times P} \Phi(\fg,u) \; d(\nu^\NN \times\eta)(\fg,u).
     \end{align*} The map $\Phi$, defined in \eqref{def.Phi}, is integrable with respect to $\nu^\NN\times\eta$, and this product measure is invariant under the cocycle $\PP\cF$. Therefore, by the Birkhoff Ergodic Theorem, we have
     \begin{align}
        \int_{G\times P} \log{\frac{\|gu\|}{\|u\|}} d\nu(g)d\eta(u) &= \int_{G^\NN\times P}\; \lim_{n\rightarrow \infty}\frac{1}{n}\sum_{i=0}^{n-1}\Phi(\cF^i(\fg,u)) \; d(\nu^\NN \times\eta)(\fg,u)\\[2pt]
        &=
    \int_{G^\NN\times P}
    \lim_{n\to\infty}\frac1n
    \log\frac{\|A^n(\fg)u\|}{\|u\|}
    \,d(\nu^\NN\times\eta)(\fg,u)\\[2pt]
        \label{eq:integralE+PminusE} &= \int_{G^\NN\times (P\setminus E_\nu)} \;\lim_{n\to\infty}\frac1n
    \log\frac{\|A^n(\fg)u\|}{\|u\|}\; d(\nu^\NN \times\eta)(\fg,u)\\ 
        & \quad\qquad + \int_{G^\NN\times E_\nu}\; \lim_{n\to\infty}\frac1n
    \log\frac{\|A^n(\fg)u\|}{\|u\|}\; d(\nu^\NN \times\eta)(\fg,u)
    \end{align} By \eqref{eq:dichotomy}, given $u\in \RR^d\setminus\{0\}$ a representative of $u\in P$, for
$\nu^\NN$-almost every $\fg$ we have 
$$
    \lim_{n\to\infty}\frac1n
    \log\frac{\|A^n(\fg)u\|}{\|u\|}
    \begin{cases}
        =\lambda_1(\nu), & \text{if }u\notin E_\nu,\\
        <\lambda_1(\nu), & \text{if }u\in E_\nu.
    \end{cases}
$$
Now assume that $\eta(E_\nu)=0$.
Then $\eta(P\setminus E_\nu)=1$. Thus,
    \begin{equation*}
        \int_{G\times P} \log{\frac{\|gu\|}{\|u\|}} d\nu(g)d\eta(u) = \int_{P\setminus E_\nu} \lambda_1(\nu) d\eta(u)= \lambda_1(\nu) \eta(P\setminus E_\nu)= \lambda_1(\nu).
    \end{equation*} This proves the first implication. In order to prove the reverse implication we will argue by contraposition: if $\eta(E_\nu)\ne 0$, then $\eta$ is not maximal.

    Consider a $\nu$-stationary measure $\eta$ such that $\eta(E_\nu)=m>0$. Hence, $\eta(P\setminus E_\nu)=1-m$. By the preceding almost-everywhere inequalities and
\eqref{eq:integralE+PminusE}, we get
    \begin{equation*}
    \int\limits_{G\times P} \log\frac{\|gu\|}{\|u\|} d\nu(g)d\eta(u) < \lambda_1(\nu)(1-m) + \lambda_1(\nu)m = \lambda_1(\nu).
    \end{equation*} Hence, $\eta$ is not maximal.
\end{proof}

\begin{remark}[Uniqueness of the maximal stationary measure] \label{remark.uniquenessofmaximalstationary}
    Aoun and Guivarc'h \cite{aoun2020random}*{Theorem 2.4 (a)} proved that when $\lambda_1(\nu)$ is simple, there exists a unique $\nu$-stationary measure $\eta$ on the projective space $P$ satisfying $\eta(E_\nu)=0$. Since Proposition \ref{prop.groundediffmaximal} identifies this very condition as the characterization of maximality, it follows that the simplicity of $\lambda_1(\nu)$ guarantees the existence of a unique maximal $\nu$-stationary measure, which we will denote by $\eta_\nu$.
\end{remark}

For each $g\in \supp\nu$, the restriction of $g$ to $E_\nu$ defines a linear transformation $g^E \in GL(E_\nu)$ . Let $E_\nu^\perp$ be the orthogonal complement of $E_\nu$. Since $\RR^d=E_\nu\oplus E_\nu^\perp$, any vector $u\in \RR^d$ can be uniquely decomposed as a sum $u=u_E+u_\perp$, where $u_E\in E$ and $u_\perp\in E^\perp$. With respect to this decomposition, each $g\in \supp\nu$ has a block upper-triangular representation:
\begin{equation}\label{eq.representationofg}
    g = \begin{pmatrix} g^E & h(g)\\0&g^\perp\end{pmatrix},    
\end{equation} where $g^\perp\in GL(E_\nu^\perp)$ is the induced linear operator on $E_\nu^\perp$, defined for $v^\perp\in E_\nu^\perp$ by $g^{\perp}(v^\perp)=\operatorname{Proj}_{E_\nu^\perp}(gv^\perp)$. This construction yields natural projection maps
\begin{align} \label{eq.definitionpiEandpiEperp}
    \pi_E:\supp\nu &\rightarrow GL(E_\nu)\qquad\qquad\text{and}\qquad \!\!\!\!\!\!\!\!\!\!\!\!\!\!\!\!\!\!\!\!\!\!&
    \pi_\perp:\supp\nu &\rightarrow GL(E_\nu^\perp)\\
    g&\mapsto g^E &g&\mapsto g^\perp
\end{align} and corresponding push-forward measures $\nu^E\coloneqq (\pi_E)_*\nu$ and $\nu^\perp\coloneqq(\pi_\perp)_*\nu$, whose supports are given by the images of $\supp\nu$:
\begin{equation}\label{eq.relationbetweensupportsequatorandperp}
    \supp\nu^E = \pi_E(\supp\nu) \quad\text{and}\quad \supp\nu^\perp= \pi_\perp(\supp\nu).
\end{equation} Thus, for any $g \in \supp\nu$, decomposed as in \eqref{eq.representationofg}, its components must lie in the corresponding supports: $g^E \in \supp\nu^E$ and $g^\perp \in \supp\nu^\perp$. Conversely, given $g^\perp\in \supp\nu^\perp$, there exists $g\in\supp\nu$ such that $g^\perp=\pi_\perp(g)$, and a similar statement holds for $\supp\nu^E$. Moreover, products of elements in $\supp\nu$ preserve the block structure shown in \eqref{eq.representationofg}. Indeed, for every $n\in \NN$, given $g_0,\ldots, g_{n-1}\in \supp\nu$, we have
\begin{equation}\label{eq.representationEperpproduct}
    g_{n-1}\ldots g_0 = \begin{pmatrix}
            g_{n-1}^E\ldots g_0^E & &h(g_{n-1}\cdots g_1g_0)\\
            0&& g_{n-1}^\perp\ldots g_0^\perp
        \end{pmatrix},
\end{equation} where 
\begin{align} \label{eq.definitionofhn}
h(g_{n-1}\ldots g_1g_0) =& \;\;g^E_{n-1}\ldots g^E_{1} h(g_0) + g^E_{n-1}\ldots g^E_{2} h(g_1)g^\perp_{0} +\\ &\;\;\cdots + g^E_{n-1}h(g_{n-2})g^\perp_{n-3}\ldots g^\perp_1g^\perp_0 + h(g_{n-1})g^\perp_{n-2}\ldots g^\perp_0.    
\end{align}

By Definition \ref{definition.convolution}, the support of the $n$-convolution $\nu^{(n)}$ is given by
$$
        \supp\nu^{(n)}=\{g_{n-1}\ldots g_0: g_i\in \supp\nu, \forall i\in\{0,\ldots,n-1\}\}.
$$ Hence, every matrix $g \in \supp\nu^{(n)}$ admits the block-triangular decomposition shown in \eqref{eq.representationEperpproduct}. Moreover, its diagonal blocks, $g^E$ and $g^\perp$, are themselves elements of the supports of the corresponding projected $n$-convolutions $(\nu^E)^{(n)}$ and $(\nu^\perp)^{(n)}$. Furthermore, consider the functions 
\begin{align} 
    \qquad\Lambda^\perp : (\supp\nu)^{\mathbb{N}} &\longrightarrow GL(E_\nu^\perp)^{\mathbb{N}} &  \Lambda^E : (\supp\nu)^{\mathbb{N}} &\longrightarrow GL(E_\nu)^{\mathbb{N}} \\ \label{eq:def_Lambda}
    (g_k)_k &\longmapsto (\pi_\perp(g_k))_k = (g_k^\perp)_k
    &(g_k)_k &\longmapsto (\pi_E(g_k))_k = (g_k^E)_k.
\end{align} These maps relate to the product measures $\nu^{\mathbb{N}}, (\nu^E)^{\mathbb{N}}$ and $(\nu^\perp)^{\mathbb{N}}$ as follows:
\begin{equation} \label{eq:pushforwardmeasuressequences}
    (\nu^\perp)^{\mathbb{N}}  = ((\pi_\perp)_*\nu)^{\mathbb{N}} = (\Lambda^\perp)_*(\nu^{\mathbb{N}}) \quad \text{and}\quad
    (\nu^E)^{\mathbb{N}}  = ((\pi_E)_*\nu)^{\mathbb{N}} = (\Lambda^E)_*(\nu^{\mathbb{N}}).
\end{equation}

\subsection{Dynamics on \texorpdfstring{$E_\nu^\perp$}{Eperp}} \label{subsectiondynamicsonperp}

A crucial result from Furstenberg and Kifer \cite{FK83}*{Lemma 3.6} states that $\lambda_1(\nu)=\max\{\lambda_1(\nu^E), \lambda_1(\nu^\perp)\}$. Moreover, by the Furstenberg-Kesten Theorem \cite{LLE}*{Theorem 3.12} and the definition of equator, we have 
\begin{equation*}
        \lambda_1(\nu^E) =\lim_{n\rightarrow \infty} \frac{1}{n} \log\|g^E_{n-1}\ldots g^E_0\|=\lim_{n\rightarrow \infty} \frac{1}{n} \log\|g_{n-1}\ldots g_0 |_{E_\nu}\|< \lambda_1(\nu)
\end{equation*} for $\nu^\NN$-almost every $\fg=(g_k)_k\in G^\NN$. Thus, it follows that 
$\lambda_1(\nu^\perp)=\lambda_1(\nu)$. In other words, the dynamics on $E_\nu^\perp$ inherits the top Lyapunov exponent of the original system. Another important feature of the projected system on $E_\nu^\perp$ is the fact that the $\nu^\perp$-equator has dimension zero (see \cite{AEV2023}*{Section 3.2}). In other words, for all $\nu^\perp$-stationary measure $\eta^\perp$,
$$
       \alpha^\perp(\eta^\perp) = \int_{GL(E_\nu^\perp) \times \PP E_\nu^\perp} \log \frac{\|g^\perp u_\perp\|}{\|u_\perp\|} \, d(\nu^\perp \times \eta^\perp)=\lambda_1(\nu^\perp)=\lambda_1(\nu).
$$

\subsection{Examples of \texorpdfstring{$\nu$}{nu}-equators}

The following example shows that the $\nu$-equator may be trivial
even when the system is reducible.

\begin{example}
Let $d=3$ and consider the probability measure $\nu=p\delta_{A_1}+(1-p)\delta_{A_2}$, with $p\in(0,1)$, where
$$
A_1=\begin{pmatrix}\tau&1&1\\0&\gamma&0\\0&0&\gamma^{-1}\end{pmatrix},
\qquad
A_2=\begin{pmatrix}\hat\tau&0&0\\0&\cos\theta&\sin\theta\\0&-\sin\theta&\cos\theta\end{pmatrix},
$$
with $\tau>\gamma>1$, $\hat\tau>\gamma$ and $\theta\in(0,\pi/2)$. There exists a unique nonzero proper $\nu$-invariant subspace of this system, which is the line $\cL=\RR e_1$. So, in particular, the system is reducible. The Lyapunov exponent associated with $\cL$ is
$$
\lambda_{\cL}
=\lambda_1(\nu^{\cL})
=p\log\tau+(1-p)\log\hat\tau.
$$
On the quotient $\RR^3/\cL$, the maps induced by $A_1$ and $A_2$
are, respectively,
$$
\begin{pmatrix}
\gamma&0\\
0&\gamma^{-1}
\end{pmatrix}
\quad\text{and}\quad
\begin{pmatrix}
\cos\theta&\sin\theta\\
-\sin\theta&\cos\theta
\end{pmatrix}.
$$
Consequently,
$$\lambda_{\cL^\perp}\leq p\log\gamma
< p\log\tau+(1-p)\log\hat\tau
=\lambda_{\cL},$$ where $\lambda_{\cL^\perp}$ denotes the top Lyapunov exponent of
the induced cocycle on $\RR^3/\cL$. Hence $\lambda_1(\nu)=\lambda_{\cL}$ and this exponent is simple. Since $\cL$ carries the top Lyapunov exponent, it does not satisfy the condition in \eqref{definitionofequator}. As $\cL$ is the only nonzero proper $\nu$-invariant subspace, it follows that $E_\nu={0}$. Thus, this system is reducible but has a zero-dimensional $\nu$-equator.
\end{example}

The next example describes a system for which
$\lambda_1(\nu)$ is simple and $\dim E_\nu=1$.

\begin{example}
    Consider the example presented in \cite{AEV2023}*{Example 3.2}, for $d=3$: define $\nu=p \delta_{A_1}+ (1-p)\delta_{A_2}$, with $p\in(0,1)$, where the matrices $A_1$ and $A_2$ are given by
\begin{align}
    A_1 &= \begin{pmatrix}
        B&0\\0&1
    \end{pmatrix}\;\; \text{with}\;\; B=\begin{pmatrix}
        \gamma & 0 \\ 0&\gamma^{-1}    \end{pmatrix}\;\;\text{and}\;\; \gamma>1,\\[\jot] 
    A_2 &= \begin{pmatrix}
        R_\theta B & 0 \\ 0&1    \end{pmatrix}\;\;\text{with}\;\; R_\theta=\begin{pmatrix}
      \cos\theta&\sin\theta\\-\sin\theta & \cos\theta    \end{pmatrix}\;\;\text{and}\;\;\theta\ne0\;\;\text{small}.
\end{align} This system admits two $\nu$-invariant subspaces: $\cL_{12}=\RR^2\times \{0\}$, associated with the exponents $\lambda_1(\nu)\approx\log{\gamma}$ and $\lambda_3\approx -\log\gamma$, and $\cL_3=\{(0,0)\}\times \RR$, corresponding to the exponent $\lambda_2(\nu)=0$. The subspace $\cL_3$ is the $\nu$-equator in this example, and it satisfies $\dim \cL_3=1$. Note also that the multiplicity of $\lambda_1(\nu)$ is one.
\end{example}

%% file: body/3ProvaThmA.tex

Let $\nu\in\operatorname{Prob}_c(G)$ have a simple top Lyapunov exponent $\lambda_1(\nu)$. We begin by introducing some notation and establishing preliminary results that are independent of the dimension of the $\nu$-equator. We then divide the proof of Theorem~\ref{theorem.HolderCont} into two cases, according to whether $E_\nu$ is zero- or one-dimensional.

Fix a compact neighborhood $V=V(\nu)\subset G$ of the support of $\nu$ and a constant $M=M(\nu)\ge 1$ such that
\begin{equation} \label{eq.definitionVandM}
  \log(\|g\|\|g^{-1}\|) \le M, \quad \text{for every}\,\, g\in V.  
\end{equation} For every $g\in G$ and every $u,u'\in P$, we have
\begin{equation} \label{eq.propertydistancedistortion}
\ddp(gu,gu')\le \|g\|\|g^{-1}\|\; \ddp(u,u').
\end{equation} Consequently, 
\begin{equation}
    \label{eq.boundM}
    -\log \ddp(u,u') - M \le -\log \ddp(gu,gu')\le -\log \ddp(u,u') + M.
\end{equation}


For $n\ge1$, define $V^{(n)}=V^{(n)}(\nu)$ as the compact set of all products $g_{n-1}\ldots g_0$ with $g_i\in V$ for every $i\in\{0,\ldots,n-1\}$. We now choose a neighborhood $\cN_2\subset \operatorname{Prob}_c(G)$ of $\nu$ in the topology $\mathfrak{T}$ small enough to satisfy $\supp\nu'\subset V$ for every $\nu'\in \cN_2$. Consequently, for any measure $\nu'$ in this neighborhood and any $n\ge 1$, the support of its $n$-convolution is contained in the corresponding product set, \textit{i.e.}, $\supp\nu'^{\,(n)}\subset V^{(n)}$.

Since $\lambda_1(\nu)$ is simple, we can fix $\gamma_0=\gamma_0(\nu)>0$ such that
\begin{equation} \label{definitiongammazero}
    \lambda_1(\nu)-\lambda_2(\nu)>3\gamma_0.
\end{equation} This finishes the initial setup.

The following result, whose proof can be found in \cite{TV2020}*{Lemma 3.4}, gives a first estimate for the difference between the top Lyapunov exponent of $\nu$ and that of a nearby probability measure $\nu'$. It constitutes the first step towards establishing pointwise Hölder continuity, with maximal stationary measures, defined on \eqref{eq.lambadaoneintegralformula}, serving as the main ingredient in the proof.

\begin{proposition}
\label{prop.desmarkovprincipal}
    For any $\nu\in \operatorname{Prob}_c(G)$ there exists a neighborhood $\mathcal{N}_3=\mathcal{N}_3(\nu)$ of $\nu$ in $\operatorname{Prob}_c(G)$ and there is a constant $C_1=C_1(\nu)>0$ such that
    $$
        |\lambda_1(\nu)-\lambda_1(\nu')|\le C_1 \delta_\theta(\nu,\nu')+ C_1 \dtheta(\eta,\eta')
    $$
    for any $\nu'\in \mathcal{N}_3$, any maximal $\nu$-stationary measure $\eta$, any maximal $\nu'$-stationary measure $\eta'$ and any $\theta\in(0,1]$.  
\end{proposition}

The statement below is standard (see, for instance, the proof of Corollary 3.9 in \cite{TV2020}). It is an intermediate step toward the estimate we need and will be used to establish Hölder regularity in both the zero-dimensional and one-dimensional cases.

\begin{proposition} \label{prop.essentialinequalitydistancespartone} Let $\nu\in \operatorname{Prob}_c(G)$. There exists a neighborhood $\mathcal{N}_4 = \mathcal{N}_4 (\nu) \subset \operatorname{Prob}_c(G)$ of $\nu$ and there are constants $C_2=C_2(\nu)>0$ and $C_3=C_3(\nu)>2$ such that
    $$
        \dtheta(\eta,\eta') \le C_2 C_3^{n} \delta_\theta(\nu,\nu') + \dtheta(\cP_{\nu}^{*n}\eta', \cP_{\nu}^{*n}\eta),
    $$ for every $\nu'$ in $\mathcal{N}_4$, for any $\eta \in \operatorname{Stat}(\nu)$, for every $\eta'\in \operatorname{Stat}(\nu')$, for any natural number $n\ge 1$ and every $\theta\in (0,1]$.
\end{proposition}

\subsection{Outline of the proof of Theorem \ref{theorem.HolderCont}}
Our goal is to prove pointwise Hölder continuity of
the top Lyapunov exponent at $\nu$. In other words, we aim to find a constant $\theta_0\in (0,1)$ such that for every $\theta\in (0,\theta_0)$ there exists a neighborhood $\cN_0$ of $\nu$ in $\operatorname{Prob}_c(G)$ and there exist constants $C_0, \beta_0>0$ such that
$$
|\lambda_1(\nu) - \lambda_1(\nu')| \le C_0 \delta_{\fT,\theta}(\nu,\nu')^{\beta_0} \quad \text{for every } \nu'\in \cN_0.
$$

The first step towards this goal is Proposition \ref{prop.desmarkovprincipal}, which reduces the problem to obtaining an appropriate bound for the distance $\dtheta(\eta,\eta')$ between maximal stationary measures. By Proposition \ref{prop.essentialinequalitydistancespartone}, this distance is bounded by the sum of a perturbation term and an iterative term involving the dual Markov operator, namely,
$$
\dtheta(\eta,\eta') \le C_2 C_3^{n} \delta_{\theta}(\nu,\nu') + \dtheta(\cP_{\nu}^{*n}\eta, \cP_{\nu}^{*n}\eta')
$$
for every $\nu'$ in a neighborhood $\mathcal{N}_4$ of $\nu$. The core of the argument now rests on estimating the contractive properties of the iteration $\dtheta(\cP_{\nu}^{*n}\eta, \cP_{\nu}^{*n}\eta')$. We divide the proof into two cases depending on the dimension of the equator.

In the zero-dimensional case, our strategy adapts a classical argument by Le Page \cite{LP89}*{Proposition 1}. We establish uniform exponential contraction of the
iterates of the dual Markov operator on $\Prob(P)$, as we can see in Proposition \ref{prop.distancebetweeniteratesofmarkovtrivialequator},
$$
\dtheta(\cP_{\nu}^{*n}\eta, \cP_{\nu}^{*n}\eta') \le C_4 \tau_1^{n\theta}\dtheta(\eta,\eta').
$$
After sufficiently many iterations, this contraction allows us to bound the distance between stationary measures in terms of the size of the perturbation. Proposition \ref{prop.desmarkovprincipal} then yields pointwise Hölder continuity of $\lambda_1$, completing the zero-dimensional case.

In the one-dimensional case, the equator is an obstruction to such uniform contraction. Our strategy is to treat separately points away from the equator and points near it. Large deviation estimates (addressed in Subsection \ref{section.largedeviations}) provide the required control away from the equator. Near the equator, we instead show that suitable maximal stationary measures assign little mass to small neighborhoods of it. The key ingredient for this mass estimate is Theorem~\ref{theorem.estimateadriyaya}, proved in Section~\ref{s:thmB} using tools developed by Avila, Eskin, and Viana \cite{AEV2023}. An approximation argument then extends the resulting estimate to all perturbations satisfying the required support condition.

Combining these two ingredients gives an exponential estimate for the iterates of suitably chosen stationary measures, provided the perturbation of the support is sufficiently small relative to the number of iterations. We then choose the number of iterations according to the size of the perturbation. This yields a Hölder estimate for the distance between maximal stationary measures and, through Proposition~\ref{prop.desmarkovprincipal}, completes the proof in the one-dimensional case.

\subsection{Zero-dimensional equator}

\subsubsection{Dynamical setting on the space of pairs}
\label{section.dynamicalsettingspaceofpairs}

Consider the product space $P \times P$ and its diagonal $\Delta = \{ (u,u) \mid u \in P \}$. Define the space of distinct pairs of lines as $\mathfrak{X} = (P \times P) \setminus \Delta$. We construct a compactification $\overline{\mathfrak{X}}$ of this space by ``blowing up'' the diagonal $\Delta$ into a boundary, the flag space $F(2,d)$, defined as
$$
    F(2,d) := \{ (F_1, F_2) : F_1 \subset F_2 \subset \mathbb{R}^d, \, \dim F_1=1 \,\,\text{and}\,\, \dim F_2=2 \},
$$
endowed with the topology inherited from the product space $\text{Gr}(1,d)\times \text{Gr}(2,d)$. The topology on the compactified space $\overline{\mathfrak{X}} = \mathfrak{X} \cup F(2,d)$ is defined by specifying the convergence to the boundary. A sequence $(u_n, v_n)_{n \in \mathbb{N}} \subset \mathfrak{X}$ converges to a point $(F_1, F_2) \in F(2,d)$ if and only if:
    \begin{enumerate}
        \item [(i)] $\lim\limits_{n\to\infty} \ddp(u_n, v_n) = 0$;
        \item [(ii)] $\lim\limits_{n\to\infty} (F_1^{(n)},F_2^{(n)} ) = (F_1,F_2)$, where $F_1^{(n)} = \text{span}\{u_n\}$ and $F_2^{(n)}=\text{span}\{u_n, v_n\}$.
    \end{enumerate}

Moreover, the space $\mathfrak{X}$ is an open subset of $\overline{\mathfrak{X}}$ and convergence for sequences staying within $\mathfrak{X}$ is the standard one.

In what follows, we denote a generic point in the set $\overline{\mathfrak{X}}$ by $\omega$. Next, we define a crucial function $\Upsilon:G\times \mathfrak{X} \to \mathbb{R}$ that measures the logarithmic distortion of distances on the projective space. This map is defined by
\begin{equation*}
    \Upsilon(g,u,u') =  \log \dfrac{\ddp(gu,gu')}{\ddp(u,u')}.
\end{equation*}
We use $g$ to denote both an element of $G$ and the action it induces on the projective space $P$ and, later, on the space $\overline{\mathfrak{X}}$. Note that $\Upsilon$ is continuous. Moreover, $\Upsilon$ satisfies (by \eqref{eq.propertydistancedistortion}): 
\begin{equation}\label{eq.bounded}
|\Upsilon(g,u,u')| \le \log\left(\|g\| \|g^{-1}\|\right)
\text{ for every } (u,u')\in \mathfrak{X}.
\end{equation}

We can extend $\Upsilon$ continuously to the boundary $F(2,d)$ in the following way. Consider a point $(F_1, F_2)\in F(2,d)$ and let $(u_n, v_n)_n \subset \mathfrak{X}$ be a sequence that converges to it. From item (i), it follows that the sequence collapses to a single line in $P$. That is, there exists $u \in P$ such that $\lim u_n = \lim v_n = u$. Let us also use $u$ to denote a unit representative of this line in $\mathbb{R}^d$. Then, $F_1 = \text{span}\{u\}$, and $F_2$ is a two-dimensional subspace containing $F_1$.

The set of lines in $P$ that are contained within the plane $F_2$ forms a great circle, denoted by $\PP F_2$, in the spherical model of $P$. Note that $u=F_1 \in \PP F_2$. The tangent space to this circle at $u$, denoted by $T_u \PP F_2$, is a one-dimensional subspace of the full tangent space $T_u P$. Therefore, there exists a unit vector $\overline{u}$ that spans $T_u \PP F_2$.

\begin{figure}[h]
    \centering    \includegraphics[width=0.35\linewidth]{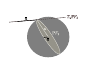}
    \caption{The point $u=F_1$ lies on the great circle $\PP F_2 \subset P$ and the vector $\overline{u}$ spans the tangent space $T_u \PP F_2$.}
\end{figure}

For any $g \in G$, we extend $\Upsilon$ to the boundary $F(2,d)$ by considering the derivative of the action 
$$
    \Upsilon(g, F_1, F_2) = \log \| D_u g \, \overline{u} \|.
$$
To see that the extension is well-defined, note that the limit point uniquely determines the line $u=F_1$ and the plane $F_2$. The sign ambiguity in choosing the spanning vector of $T_u \PP F_2$ is then eliminated by the norm, ensuring independence from the choice of the converging sequence.

Identifying the tangent space $T_u P$ with the orthogonal complement $u^\perp \subset \mathbb{R}^d$, the derivative $D_u g: T_u P \to T_{gu} P$ applied to a tangent vector $\dot{u} \in u^\perp$ is given by the formula
$$D_u g (\dot{u}) = \frac{\text{Proj}_{(g u)^\perp} \, g \dot{u}}{\|gu\|}\|u\|.$$
In summary, the extended map $\Upsilon:G\times \overline{\mathfrak{X}} \rightarrow \RR$ is given by
\begin{equation*}
    \Upsilon(g,\omega) = \begin{cases}
        \log\dfrac{\ddp(gu,gu')}{\ddp(u,u')}, \;&\text{if}\; \omega=(u,u')\in \mathfrak{X},\\[5mm]
        \log \dfrac{\|\text{Proj}_{(g u)^\perp} \, g \overline{u}\|}{\|gu\|},\;&\text{if}\; \omega=(F_1,F_2)\in F(2,d),
    \end{cases} 
\end{equation*} where $u$ and $\overline{u}$ are orthogonal unit vectors such that $F_1=\text{span}\{u\}$ and $F_2=\text{span}\{u,\overline{u}\}$.


The group $G$ acts continuously on $\overline{\mathfrak{X}}$ by 
\begin{equation}
    g\omega = \begin{cases} 
        (gu,gu'), & \text{for } \omega=(u,u')\in \mathfrak{X}, \\
        (gF_1,gF_2), & \text{for } \omega=(F_1,F_2)\in F(2,d).    
    \end{cases}
\end{equation}

This induces the following cocycle 
\begin{align*}
    \overline{\mathfrak{X}}\cF:G^\NN\times \overline{\mathfrak{X}} &\rightarrow G^\NN\times \overline{\mathfrak{X}}\\
    (\fg,\omega)&\mapsto (\sigma(\fg), A(\fg)\omega),
\end{align*} where $\fg=(g_k)_{k\ge0}\in G^\NN$ and the maps $\sigma$ and $A$ are exactly as in \eqref{eq.projectivecocycledefinition}. Then, for $n\ge 1$, we have $
    \overline{\mathfrak{X}}\cF^{\,n} (\fg,\omega) = ((g_{k+n})_{k\ge 0}\,,\, g_{n-1}\ldots g_0 \omega).
$

The associated Markov operator $\overline{\mathfrak{X}}\cP_{\nu}:\mathfrak{B}(\overline{\mathfrak{X}})\rightarrow \mathfrak{B}(\overline{\mathfrak{X}})$ and its dual $\overline{\mathfrak{X}}\cP_\nu^*$ act on functions and measures, respectively, by 
$$
    \overline{\mathfrak{X}}\cP_{\nu} (\varphi)(\omega) = \int_G \varphi(g\omega)\,d\nu(g) \quad \text{and} \quad \overline{\mathfrak{X}}\cP_\nu^* (\zeta) = \int_G g_* \zeta\; d\nu(g).
$$
A probability measure $\zeta$ is said to be $\overline{\mathfrak{X}}\cP_{\nu}$-invariant if it is a fixed point of the dual operator, i.e., $\zeta = \int_G g_* \zeta\; d\nu(g)$. For any probability measure $\zeta$ on $\overline{\mathfrak{X}}$, a fundamental equivalence holds (see \cite{LLE}*{Proposition 5.5}) :
 $$
    \zeta\; \text{is}\; \overline{\mathfrak{X}}\cP_{\nu}\text{-invariant if and only if } \nu^\NN\times\zeta\; \text{is invariant under } \overline{\mathfrak{X}}\cF.
 $$ Moreover, since $\overline{\mathfrak{X}}$ is a compact metric space, two properties follow: (i) There exists at least one probability measure which is $\overline{\mathfrak{X}}\cP_\nu$-invariant (see \cite{LLE}*{Proposition 5.6}). (ii) The dual operator $\overline{\mathfrak{X}}\cP_\nu^*$ is continuous relative to the $\text{weak}^*$ topology in the space of probability measures on $\overline{\mathfrak{X}}$ (see \cite{LLE}*{Lemma 5.8}).

A straightforward calculation shows that, on the open dense set $\mathfrak{X}$, the map $\Upsilon$ satisfies the relation $\Upsilon(gg',\omega) = \Upsilon(g, g'\omega) + \Upsilon(g', \omega)$. Since both the action of $G$ and the extended map $\Upsilon$ are continuous on the compactification $\overline{\mathfrak{X}}$, this identity extends to the boundary $F(2,d)$ by density. By an induction argument, for any $n \ge 1$ and $\omega \in \overline{\mathfrak{X}}$, we obtain 
\begin{equation}\label{eq.cocyclePsiomega}
    \Upsilon(g_{n-1}\ldots g_0,\omega) = \Upsilon(g_{n-1},g_{n-2}\ldots g_0\omega) +\ldots +\Upsilon(g_0, \omega).
\end{equation}
Now, define the projection $\pi: G^\NN \times \overline{\mathfrak{X}} \to G \times \overline{\mathfrak{X}}$ as $\pi(\fg, \omega) = (g_0, \omega)$. Then, the equality in \eqref{eq.cocyclePsiomega} can be compactly written as
\begin{equation}\label{eq.soma}
        \Upsilon(g_{n-1} \ldots g_0,\omega)
         = \sum_{j=0}^{n-1} (\Upsilon \circ \pi)\left(\,\overline{\mathfrak{X}}\cF^{\,j}(\fg,\omega)\right).
\end{equation}  This identity, which we will refer to as the cocycle property, allows us to analyze the evolution of $\Upsilon$ as a Birkhoff sum over the orbits of the cocycle.


\subsubsection{Estimates}
\label{section/3.4integralestimate}

The following proposition provides a key estimate that remains valid even when the $\nu$-equator $E_\nu \subset P$ is non-empty. To state it, consider the natural coordinate projections $\tilde{\pi}_1, \tilde{\pi}_2: \overline{\mathfrak{X}} \to P$ extended to the compactification. For a point $\omega = (u,u') \in \mathfrak{X}$, we set $\tilde{\pi}_1(\omega) = u$ and $\tilde{\pi}_2(\omega) = u'$. Since a boundary point $\omega = (F_1, F_2) \in F(2,d)$ represents the limit of a sequence where both coordinates converge to the same line, we define $\tilde{\pi}_1(\omega) = \tilde{\pi}_2(\omega) = F_1$.

\begin{proposition}\label{prop.integralestimate}
Let $\nu\in \operatorname{Prob}_c(G)$ with $\lambda_1(\nu)$ simple. Then,
$$
\int_{G \times \overline{\mathfrak{X}}} \Upsilon \, d(\nu \times \zeta) \leq \lambda_2(\nu) - \lambda_1(\nu)
$$ for any $\overline{\mathfrak{X}}\cP_\nu$-invariant probability measure $\zeta$ such that $(\tilde{\pi}_1)_*\zeta(E_\nu)=0$ and $(\tilde{\pi}_2)_*\zeta(E_\nu)=0$.
\end{proposition}


\begin{proof}

Let $\zeta$ be a probability measure on $\overline{\mathfrak{X}}$ satisfying the hypothesis, and let $\pi$ be as in \eqref{eq.soma}. From the relation $\nu\times\zeta=\pi_*(\nu^\NN\times\zeta)$, we get
    \begin{equation} \label{eq.Psiintegral1}
        \int_{G\times \overline{\mathfrak{X}}} \Psi \, d\nu \, d\zeta
         = \int_{G^\NN\times \overline{\mathfrak{X}}} (\Psi \circ \pi) \, d\nu^\NN \, d\zeta.
    \end{equation}
    
Since $\zeta$ is a $\overline{\mathfrak{X}}\cP_\nu^*$-invariant probability measure, the product probability $\nu^\NN\times\zeta$ is $\overline{\mathfrak{X}}\cF$-invariant. So, the Birkhoff ergodic theorem tells us that 
    \begin{equation}\label{eq.Psi}
    \begin{aligned}
      \int\limits_{G^\NN\times \overline{\mathfrak{X}}} (\Psi \circ \pi) \, d\nu^\NN \, d\zeta  & = \int\limits_{G^\NN\times \overline{\mathfrak{X}}}\lim_n \frac 1n \sum_{j=0}^{n-1} (\Psi \circ \pi)\circ \overline{\mathfrak{X}}\cF^j \,\,d\nu^\NN\,d \zeta
    \end{aligned}
    \end{equation} Thus, combining the cocycle property \eqref{eq.soma} with the equations \eqref{eq.Psiintegral1} and \eqref{eq.Psi}, we conclude that 
    \begin{equation} \label{eq.integraloflimitPsi}
        \int_{G\times \overline{\mathfrak{X}}} \Psi \, d\nu \, d\zeta = \int_{G^\NN\times \overline{\mathfrak{X}}} \lim_n \frac 1n \Psi(g_{n-1} \ldots g_0, \omega) \, d\nu^\NN(\fg) \, d\zeta(\omega).
    \end{equation} The idea is to prove that the limit inside the last integral is $\nu^\NN\times \zeta$-almost everywhere bounded from above by the expression $\lambda_2(\nu) - \lambda_1(\nu)$.

  We will split our analysis into two cases: points in $\mathfrak{X}$ and points in $F(2,d)$. First, let $\omega=(u,u')\in \mathfrak{X}$ and $\fg=(g_k)_{k\ge0} \in G^\NN$. Notice that
    \begin{equation*}
        \lim_n \frac 1n \Upsilon(g_{n-1} \ldots g_0, \omega) = \lim_n\frac 1n \Upsilon(A^n(\fg),u,u')
        = \lim_n\frac{1}{n}\log \ddp(A^n(\fg)u, A^n(\fg)u').
    \end{equation*}

Since $(\tilde{\pi}_1)_*\zeta(E_\nu)=0$ and $(\tilde{\pi}_2)_*\zeta(E_\nu)=0$, it follows that $u\notin E_\nu$ and $u'\notin E_\nu$ for $\zeta$-almost every point $(u,u')\in \mathfrak{X}$. Then, by \eqref{eq:dichotomy},
$$
    \lim_n\frac 1n \log \|g_{n-1}\ldots g_0 u\|=
    \lim_n\frac 1n \log \|g_{n-1}\ldots g_0 u'\|=\lambda_1(\nu)
$$
for $\nu^\NN\times\zeta$-almost every point $(\fg, u, u') \in G^\NN\times \mathfrak{X}$. Hence, $u\notin V^2_\fg$ and $u'\notin V^2_\fg$ for $\nu^\NN\times\zeta$-almost every point. 

Given that the first Lyapunov exponent $\lambda_1(\nu)$ is simple, the Oseledets subspace $V^2_\fg$ has codimension one in $\RR^d$ for $\nu^\NN$-almost every sequence $\fg$. Thus, it is a hyperplane. Let $\Sigma=\text{span}\{u,u'\}$ be the plane generated by $u$ and $u'$. Hence, the intersection $\Sigma\cap V^2_\fg$ must have dimension one, for $\nu^\NN\times \zeta$-almost every point $(\fg,u,u')\in G^\NN\times \mathfrak{X}$. Moreover, there is $w\in \Sigma\cap V^2_\fg\setminus\{0\}$ such that $v = u + w$ is parallel to $u'$, \textit{i.e.}, there exists $\gamma\in\RR$ with $v=\gamma u'$. 

For a nonzero vector $x\in\RR^d$, we write $[x]\in P$ for its projective class, consisting of all nonzero scalar multiples of $x$. Note that, for every $n\in\NN$,
$$
    [A^n(\fg)v] = [A^n(\fg)(\gamma u')] = [\gamma A^n(\fg)u'] = [A^n(\fg)u'].
$$ This implies that $\ddp(A^n(\fg)u,A^n(\fg)u')=\ddp(A^n(\fg)u,A^n(\fg)v)$ for all $n\in\NN$. Thus, we will estimate the distance between $A^n(\fg)u$ and $A^n(\fg)v$ instead of the original one.

Then, as a consequence of the Law of Sines, we get that for every $n\ge 1$:
    \begin{eqnarray*}
         \ddp(A^n(\fg)u,A^n(\fg)v) &=& \frac{\|A^n(\fg)u - A^n(\fg)v\|}{\|A^n(\fg)u\|}|\sin{\angle(A^n(\fg)v,A^n(\fg)w)}|
         \le \frac{\|A^n(\fg)w\|}{\|A^n(\fg)u\|}.
    \end{eqnarray*} Since $w\in V^2_\fg$ and $u \in \RR^d\setminus V^2_\fg$, $\nu^\NN\times\zeta$-almost everywhere in $G^\NN\times \mathfrak{X}$, then 
    \begin{eqnarray*}
        \lim_n\frac 1n\Upsilon(A^n(\fg),u,u') 
        &\le& \lim_n\frac{1}{n}\log\|A^n(\fg)w\| - \lim_n\frac{1}{n}\log\|A^n(\fg)u\|\\
        &\le& \lambda_2(\nu) - \lambda_1(\nu),\,\,\,\nu^\NN\times\zeta\text{-almost everywhere.}
    \end{eqnarray*}

Now, let $\omega = (F_1,F_2) \in F(2,d)$ and define $u$ and $\overline{u}$ as before, \textit{i.e.}, $u \in F_1$, $u$ and $\overline{u}$ are orthogonal unit vectors and $F_2 = \text{span}\{u,\overline{u}\}$. Then, for all $\fg=(g_k)_{k\ge0}\in G^\NN$,
\begin{align*}
    \lim_n\frac{1}{n} \Upsilon(g_{n-1}\ldots g_0, (F_1,F_2)) &= \lim_n \frac{1}{n} \log \|D_{u} g_{n-1}\ldots g_0 (\overline{u})\|\\ &= \lim_n \frac{1}{n} \log \frac{\|\operatorname{Proj}_{(g_{n-1}\ldots g_0 u)^{\perp}}g_{n-1}\ldots g_0 \overline{u}\|}{\|g_{n-1}\ldots g_0 u\|}
\end{align*}

We have that $(\tilde{\pi}_1)_*\zeta(E_\nu)=0$ and $(\tilde{\pi}_2)_*\zeta(E_\nu)=0$. This implies that $F_1 \not\subset E_\nu$ for $\zeta$-almost every $(F_1,F_2)\in F(2,d)$. Hence, $u\notin V^2_\fg$ for $(\nu^\NN\times \zeta)$-almost every $(\fg,(F_1,F_2))\in G^\NN\times F(2,d)$. Moreover, as mentioned earlier, whenever $V^2_\fg$ is well-defined, it is a hyperplane. In particular, this implies that $\dim(F_2 \cap V^2_\fg) = 1$.

Let $\fg\in G^\NN$ be such that there exists $V^2_\fg$ and $u\notin V^2_\fg$. There exists $\gamma\in \RR$ such that $w=\gamma u - \overline{u}$ is in $V^2_\fg$. Then, $\overline{u} = \gamma u - w$. Moreover, since $\operatorname{Proj}_{(g_{n-1}\ldots g_0 u)^{\perp}} \gamma g_{n-1}\ldots g_0 u = 0$, we have, for $\zeta$-almost every $(F_1,F_2)\in F(2,d)$,
\begin{align*}
    \lim_n\frac{1}{n} \Upsilon(g_{n-1}\ldots g_0, (F_1,F_2)) & = \lim_n \frac{1}{n} \log \frac{\left\|\operatorname{Proj}_{(g_{n-1}\ldots g_0 u)^{\perp}} g_{n-1}\ldots g_0 w\right\|}{\|g_{n-1}\ldots g_0 u\|}\\
    & \le \lim_n \frac{1}{n} \log \frac{\left\| g_{n-1}\ldots g_0 w\right\|}{\|g_{n-1}\ldots g_0 u\|}\\
    &\le \lim_n\frac{1}{n}\log\|A^n(\fg)w\| - \lim_n\frac{1}{n}\log\|A^n(\fg)u\|\\
    &\le \lambda_2(\nu) - \lambda_1(\nu),\,\,\,\nu^\NN\times\zeta\text{-almost everywhere.}
\end{align*} This follows from the fact that the projection has norm less or equal than one, $w\in V^2_\fg$ and $u \in \RR^d\setminus V^2_\fg$ for $\nu^\NN$-almost every $\fg
\in G^\NN$. In conclusion,
\begin{align}
     \int\limits_{G\times \overline{\mathfrak{X}}} \Upsilon \, d(\nu \times\zeta) &= \int\limits_{G^\NN\times \mathfrak{X}}\!\!\!\! \lim_n \frac 1n \Upsilon(g_{n-1} \ldots g_0, \omega) \, d(\nu^\NN\times\zeta) + \!\!\!\!\int\limits_{G^\NN\times F(2,d)} \!\!\!\! \lim_n \frac 1n \Upsilon(g_{n-1} \ldots g_0, \omega)\, d(\nu^\NN\times\zeta)\\[4pt]
     &\le \lambda_2(\nu) - \lambda_1(\nu).
\end{align} 

\end{proof}

From now on, assume that $E_\nu$ is zero-dimensional as a vector subspace (and, consequently, $E_\nu = \emptyset\subset P$). Under this assumption, the conditions $(\tilde{\pi}_1)_*\zeta(E_\nu)=0$ and $(\tilde{\pi}_2)_*\zeta(E_\nu)=0$ are trivially satisfied for any probability measure $\zeta$ on $\overline{\mathfrak{X}}$, meaning that Proposition \ref{prop.integralestimate} holds for every $\overline{\mathfrak{X}}\mathcal{P}_\nu$-invariant measure $\zeta$. With this setting established, our next goal is to prove the following exponential estimate.

\begin{proposition} \label{prop.exponentialestimatefortheintegral}
    Let $\nu\in \operatorname{Prob}_c(G)$ such that $\lambda_1(\nu)$ is simple and $\dim E_\nu=0$. There exists a constant $\theta_1=\theta_1(\nu)\in (0,1)$ such that, for every $\theta\in (0,\theta_1]$: 
    \begin{equation}        \lim_{n\rightarrow\infty}\left(\;\sup_{\omega\in \overline{\mathfrak{X}}}\,\int_{G}e^{\theta \Upsilon(g,\omega)}\,d\nu^{(n)}(g)\right)^{1/n} \le e^{-\gamma_0\theta}.
    \end{equation}
\end{proposition}

\begin{proof}
Let $\theta>0$ and $n\ge 1$. Define
$$
    \varrho_n(\theta) = \sup_{\omega \in \overline{\mathfrak{X}}} \int_G e^{\theta \Upsilon(g,\omega)}\,d\nu^{(n)}(g).
$$
The sequence $(\varrho_n(\theta))_{n\ge 1}$ is submultiplicative due to the cocycle property \eqref{eq.soma}. By Fekete's Lemma, $\lim_n \varrho_n(\theta)^{1/n} = \inf_n \varrho_n(\theta)^{1/n}$. Therefore, to establish the result, it is sufficient to find constants $\theta_1 > 0$ and $n_0\ge1$ such that
\begin{equation}\label{eq.suffices}
    \varrho_{n_0}(\theta)\le e^{-\gamma_0n_0\theta} \quad \text{for every } \theta\in (0,\theta_1].
\end{equation}

Let $M>0$ be as in \eqref{eq.definitionVandM}. This choice guarantees that $|\Upsilon(g,\omega)|\le Mn$ for every $g\in \supp\nu^{(n)}$, $\omega \in \overline{\mathfrak{X}}$ and $n\ge 1$. Recall that for every $\gamma\in \RR$, we have that $e^\gamma\le 1 + \gamma + \left(\frac{\gamma^2}{2}e^{|\gamma|}\right)$. Therefore, for all $n\ge 1$, 
\begin{equation}   \label{eq.boundoftherealexponential}
        \begin{aligned}
            \varrho_n(\theta)&\le 1 + \theta \sup_{\omega \in \overline{\mathfrak{X}}}\int_G \Upsilon(g,\omega)\,d\nu^{(n)}(g) + \frac{\theta^2}{2}\sup_{\omega \in \overline{\mathfrak{X}}}\int_G\Upsilon^2(g,\omega)e^{\theta|\Upsilon(g,\omega)|}\,d\nu^{(n)}(g)\\
            &\le 1 + \theta \sup_{\omega \in \overline{\mathfrak{X}}}\int_G \Upsilon(g,\omega)\,d\nu^{(n)}(g) + \frac{\theta^2}{2}(Mn)^2e^{Mn\theta}.
        \end{aligned}
    \end{equation}

To bound the linear term in $\theta$ we establish the following estimate.

\begin{claim} There exist infinitely many values of $n\ge 1$ such that
\begin{equation*}
    \sup_{\omega \in \overline{\mathfrak{X}}} \int_G \Upsilon(g,\omega)\,d\nu^{(n)}(g) \le -2\gamma_0 n.
\end{equation*}
\label{claim.infinitelymanyns}
\end{claim}

Indeed, by compactness and continuity, for each $n\ge1$, the supremum is attained at some $\omega_n \in \overline{\mathfrak{X}}$. Then, using \eqref{eq.soma},
        \begin{align*}
            \frac{1}{n}\sup_{\omega \in \overline{\mathfrak{X}}} \int_G \Upsilon(g,\omega)\,d\nu^{(n)}(g)&= \frac{1}{n} \int_G \Upsilon(g,\omega_n)\,d\nu^{(n)}(g)\\
            &= \int_{G^\NN} \frac{1}{n} \sum\limits_{j=0}^{n-1} \Upsilon(g_j,g_{j-1}\ldots g_0 \omega_n)\,d\nu^\NN(g_k)_{k\ge0}.
        \end{align*} The expression above may be rewritten as 
        \begin{equation*}
            \int_{G\times \overline{\mathfrak{X}}} \Upsilon(g,\omega)\,d\nu(g)\,d\zeta_n(\omega)\quad\text{with}\quad \zeta_n:=\frac{1}{n} \sum\limits_{j=0}^{n-1} \int_G g_*\delta_{\omega_n}\,d\nu^{(j)}(g),
        \end{equation*}where $\delta_{\omega_n}$ is the Dirac mass at $\omega_n$. Restricting to a subsequence, if necessary, we may suppose that $(\zeta_n)_n$ converges in the weak$^*$-topology to some probability measure $\zeta$ on $\overline{\mathfrak{X}}$. Thus, by continuity,
        \begin{equation} \label{eq.convergenceoftheintegral}
            \frac{1}{n}\sup_{\omega \in \overline{\mathfrak{X}}} \int_G \Upsilon(g,\omega)\,d\nu^{(n)}(g) \quad\text{ converges to}\quad \int_{G\times \overline{\mathfrak{X}}} \Upsilon\,d\nu\,d\zeta. 
        \end{equation} Note that $\zeta_n$ can be expressed as $ \zeta_n:=\frac{1}{n} \sum\limits_{j=0}^{n-1} \overline{\mathfrak{X}}\cP_\nu^{*j}(\delta_{\omega_n})$. Then, a direct computation shows that the dual operator satisfies
        $$
            \overline{\mathfrak{X}}\cP_\nu^{*}\zeta_n = \zeta_n + \frac{1}{n} \mathfrak{X}\cP_\nu^{*n}(\delta_{\omega_n}) - \frac{1}{n}\delta_{\omega_n}.
        $$ Moreover, $\overline{\mathfrak{X}}\cP_\nu^*$ is continuous relative to the weak$^*$-topology in $\operatorname{Prob}(\overline{\mathfrak{X}})$ (see \cite{LLE}*{Lemma 5.8}). So, taking the limit when $n$ goes to infinity, we get
        $$
            \overline{\mathfrak{X}}\cP_\nu^{*}\zeta = \zeta,
        $$ that is, the measure $\zeta$ is $\overline{\mathfrak{X}}\cP_\nu$-invariant. Since the $\nu$-equator has dimension zero, we have $(\tilde{\pi}_1)_*\zeta(E_\nu)=0$ and $(\tilde{\pi}_2)_*\zeta(E_\nu)=0$.

        Therefore, by Proposition \ref{prop.integralestimate}, the limit in \eqref{eq.convergenceoftheintegral} is at most $\lambda_2(\nu)-\lambda_1(\nu)$, which is strictly smaller than $-3\gamma_0$ (see \eqref{definitiongammazero}). Consequently, there must exist infinitely many values of $n$ for which the expression on the left-hand side of \eqref{eq.convergenceoftheintegral} is less than $-2\gamma_0$, proving the claim.

To conclude the proof of the proposition, fix an integer $n_0\ge 1$ satisfying the Claim \ref{claim.infinitelymanyns}. Substituting this choice into \eqref{eq.boundoftherealexponential} yields
$$
\varrho_{n_0}(\theta)\le 1 - 2\gamma_0 n_0 \theta + \frac{\theta^2}{2}(Mn_0)^2e^{Mn_0\theta}.
$$ We may chose $\theta_1=\theta_1(\nu)\in (0,1)$ small enough that the right-hand side is bounded by $e^{-\gamma_0n_0\theta}$ for all $\theta\in (0,\theta_1]$, completing the proof.
\end{proof}

We now analyze the contractive properties of the dual Markov operator $\cP_\nu^*$. Specifically, we will estimate the distance $\dtheta(\cP_\nu^{*n}\eta, \cP_\nu^{*n}\eta')$ between the iterates of two arbitrary probability measures $\eta,\eta'\in \operatorname{Prob}(P)$. This contraction is the key component required to establish the distance estimate between stationary measures in Proposition \ref{prop.essentialinequalitydistancespartone}.

In what follows, let $\theta_1=\theta_1(\nu)\in (0,1)$ be the constant given by Proposition \ref{prop.exponentialestimatefortheintegral}.

\begin{proposition}
\label{prop.distancebetweeniteratesofmarkovtrivialequator}
Let $\nu\in\operatorname{Prob}_c(G)$ be such that
$\lambda_1(\nu)$ is simple and $\dim E_\nu=0$.
For every $\theta\in(0,\theta_1]$, there exist constants
$C_4 = C_4(\nu,\theta)>1$ and
$\tau_1=\tau_1(\nu,\theta)\in(0,1)$ such that
$$    \dtheta(\cP_\nu^{*n}\eta,\cP_\nu^{*n}\eta')
    \leq C_4 \tau_1^{n\theta}\dtheta(\eta,\eta')
$$
for every $\eta,\eta'\in\operatorname{Prob}(P)$
and every integer $n\geq1$.
\end{proposition}

\begin{proof}
Choose $0<\gamma<\gamma_0$, for instance
$\gamma=\gamma_0/2$, and set
$
    \tau_1=e^{-\gamma}\in(0,1).
$
For each $\theta\in(0,\theta_1]$, define
$$
    \varrho_n(\theta)
    =
    \sup_{\omega\in\overline{\mathfrak X}}
    \int_G e^{\theta\Upsilon(g,\omega)}\,d\nu^{(n)}(g).
$$
By Proposition~\ref{prop.exponentialestimatefortheintegral},
$$
    \lim_{n\to\infty}\varrho_n(\theta)^{1/n}
    \leq e^{-\gamma_0\theta}
    < e^{-\gamma\theta}
    = \tau_1^\theta.
$$
Consequently, for each fixed $\theta\in(0,\theta_1]$,
there exists a constant
$C_4=C_4(\nu,\theta)>1$ such that
\begin{equation} \label{eq.estimaterho}
    \varrho_n(\theta)
    \leq C_4\tau_1^{n\theta}
\end{equation} for every $n\geq1$. Recall that 
\begin{equation*}
    \dtheta (\cP_\nu^{*n}\eta,\cP_\nu^{*n}\eta') = \sup\left\{\int_{P}\varphi(u) d(\cP_\nu^{*n}\eta-\cP_\nu^{*n}\eta')(u): \varphi \in L_\theta(P)\right\},
\end{equation*} 
where $L_\theta(P)$ denotes the set of all $1$-Lipschitz functions on $P$ with respect to the metric $(\ddp)^{\theta}$. Let $\varphi\in L_\theta(P)$ and $n\geq1$.
For distinct points $u,u'\in P$, the definition
of the Markov operator gives
\begin{align*}
    |\cP_\nu^n\varphi(u)-\cP_\nu^n\varphi(u')|
    &\le
    \int_G|\varphi(gu)-\varphi(gu')|\,d\nu^{(n)}(g)
    \le
    \int_G\ddp(gu,gu')^\theta\,d\nu^{(n)}(g)\\
    &\le
    \ddp(u,u')^\theta
    \int_G e^{\theta\Upsilon(g,u,u')}\,d\nu^{(n)}(g)
    \le \varrho_n(\theta)\ddp(u,u')^\theta.
\end{align*} Thus, by \eqref{eq.estimaterho},
$$
 |\cP_\nu^n\varphi(u)-\cP_\nu^n\varphi(u')| \leq C_4\tau_1^{n\theta}\ddp(u,u')^\theta.
$$ The same estimate holds trivially when $u=u'$.
Therefore,
$
    \frac{\cP_\nu^n\varphi}
         { C_4 \tau_1^{n\theta}}
    \in L_\theta(P).
$ Consequently, we conclude that for any $\varphi \in L_\theta(P)$ and for every
$\eta,\eta'\in\operatorname{Prob}(P)$,
$$
    \left|
        \int_P\varphi\,
        d(\cP_\nu^{*n}\eta-\cP_\nu^{*n}\eta')
    \right| = \left|
        \int_P\cP_\nu^n\varphi\,d(\eta-\eta')
    \right|\leq C_4\tau_1^{n\theta}
    \dtheta(\eta,\eta').
$$
Taking the supremum over $\varphi\in L_\theta(P)$
completes the proof.

\end{proof}

When the $\nu$-equator has dimension zero, the condition $\eta(E_\nu)=0$ holds for every probability measure $\eta$ on $P$. If, in addition, the first Lyapunov exponent is simple, the uniqueness result of Aoun and Guivarc'h (see Remark \ref{remark.uniquenessofmaximalstationary}) ensures that there exists a unique $\nu$-stationary measure, denoted by $\eta_\nu$, which is necessarily maximal. Let $\mathcal{N}_4 = \mathcal{N}_4(\nu)\subset \Prob_c(G)$ be the neighborhood of $\nu$ defined in Proposition \ref{prop.essentialinequalitydistancespartone}.

\begin{proposition}
\label{prop.distanceofthestationarytrivialequator}
Let $\nu\in\operatorname{Prob}_c(G)$ be such that
$\lambda_1(\nu)$ is simple and $\dim E_\nu=0$. For every $\theta\in(0,\theta_1]$,
there exists a constant $C_5=C_5(\nu,\theta)>0$
satisfying
$$
    \dtheta(\eta_\nu,\eta')
    \leq C_5\delta_\theta(\nu,\nu')
$$
for every $\nu'\in\mathcal{N}_4$ and every $\eta'$ $\nu'$-stationary.
\end{proposition}

\begin{proof}
By Proposition \ref{prop.essentialinequalitydistancespartone}, there exist constants $C_2>0$ and $C_3>2$ such that
\begin{equation}\label{eq.absorption_step1}
    \dtheta(\eta_\nu,\eta') \le C_2 C_3^{n} \delta_\theta(\nu,\nu') + \dtheta(\cP_{\nu}^{*n}\eta_\nu, \cP_{\nu}^{*n}\eta')
\end{equation}
for every $\nu'$ in $\mathcal{N}_4$, every $\nu'$-stationary measure $\eta'$, all integer $n\ge 1$, and all $\theta\in (0,1]$. 

Fix
$\theta\in(0,\theta_1]$. By Proposition \ref{prop.distancebetweeniteratesofmarkovtrivialequator} there exist constants $C_4=C_4(\nu,\theta)>1$ and
$\tau_1=\tau_1(\nu)\in(0,1)$ such that
$$
    \dtheta (\cP_\nu^{*n}\eta_\nu,\cP_\nu^{*n}\eta') \le C_4\tau_1^{n\theta}\dtheta(\eta_\nu,\eta')
$$ for every integer $n\geq1$ and every
$\eta'\in\operatorname{Prob}(P)$. Combining this estimate with \eqref{eq.absorption_step1}, we obtain
$$
    \dtheta(\eta_\nu,\eta') \le C_2 C_3^{n} \delta_\theta(\nu,\nu') + C_4 \tau_1^{n\theta}\dtheta(\eta_\nu,\eta')
$$ for every $\nu'\in \mathcal{N}_4$, every $\nu'$-stationary
probability measure $\eta'$, and every integer $n\geq1$.

Since $\tau_1\in(0,1)$ and $\theta>0$, we may choose $\bar{n}=\bar{n}(\nu,\theta)\ge1$ as the smallest integer satisfying $C_4\tau_1^{\bar{n}\theta}\le 1/2$, and set $C_5 = C_5(\nu,\theta) = 2C_2 C_3^{\bar{n}}$. Applying the preceding estimate with $n=\bar n$ yields
\begin{align*}
    \dtheta(\eta_\nu, \eta') \le \frac{1}{2} \dtheta (\eta_\nu, \eta') + \frac{C_5}{2} \delta_\theta(\nu,\nu').
\end{align*}
Subtracting $\frac{1}{2}\dtheta(\eta_\nu, \eta')$ from both sides and multiplying by 2 completes the proof.
\end{proof}

\subsubsection{Completing the proof of the zero-dimensional case}
\label{section/3.5conclusion}

Set $\theta_0=\theta_1$ and fix
$\theta\in(0,\theta_0]$.
Recall that $\eta_\nu$ is the unique $\nu$-stationary
probability measure and it is maximal.

By Proposition~\ref{prop.desmarkovprincipal}, there exist
a neighborhood $\mathcal{N}_3$ of $\nu$ and a constant $C_1>0$
such that
$$
    |\lambda_1(\nu)-\lambda_1(\nu')|
    \leq
    C_1\delta_\theta(\nu,\nu')
    +C_1\dtheta(\eta_\nu,\eta')
$$
for every $\nu'\in \mathcal{N}_3$ and every maximal
$\nu'$-stationary measure $\eta'$.

By Proposition~\ref{prop.distanceofthestationarytrivialequator},
there exist a neighborhood $\mathcal{N}_4$ of $\nu$,
independent of $\theta$, and a constant
$C_5=C_5(\nu,\theta)>0$ such that
$$
    \dtheta(\eta_\nu,\eta')
    \leq C_5\delta_\theta(\nu,\nu')
$$
for every $\nu'\in\mathcal{N}_4$ and every
$\nu'$-stationary measure $\eta'$.

Set $\cN_0=\mathcal{N}_3\cap \mathcal{N}_4$.
For any $\nu'\in\cN_0$, choose a maximal
$\nu'$-stationary probability measure $\eta'$.
Combining the preceding estimates gives
$$
\begin{aligned}
    |\lambda_1(\nu)-\lambda_1(\nu')|
    &\leq
    C_1\delta_\theta(\nu,\nu')
    +C_1\dtheta(\eta_\nu,\eta')\\
    &\leq
    C_1(1+C_5)\delta_\theta(\nu,\nu')\\
    &\leq
    C_1(1+C_5)\delta_{\mathfrak T,\theta}(\nu,\nu').
\end{aligned}
$$
Thus, defining
$ C_0= C_0(\nu,\theta)=C_1(1+C_5)$, we obtain
$$
    |\lambda_1(\nu)-\lambda_1(\nu')|
    \leq
     C_0\delta_{\mathfrak T,\theta}(\nu,\nu')
$$ for every $\nu'\in\cN_0$. This proves the zero-dimensional case of
Theorem~\ref{theorem.HolderCont}, with $\beta_0=1$.

\subsection{One-dimensional equator}

\subsubsection{Large deviation estimates}
            \label{section.largedeviations}

The goal of this section is to establish quantitative estimates for the convergence of finite-time growth rates toward the corresponding Lyapunov exponents. These estimates will play an important role in the proof of Proposition \ref{propIntegralEstimateDistance} and are inspired by \cite{AEV2023}*{Theorem 4.1}.

Throughout this subsection, we assume that
$\nu\in\Prob_c(G)$ has a nontrivial equator:
$\dim E_\nu\geq1$.
The equator need not be one-dimensional, and we do not
assume that the top Lyapunov exponent is simple.

To motivate the first estimate, we consider the projected dynamics on $E_\nu^\perp$. Since the $\nu^\perp$-equator has dimension zero, it follows from \cite{FK83}*{Theorem 3.5} that for every $u_\perp\in\PP E_\nu^\perp$,
$$
\lim_{n\to\infty}\frac1n \log \frac{\|g_{n-1}^\perp\cdots g_0^\perp u_\perp\|}{\|u_\perp\|}
= \lambda_1(\nu),
$$
for $(\nu^\perp)^\NN$-almost every sequence $(g_k^\perp)_{k\ge0}$. Consequently, for every fixed $u_\perp\in\PP E_\nu^\perp$
and every $\gamma>0$, the
$(\nu^\perp)^\NN$-measure of the set of sequences satisfying
$$
    \frac1n\log
    \frac{\|g_{n-1}^\perp\cdots g_0^\perp u_\perp\|}
         {\|u_\perp\|}
    \in(\lambda_1(\nu)-\gamma,\lambda_1(\nu)+\gamma)
    \qquad\text{for every }n\geq n_1
$$
tends to one as $n_1\to\infty$. Since our arguments are carried out in the original sequence space $G^\NN$, we will consider the pullback of such sets under the projection map $\Lambda_\perp$ (defined in \eqref{eq:def_Lambda}). Proposition \ref{propEperpuniformmeasure} shows that the complement of the resulting set in $G^\NN$ has exponentially small measure.

To motivate the corresponding bound on the equator, a similar analysis is applied to the dynamics on $E_\nu$. By the definition of the $\nu$-equator,
$$
\lim_{n\to\infty}\frac1n \log \frac{\|g^E_{n-1}\cdots g^E_0u_E\|}{\|u_E\|}
\le \lambda_2(\nu)
$$
for every $u_E\in\PP E_\nu$ and $(\nu^E)^\NN$-almost every sequence $(g^E_k)_{k\ge0}$. Proposition \ref{propEuniformmeasure} provides the corresponding quantitative estimate in this setting.

\begin{remark}[One-dimensional systems]
\label{rem:one-dimensional-large-deviations}
Although the original system acts on $\mathbb R^d$
with $d\geq2$, the auxiliary systems considered below
may act on one-dimensional spaces.
The estimates of interest remain valid in this case, as follows.

Let $\nu\in\operatorname{Prob}_c(GL(1,\mathbb R))$. First, we identify $GL(1,\mathbb R)$ with
$\mathbb R\setminus\{0\}$. Let $a_j$ denote the $j$-th coordinate of the sequence $(a_j)_{j\ge0}\in GL(1,\mathbb R)^{\mathbb N}$.
Under the distribution $\nu^{\mathbb N}$, the random variables
$X_j=\log|a_j|$ are independent and identically
distributed. They are bounded because $\nu$ has
compact support in $GL(1,\mathbb R)$.
Since
$$
    \frac1n\log|a_{n-1}\cdots a_0|
    =\frac1n\sum_{j=0}^{n-1}X_j,
$$
the strong law of large numbers gives
$$
    \lambda_1(\nu)
    =\lim_{n\to\infty}\frac1n\log|a_{n-1}\cdots a_0|
    =\mathbb E[X_0]
    =\int\log|a|\,d\nu(a)
$$
for $\nu^{\mathbb N}$-almost every sequence. We may choose real numbers $L<L'$ such that
$$
    L\leq\log|a|\leq L'
$$ for every $a\in\operatorname{supp}\nu$. Now, fix $\gamma>0$.
Applying Hoeffding's inequality
\cite{Hoeffding1963}*{Theorem 2} to $X_j$ and to $-X_j$ and combining
the two one-sided bounds, we obtain
$$
    \nu^{\mathbb N}\left(
        \left\{
        (a_j)_{j\geq0}:
        \left|
        \frac1n\log|a_{n-1}\cdots a_0|-\lambda_1(\nu)
        \right|\geq\gamma
        \right\}
    \right)
    \leq
    2e^{-\frac{2\gamma^2}{(L'-L)^2}n}
$$
for every $n\geq1$. Set $c=2\gamma^2/(L'-L)^2>0$.
For each integer $N\geq1$, define the measurable set
$$
    \mathcal E_N
    =
    \left\{
        (a_j)_{j\geq0}:
        \left|
        \frac1n\log|a_{n-1}\cdots a_0|-\lambda_1(\nu)
        \right|<\gamma
        \text{ for every }n\geq N
    \right\}.
$$
Hence,
$$
    \nu^{\mathbb N}(\mathcal E_N^c)
    \leq
    \sum_{n=N}^{\infty}2e^{-cn}
    =
    \frac{2}{1-e^{-c}}e^{-cN}.
$$
Thus, outside a set of exponentially small measure,
both the upper and lower estimates hold for every $n\geq N$. 
\end{remark}


\begin{proposition}
    \label{propEperpuniformmeasure}
    Let $\nu \in \operatorname{Prob}_c(G)$. For all $\gamma>0$, there exist constants $B_1= B_1(\nu,\gamma)>0$ and $b_1 = b_1(\nu,\gamma)>0$ such that for every $u_{\perp}\in \PP E_\nu^{\perp}$ and every integer $ n_1\ge 1$, there exists a measurable set $\xi_1=\xi_1(\nu, \gamma, u_{\perp},n_1)\subset G^\NN$ satisfying:
    
    \begin{itemize}
        \item [(a)] $\nu^\NN(G^\NN \setminus \xi_1)\leq B_1e^{-b_1n_1}$ and

        \item [(b)] For all $(g_k)_{k\ge0}\in\xi_1$ and $n\geq n_1$ 
        $$\dfrac 1n \log\dfrac{\|g_{n-1}^\perp\ldots g_0^\perp u_{\perp}\|}{\|u_{\perp}\|}\in (\lambda_1(\nu)-\gamma, \lambda_1(\nu)+\gamma).$$
    \end{itemize}
\end{proposition}

\begin{proof}

Since the $\nu^\perp$-equator has dimension zero, every
$\nu^\perp$-stationary measure $\eta^\perp$ satisfies
$$
    \alpha^\perp(\eta^\perp)
    =\lambda_1(\nu^\perp)
    =\lambda_1(\nu).
$$
If $\dim E_\nu^\perp\geq2$, we apply
\cite{AEV2023}*{Theorem 4.1} to
$(GL(E_\nu^\perp),\nu^\perp)$.
If $\dim E_\nu^\perp=1$, we instead apply
Remark~\ref{rem:one-dimensional-large-deviations}. In either case, for every $\gamma>0$ there exist constants
$B_1=B_1(\nu^\perp,\gamma)>0$ and
$b_1=b_1(\nu^\perp,\gamma)>0$ such that,
for every $u_\perp\in\PP E_\nu^\perp$ and every integer
$n_1\geq1$, there exists a measurable set
$
    \cE_1
    =\cE_1(\nu^\perp,\gamma,n_1,u_\perp)
    \subset GL(E_\nu^\perp)^\NN
$
satisfying:
\begin{itemize}
    \item[(a')]
    $(\nu^\perp)^\NN
    \bigl(GL(E_\nu^\perp)^\NN\setminus\cE_1\bigr)
    \leq B_1 e^{-b_1 n_1}$;

    \item[(b')]
    for every $(g_k^\perp)_{k\geq0}\in\cE_1$
    and every $n\geq n_1$,
    \begin{equation}\label{eq.convergenceprojectedperp}
        \frac1n\log
        \frac{\|g_{n-1}^\perp\cdots g_0^\perp u_\perp\|}
             {\|u_\perp\|}
        \in
        \bigl(\lambda_1(\nu)-\gamma,\lambda_1(\nu)+\gamma\bigr).
    \end{equation}
\end{itemize}

Notice that $\cE_1$ is a subset of $GL(E_\nu^\perp)^\NN$. Our goal is to find a subset of $G^\NN$ that satisfies properties (a) and (b). To achieve this, we recall the function  $\Lambda^\perp$, given in \eqref{eq:def_Lambda}, and define
$$
    \xi_1:=(\Lambda^\perp)^{-1}(\cE_1) \subset G^\NN.
$$ Since $(\Lambda^\perp)_*\nu^\NN=(\nu^\perp)^\NN$, we have
$$
    \nu^\NN(G^\NN \setminus \xi_1) = (\nu^\perp)^\NN( GL(E_\nu^\perp)^\NN \setminus \cE_1)\leq B_1 e^{-b_1 n_1}.
$$ If $(g_k)_{k\ge0}\in\xi_1$, then
$\Lambda^\perp((g_k)_{k\ge0})\in\cE_1$
and therefore the estimate \eqref{eq.convergenceprojectedperp} holds for every
$n\ge n_1$. 
   
\end{proof}


\begin{lemma}
\label{lemRuniformmeasure}
Let $m\geq1$ and let
$\mu\in\operatorname{Prob}_c(GL(m,\mathbb R))$.
For every $\gamma>0$, there exist constants
$B_2=B_2(\mu,\gamma)>0$ and $b_2=b_2(\mu,\gamma)>0$
such that, for every $u\in\mathbb P\mathbb R^m$
and every integer $n_2\geq1$, there exists a measurable
set
$\xi_2=\xi_2(\mu,\gamma,u,n_2)
\subset GL(m,\mathbb R)^{\mathbb N}$
satisfying:
\begin{enumerate}
    \item[(a)]
    $\mu^{\mathbb N}(GL(m,\RR)^\NN\setminus \xi_2)\leq B_2e^{-b_2 n_2}$;

    \item[(b)]
    for every $(g_j)_{j\geq0}\in\xi_2$
    and every $n\geq n_2$,
    $$
        \frac1n\log
        \frac{\|g_{n-1}\cdots g_0u\|}{\|u\|}
        <\lambda_1(\mu)+\gamma.
    $$
\end{enumerate}
\end{lemma}

\begin{proof}
If $m=1$, the conclusion follows from
Remark~\ref{rem:one-dimensional-large-deviations}.
Assume henceforth that $m\geq2$ and set
$G=GL(m,\mathbb R)$ and $P=\mathbb P\mathbb R^m$.

We use the notation introduced in \cite{AEV2023}*{Section 4} for the system $(G,\mu)$, as follows. Let $S=\supp\mu$ and define $\varphi=\Phi\circ\iota^{-1}$, where $\iota(g,v)=(g,gv)$ and $\Phi$ is given by \eqref{def.Phi}. Then
$$
\varphi(g,v)
=
\log\frac{\|v\|}{\|g^{-1}v\|}.
$$

Consider the operator $\cQ_\mu$ defined on $C(S\times P)$ by
$$
        \cQ_\mu \psi(g',u) = \int_S \psi(g,gu)\,d\mu(g),
$$ which does not depend on $g'$. The dual operator $\cQ_\mu^*$ acts in the space of finite signed measures on the product $S\times P$ by: given $\psi\in C(S\times P)$,
$$
    \int_{S\times P} \psi\,d\cQ_\mu^*\zeta = \int_{S\times P} \cQ_\mu\psi\,d\zeta.
$$ A measure $\zeta$ is said to be $\cQ_\mu$-invariant if $\cQ_\mu^*(\zeta)=\zeta$. Moreover, define
$$
\Sigma(\mu,\varphi) = \sup\left\{\int_{S\times P} \varphi\,d\zeta\,:\,\zeta\; \text{is a } \cQ_\mu\text{-invariant probability measure}\right\}.
$$ Using the definition of $\varphi$, the result \cite{AEV2023}*{Lemma 4.7} and equation \eqref{eq.lambdaonemaximumintegral}, we obtain
\begin{align*}
\Sigma(\mu,\varphi)
&=
\sup
\left\{
\int_{S\times P}
\varphi\,d\zeta\,:\,\zeta\; \text{is a } \cQ_\mu\text{-invariant probability measure}
\right\} \\
&=
\sup
\left\{
\int_{S\times P}
\Phi\,d(\mu\times\eta):
\eta\ \text{is }\mu\text{-stationary}
\right\} \\
&=
\lambda_1(\mu).
\end{align*}

Applying \cite{AEV2023}*{Lemma 4.5} with
$\varepsilon=\gamma$, we obtain constants $B_2=B_2(\mu, \gamma)>0$ and $b_2=b_2(\mu, \gamma)>0$ such that for every $u\in P$ and every integer $n_2\ge1$ there exists a measurable set $\xi_2=\xi_2(\mu,\gamma,u,n_2)\subset G^\NN$ satisfying
\begin{enumerate}
    \item[(a')] $\mu^\NN(G^\NN\setminus\xi_2    ) \le B_2e^{-b_2n_2}$ and

    \item[(b')] for every $(g_k)_{k\ge0}\in\xi_2$ and every $n\ge n_2$,
    \begin{equation*}
        \frac1n
        \log
        \frac{\|g_{n-1}\cdots g_0u\|}
             {\|u\|} = \frac1n
        \sum_{j=0}^{n-1}
        \varphi(g_j,g_j\cdots g_0u) < \Sigma(\mu,\varphi)+\gamma = \lambda_1(\mu)+\gamma.
    \end{equation*} 
\end{enumerate}

This
concludes the proof.

\end{proof}


\begin{proposition}
    \label{propEuniformmeasure}
    Let $\nu\in \operatorname{Prob}_c(G)$ such that $\dim E_\nu\ge 1$. For all $\gamma>0$, there exist constants $B_3= B_3(\nu,\gamma)>0$ and $b_3= b_3(\nu,\gamma)>0$ such that for every $u_E\in \PP E_\nu$ and every integer $ n_3\ge 1$, there exists a measurable set $\xi_3=\xi_3(\nu, \gamma, u_E,n_3)\subset G^\NN$ satisfying:
    
    \begin{itemize}
        \item [(a)] $\nu^\NN(G^\NN \setminus \xi_3)\leq B_3e^{-b_3n_3}$ and

        \item [(b)] $\dfrac 1n \log\dfrac{\|g_{n-1}^E\ldots g_0^E u_E\|}{\|u_E\|} < \lambda_2(\nu)+\gamma$, for all $(g_k)_{k\ge0}\in\xi_3$ and $n\geq n_3$.
    \end{itemize}
\end{proposition}

\begin{proof}

Consider the restricted system $(GL(E_\nu),\nu^E)$. Applying Lemma \ref{lemRuniformmeasure}, for every $\gamma>0$  there exist constants $B_3=B_3(\nu^E,\gamma)>0$ and $b_3=b_3(\nu^E,\gamma)>0$ such that for every $u_E\in \PP E_\nu$ and every integer $n_3\ge 1$ there exists a measurable set $\cE_3\subset GL(E_\nu)^\NN$ satisfying
    \begin{enumerate}
        \item [(a')] $(\nu^E)^\NN(GL(E_\nu)^\NN\setminus \cE_3)\le B_3e^{-b_3n_3}$ and
        \item [(b')] For every $(g^E_k)_k\in \cE_3$ and $n\ge n_3$, we have \begin{equation*}
            \frac{1}{n}\log\frac{\|g^E_{n-1}\ldots g^E_0 u_E\|}{\|u_E\|} <\lambda_1(\nu^E) + \gamma
            \le\lambda_2(\nu) + \gamma. 
        \end{equation*}
    \end{enumerate}

Define $\xi_3=(\Lambda^E)^{-1}(\cE_3)\subset G^\NN$, where the map $\Lambda^E$ is given by \eqref{eq:def_Lambda}. Since $(\Lambda^E)_*\nu^\NN=(\nu^E)^\NN$, we obtain
$$
\nu^\NN(G^\NN\setminus\xi_3)
=
(\nu^E)^\NN(GL(E_\nu)^\NN\setminus\cE_3)
\le
B_3e^{-b_3n_3}.
$$ Moreover, if $(g_k)_{k\ge0}\in\xi_3$, then $\Lambda^E((g_k)_{k\ge0})=(g_k^E)_{k\ge0}\in\cE_3$. Hence, for every $n\ge n_3$,
$$
\frac1n
\log
\frac{\|g_{n-1}^E\cdots g_0^E u_E\|}
{\|u_E\|}
<
\lambda_2(\nu)+\gamma.
$$ Finally, the proof is complete.

\end{proof}

\begin{proposition}\label{propExtuniformmeasure}
Let $\nu \in \operatorname{Prob}_c(G)$. For all $\gamma > 0$, there exist constants $B_4 = B_4(\nu,\gamma) > 0$ and $b_4 = b_4(\nu,\gamma) > 0$ such that for every $w \in \mathbb{P}\left(\bigwedge^2 \mathbb{R}^d\right)$ and every integer $n_4 \ge 1$, there exists a measurable set $\xi_4 = \xi_4(\nu, \gamma, w, n_4) \subset G^\mathbb{N}$ satisfying:
\begin{enumerate}
    \item [(a)] $\nu^\mathbb{N}(G^\mathbb{N} \setminus \xi_4) \le B_4 e^{-b_4 n_4}$, and
    \item [(b)] $\dfrac{1}{n} \log \dfrac{\|\wedge^2(g_{n-1}\cdots g_0) w\|}{\|w\|} < \chi_1(\nu) + \chi_2(\nu) + \gamma$ for all $(g_k)_{k \ge 0} \in \xi_4$ and every $n \ge n_4$.
\end{enumerate}
\end{proposition}
\begin{proof}
Consider the linear group $\tilde{G} \coloneqq \operatorname{GL}(\bigwedge^2 \mathbb{R}^d)$. Define the map $\pi_{\wedge 2} : G \to \tilde{G}$ by $$\pi_{\wedge 2}(g) = \wedge^2 g,$$ which gives the representation of an invertible matrix on the second exterior power. Consider the push-forward probability measure $\tilde{\nu} \coloneqq (\pi_{\wedge 2})_* \nu$. By Remark~\ref{rem:secondexteriorpower}, we conclude that 
$\tilde\nu\in\Prob_c(\tilde G)$ and
$$
    \lambda_1(\tilde\nu)=\chi_1(\nu)+\chi_2(\nu).
$$

Applying Lemma~\ref{lemRuniformmeasure} to $(\tilde{G},\tilde\nu)$,
we obtain, for every $\gamma > 0$, constants $B_4 > 0$ and $b_4 > 0$ such that, for every $w \in \mathbb{P}\left(\bigwedge^2 \mathbb{R}^d\right)$ and every integer $n_4 \ge 1$, there is a measurable set $\cE_4 \subset \tilde{G}^\mathbb{N}$ satisfying:
\begin{enumerate}
    \item [(i)] $\tilde{\nu}^\mathbb{N}(\tilde{G}^\mathbb{N} \setminus \cE_4) \le B_4 e^{-b_4 n_4}$.
    \item [(ii)] For every sequence $(\tilde{g}_k)_{k \ge 0} \in \cE_4$ and all $n \ge n_4$:
    \begin{equation*}
        \frac{1}{n} \log \frac{\|\tilde{g}_{n-1} \cdots \tilde{g}_0 w\|}{\|w\|} < \lambda_1(\tilde{\nu}) + \gamma = \chi_1(\nu) + \chi_2(\nu) + \gamma.
    \end{equation*}
\end{enumerate} The constants $B_4$ and $b_4$ depend only on $\nu$
and $\gamma$, since $\tilde\nu$ is determined by $\nu$. To conclude, let $\Lambda^{\wedge 2} : G^\mathbb{N} \to \tilde{G}^\mathbb{N}$ be the map defined by $\Lambda^{\wedge 2}((g_k)_{k\ge 0}) = (\pi_{\wedge 2}(g_k))_{k\ge 0}$, and define the set:
\begin{equation*}
    \xi_4 \coloneqq (\Lambda^{\wedge 2})^{-1}(\cE_4) \subset G^\mathbb{N}.
\end{equation*} Since $\Lambda^{\wedge 2}$ is continuous, the set
$\xi_4$ is measurable. Since $(\Lambda^{\wedge 2})_*\nu^\NN=\tilde\nu^\NN$, we have:
\begin{equation*}
    \nu^\mathbb{N}(G^\mathbb{N} \setminus \xi_4) = (\Lambda^{\wedge 2})_*\nu^\mathbb{N}(\tilde{G}^\mathbb{N} \setminus \cE_4) = \tilde{\nu}^\mathbb{N}(\tilde{G}^\mathbb{N} \setminus \cE_4) \le B_4 e^{-b_4 n_4}.
\end{equation*}
This proves {\rm (a)}. Moreover, since
$$
    (\wedge^2g_{n-1})\cdots(\wedge^2g_0)
    =\wedge^2(g_{n-1}\cdots g_0),
$$
property {\rm (ii)} implies that, for every
$(g_k)_{k\geq0}\in\xi_4$ and every $n\geq n_4$,
$$
    \frac1n\log
    \frac{\|\wedge^2(g_{n-1}\cdots g_0)w\|}{\|w\|}
    <\chi_1(\nu)+\chi_2(\nu)+\gamma.
$$
This proves {\rm (b)} and completes the proof.
\end{proof}

\subsubsection{Exponential estimates}
    \label{sectionhyperbolicestimate}

First, we establish an upper bound for the distance between the measures $\cP_\nu^{*n}\eta$ and $\cP_\nu^{*n}\eta'$, for any $\eta,\eta'\in \operatorname{Prob}(P)$. Then, we will apply this bound to demonstrate that this distance decays exponentially fast when considering a specific class of measures.

\begin{lemma}
\label{lemma:integralestimateofthedistanceiteratemarkovoperator}
    Let $\eta,\eta'\in \operatorname{Prob}(P)$. Then,
    $$
        \dtheta(\cP_\nu^{*n}\eta,\cP_\nu^{*n}\eta') \le \int_{G\times P\times P} \ddp (gu,gu')^\theta d\nu^{(n)}(g)d\eta(u)d\eta'(u')
    $$
    for every $n\in\NN, n\ge1$, and every $\theta\in (0,1]$.
\end{lemma}
\begin{proof}
    Consider the product measure $\tilde{\eta}=\eta\times\eta'$ on $P\times P$. Let $\pi_1$ and $\pi_2$ be the coordinate maps defined in \eqref{eq.definitionofprojectionspi}. Since ${\pi_1}_* \tilde{\eta}=\eta$ and ${\pi_2}_* \tilde{\eta} =\eta'$, it follows that $\tilde{\eta}$ is a coupling of $\eta$ and $\eta'$. Then, by Lemma \ref{lemma:ndiagonalpushforward}, $\tilde{\cP_\nu^{*n}}\tilde{\eta}$ is a coupling of $\cP^{*n}_\nu\eta$ and $\cP^{*n}_\nu\eta'$, for every $n\ge1$.  

    By the definition of $\dtheta$, for every $n\ge 1$ and every $\theta\in(0,1]$ we have
    \begin{eqnarray*}
        \dtheta(\cP_\nu^{*n}\eta,\cP_\nu^{*n}\eta') &=& \inf\left\{\int_{P\times P}\ddp^\theta d\xi:\xi\;\, \textnormal{is a coupling of}\;\, \cP_\nu^{*n}\eta \;\,\textnormal{and}\;\, \cP_\nu^{*n}\eta'\right\}\\[\jot]
        &\le& \int_{P\times P}\ddp(u,u')^\theta d\tilde{\cP}_\nu^{*n}(\tilde{\eta})(u,u')\\[\jot]
        &=& \int_{G\times P\times P} \ddp(gu,gu')^\theta d\nu^{(n)}(g)d\eta(u)d\eta'(u'),
    \end{eqnarray*} where the last step follows from \eqref{eq:ncouplingexpression}.

\end{proof}

To obtain an exponential estimate for $\dtheta(\cP_\nu^{*n}\eta,\cP_\nu^{*n}\eta')$, our approach is to first control the integral term $\int_{G} \ddp(gu,gu')^\theta d\nu^{(n)}(g)$. Proposition \ref{propIntegralEstimateDistance} provides such an estimate for points $u$ and $u'$ away from the $\nu$-equator. This alone is not enough to deduce exponential decay
of this distance, since it does not control the contribution of pairs with at least one point near the equator. To handle these remaining pairs, we use Proposition \ref{prop:estimateforallnotonlygeneric},
which provides maximal stationary measures assigning little mass to neighborhoods of the equator, under a suitable condition on the support of the perturbation. Combining these two ingredients, we obtain the
exponential decay estimate in
Proposition~\ref{propEstimateDistanceBalancedMeasures}
for suitably chosen maximal stationary measures,
under the required support condition.

For $r>0$, let $E_\nu(r)$ denote the closed $r$-neighborhood of $E_\nu$ in $P$. Moreover, recall the constant $\gamma_0>0$ from
\eqref{definitiongammazero}.

\begin{proposition}\label{propIntegralEstimateDistance}
    Let $\nu \in \operatorname{Prob}_c(G)$ be such that $\lambda_1(\nu)$ is simple and $\dim E_\nu\geq1$. For every $\gamma \in (0, \gamma_0/3]$, there exist constants $C_6 = C_6(\nu, \gamma) > 0$, $\tau_2 = \tau_2(\nu, \gamma) \in (0,1)$, and $\theta_2=\theta_2(\nu,\gamma)\in(0,1]$ such that, for every $\theta \in (0, \theta_2]$ and every integer $n \ge 1$, the inequality
    \begin{equation}
        \int_G \ddp(gu,gu')^\theta d\nu^{(n)}(g) \le C_6 \tau_2^{n\theta}
    \end{equation}
    holds uniformly for all pairs $u, u' \in P$ satisfying $u \notin E_\nu(e^{-n\gamma})$ and $u' \notin E_\nu(e^{-n\gamma})$.
\end{proposition}

\begin{proof}

Let $\nu\in\Prob_c(G)$ be such that $\lambda_1(\nu)$
is simple and $\dim E_\nu\geq1$.

Fix $\gamma\in(0,\gamma_0/3]$.
Fix an integer $n\geq1$ and points
$u\notin E_\nu(e^{-n\gamma})$ and $u'\notin E_\nu(e^{-n\gamma})$. Choose nonzero vector representatives, still denoted
by $u,u'$, and write their orthogonal decompositions as $u=u_E+u_\perp$ and $u'=u'_E+u'_\perp$. Since both projective points lie outside the equator,
$u_\perp$ and $u'_\perp$ are nonzero.

 For $\theta\in (0,1]$, we get from the triangle inequality on the projective space that
    $$
        \ddp(gu,gu')^\theta\le \ddp(gu,gu_\perp)^\theta+\ddp(gu_\perp,gu'_\perp)^\theta+\ddp(gu',gu'_\perp)^\theta.
    $$

Consequently, integrating with respect to $\nu^{(n)}$, we obtain
\begin{align}
    \int_G \ddp(gu,gu')^\theta d\nu^{(n)}(g) &\le \int_G \ddp(gu,gu_\perp)^\theta d\nu^{(n)}(g)  +\int_G \ddp(gu_\perp,gu'_\perp)^\theta d\nu^{(n)}(g)\nonumber \\ \label{eq.triangularinequalityintegral}
    & \qquad +\int_G \ddp(gu'_\perp,gu')^\theta d\nu^{(n)}(g).
\end{align}

We will now estimate each of the integrals on the right-hand side of 
\eqref{eq.triangularinequalityintegral}.

\proofpart{Part 1 - First integral} We consider the case where $u \notin E_\nu^\perp$, as otherwise the distance $\ddp(gu, gu_\perp)$ vanishes identically. By the Law of Sines, Proposition \ref{propEperpuniformmeasure}, and Proposition \ref{propEuniformmeasure}, for a given $n_1=n_2=n \in \mathbb{N}$, and for all $(g_k)_{k\ge0} \in \xi_1(\nu, \gamma, u_\perp,n) \cap \xi_3(\nu, \gamma, u_E, n)$, the following holds
\begin{align*}
    \ddp(g_{n-1}\ldots g_0 u, g_{n-1}\ldots g_0 u_\perp) &\le \frac{\|g_{n-1}\ldots g_0 u-g_{n-1}\ldots g_0u_\perp\|}{\|g_{n-1}\ldots g_0 u_\perp\|} = \frac{\|g_{n-1}\ldots g_0 u_E\|}{\|g_{n-1}\ldots g_0 u_\perp\|}\\[3pt]
    &\le e^{n(\lambda_2 + \gamma)} e^{-n(\lambda_1 - \gamma)} \frac{\|u_E \|}{\| u_\perp \|}\le e^{n(\lambda_2 - \lambda_1 + 3\gamma)},
\end{align*}
this last inequality holds because $\|u_E\|\le e^{n\gamma}\|u_\perp\|$. This follows from the fact that $u\notin E_\nu(e^{-n\gamma})$. Indeed, 
$u\notin E_\nu(e^{-n\gamma})$ implies that $\ddp(u,v)\ge e^{-n\gamma}$, for all $v\in E_\nu$. In particular, take $v=u_E$. On the other hand, we have that $\ddp(u,u_E)\le \frac{\| u - u_E\|}{\| u_E\|}= \frac{\| u_\perp \|}{\| u_E\|}$. This yields the claim.

Since $5\gamma < 3 \gamma_0< \lambda_1(\nu) - \lambda_2(\nu)$, we have
\begin{equation} \label{eq:differencelambdasgamma}
    \lambda_2(\nu)-\lambda_1(\nu)+3\gamma
    < -2\gamma.
\end{equation}
Consequently, for all $(g_k)_{k\ge0} \in \xi_1 \cap \xi_3$,
    $$
        \ddp(g_{n-1}\ldots g_0 u, g_{n-1}\ldots g_0 u_\perp)^\theta \le e^{-2n\gamma\theta}.
    $$ Splitting the domain of integration into $\xi_1 \cap \xi_3$ and its complement, we obtain
    \begin{align*}
        \int_G \ddp(gu,gu_\perp)^\theta d\nu^{(n)}(g) &= \int_{G^\NN} \ddp(g_{n-1}\ldots g_0u,g_{n-1}\ldots g_0u_\perp)^\theta d\nu^\NN((g_k)_{k\ge0})\\
        &= \int_{\xi_1 \cap \xi_3} \ddp(g_{n-1}\ldots g_0u,g_{n-1}\ldots g_0u_\perp)^\theta d\nu^\NN((g_k)_{k\ge0})\\
         & \qquad + \int_{G^\NN\setminus (\xi_1 \cap \xi_3)} \ddp(g_{n-1}\ldots g_0u,g_{n-1}\ldots g_0 u_\perp)^\theta d\nu^\NN((g_k)_{k\ge0})\\
        &\le e^{-2n\gamma\theta} + B_{(1)}e^{-b_{(1)}n},
    \end{align*}
again by Propositions \ref{propEperpuniformmeasure} and \ref{propEuniformmeasure}, choosing $B_{(1)}:= 2\max\{B_1,B_2\}$ and $b_{(1)}:=\min\{b_1, b_2\}$. Notice that these constants depend only on $\nu$ and $\gamma$.

\proofpart{Part 2 - Second integral} We now estimate the middle term of \eqref{eq.triangularinequalityintegral}:
\begin{equation*}
    \int_G \ddp(gu_\perp,gu'_\perp)^\theta \, d\nu^{(n)}(g).
\end{equation*}
Recall that $u_\perp, u'_\perp \in E_\nu^\perp \setminus \{0\}$ are the orthogonal projections of $u$ and $u'$. We use them directly as representatives for their respective projective directions. Let $\wedge$ denote the wedge product. If $u_\perp \wedge u'_\perp = 0$, then the projective directions coincide, the integrand vanishes identically, and the estimate holds trivially. Thus, we may assume $u_\perp \wedge u'_\perp \neq 0$.

Let $g = g_{n-1}\cdots g_0$ with $g_i\in \supp\nu$ for $i\in \{0,\ldots,n-1\}$. Since $u_\perp, u'_\perp \in E_\nu^\perp$, their components in $E_\nu$ are zero. By the block-triangular representation of the product established in \eqref{eq.representationEperpproduct}, we have
\begin{align*}
    gu_\perp &= \begin{pmatrix} h(n)u_\perp \\ g^\perp_{n-1}\cdots g^\perp_0 u_\perp \end{pmatrix} 
    \\[2pt]
    gu'_\perp &= \begin{pmatrix} h(n)u'_\perp \\ g^\perp_{n-1}\cdots g^\perp_0 u'_\perp \end{pmatrix}
\end{align*}
where $h(n)$ is given by \eqref{eq.definitionofhn}. In particular, by the orthogonality, we have
\begin{equation}\label{eq.normlowerboundgx}
    \|gu_\perp\| \ge \|g^\perp_{n-1}\cdots g^\perp_0 u_\perp\| \qquad \text{and} \qquad \|gu'_\perp\| \ge \|g^\perp_{n-1}\cdots g^\perp_0 u'_\perp\|.
\end{equation}

The projective distance formula gives
 \begin{align}
    \ddp(gu_\perp,gu'_\perp) &= \frac{\|gu_\perp \wedge gu'_\perp\|}{\|gu_\perp\| \, \|gu'_\perp\|} = \frac{\|\wedge^2(g_{n-1}\cdots g_0)(u_\perp \wedge u'_\perp)\|}{\|g_{n-1}\cdots g_0 u_\perp\| \, \|g_{n-1}\cdots g_0 u'_\perp\|}\\[4pt]
    &\le \frac{\|\wedge^2(g_{n-1}\cdots g_0)(u_\perp \wedge u'_\perp)\|}{\|g^\perp_{n-1}\cdots g^\perp_0 u_\perp\|\, \|g^\perp_{n-1}\cdots g^\perp_0 u'_\perp\|}.  \label{eq.inequalitywithexteriorandperp}
\end{align}

On the other hand, let $\xi_{u_\perp} \coloneqq \xi_1(\nu, \gamma, u_\perp, n)$ and $\xi_{u'_\perp} \coloneqq \xi_1(\nu, \gamma, u'_\perp, n)$ be the sets given by Proposition \ref{propEperpuniformmeasure}. On the intersection $\xi_{u_\perp} \cap \xi_{u'_\perp}$, we have the lower bounds
\begin{equation*}
    \|g^\perp_{n-1}\cdots g^\perp_0 u_\perp\| \ge e^{n(\lambda_1(\nu)-\gamma)}\|u_\perp\| 
    \qquad \text{and} \qquad 
    \|g^\perp_{n-1}\cdots g^\perp_0 u'_\perp\| \ge e^{n(\lambda_1(\nu)-\gamma)}\|u'_\perp\|.
\end{equation*}

Furthermore, applying Proposition \ref{propExtuniformmeasure} to the second exterior power cocycle with the 2-vector $w = u_\perp \wedge u'_\perp$, we get a set $\xi_{u_\perp \wedge u'_\perp} \coloneqq \xi_4(\nu, \gamma, u_\perp \wedge u'_\perp, n)$,  such that for every sequence $(g_k)_{k\ge0} \in \xi_{u_\perp \wedge u'_\perp}$, we have
\begin{equation*}
    \|\wedge^2(g_{n-1}\cdots g_0)(u_\perp \wedge u'_\perp)\| \le e^{n(\lambda_1(\nu)+\lambda_2(\nu)+\gamma)} \|u_\perp \wedge u'_\perp\|.
\end{equation*} Notice that the equalities $\chi_1(\nu)=\lambda_1(\nu)$ and $\chi_2(\nu)=\lambda_2(\nu)$ follows from the simplicity of $\lambda_1(\nu)$.

Hence, for every sequence in the intersection $\xi_{(2)} \coloneqq \xi_{u_\perp} \cap \xi_{u'_\perp} \cap \xi_{u_\perp \wedge u'_\perp}$, combining \eqref{eq.inequalitywithexteriorandperp} with the estimates obtained above, we get
\begin{align*}
    \ddp(gu_\perp,gu'_\perp) &\le \frac{\|\wedge^2(g_{n-1}\cdots g_0)(u_\perp \wedge u'_\perp)\|}{\|g^\perp_{n-1}\cdots g^\perp_0 u_\perp\|\, \|g^\perp_{n-1}\cdots g^\perp_0 u'_\perp\|}\\[5pt]
    &\le \frac{e^{n(\lambda_1(\nu)+\lambda_2(\nu)+\gamma)} \|u_\perp \wedge u'_\perp\|}{e^{2n(\lambda_1(\nu)-\gamma)} \|u_\perp\| \, \|u'_\perp\|} \\[5pt]
    &= e^{n(\lambda_2(\nu)-\lambda_1(\nu)+3\gamma)} \ddp(u_\perp,u'_\perp) \\[5pt]
    &\le e^{n(\lambda_2(\nu)-\lambda_1(\nu)+3\gamma)}\\
    &\le e^{-2n\gamma}.
\end{align*} The last inequality comes from \eqref{eq:differencelambdasgamma}. By Propositions~\ref{propEperpuniformmeasure}
and~\ref{propExtuniformmeasure},
$$
\begin{aligned}
    \nu^\NN(G^\NN\setminus\xi_{(2)})
    &\leq
    \nu^\NN(G^\NN\setminus\xi_{u_\perp})
    +\nu^\NN(G^\NN\setminus\xi_{u'_\perp}) +
    \nu^\NN(G^\NN\setminus
                 \xi_{u_\perp\wedge u'_\perp})\\
    &\leq2B_1e^{-b_1n}+B_4e^{-b_4n}\\
    &\leq B_{(2)}e^{-b_{(2)}n},
\end{aligned}
$$
where $B_{(2)}=2B_1+B_4$ and $b_{(2)}=\min\{b_1,b_4\}$. To conclude, since the projective distance $\ddp$ is uniformly bounded by $1$, splitting the domain of integration yields
\begin{align*}
    \int_G \ddp(gu_\perp,gu'_\perp)^\theta \, d\nu^{(n)}(g) &= \int_{G^\mathbb{N}} \ddp(g_{n-1}\cdots g_0u_\perp, g_{n-1}\cdots g_0u'_\perp)^\theta \, d\nu^\mathbb{N}((g_k)_{k\ge0}) \\[5pt]
    &\le e^{-2n\gamma\theta} + \nu^\mathbb{N}(G^\mathbb{N}\setminus \xi_{(2)}) \\[5pt]
    &\le e^{-2n\gamma\theta} + B_{(2)}e^{-b_{(2)}n}.
\end{align*}

\proofpart{Part 3 - Third integral} The analysis of the third integral follows by symmetry from the estimates established in Part 1. Applying the estimate from Part 1 to $u'$ and using
the symmetry of the projective distance, we obtain
$$
    \int_G\ddp(gu'_\perp,gu')^\theta\,d\nu^{(n)}(g)
    \leq e^{-2n\gamma\theta}
         +B_{(1)}e^{-b_{(1)}n}.
$$

\proofpart{Part 4 - Final estimate}
Combining the estimates from Parts 1, 2 and 3 with
\eqref{eq.triangularinequalityintegral}, we obtain
$$
    \int_G\ddp(gu,gu')^\theta\,d\nu^{(n)}(g)
    \leq
    3e^{-2n\gamma\theta}
    +2B_{(1)}e^{-b_{(1)}n}
    +B_{(2)}e^{-b_{(2)}n}.
$$
Set $B_0=2B_{(1)}+B_{(2)}$ and     $b_0=\min\{b_{(1)},b_{(2)}\}$. Then,
$$
    \int_G\ddp(gu,gu')^\theta\,d\nu^{(n)}(g)
    \leq3e^{-2n\gamma\theta}+B_0e^{-b_0n}.
$$

Choose
$$
    \theta_2=
    \min\left\{1,\frac{b_0}{2\gamma}\right\}
    \in(0,1].
$$
For every $\theta\in(0,\theta_2]$, we have
$2\gamma\theta\leq b_0$, and hence $e^{-b_0n}\leq e^{-2n\gamma\theta}$. Consequently,
$$
    \int_G\ddp(gu,gu')^\theta\,d\nu^{(n)}(g)
    \leq(3+B_0)e^{-2n\gamma\theta}.
$$
Taking $C_6=3+B_0$ and $\tau_2=e^{-2\gamma}\in(0,1)$, we obtain
$$
    \int_G\ddp(gu,gu')^\theta\,d\nu^{(n)}(g)
    \leq C_6\tau_2^{n\theta}.
$$
The constants $C_6$ and $\tau_2$ depend only on $\nu$ and $\gamma$.
Thus the estimate holds uniformly for every integer
$n\geq1$, every $\theta\in(0,\theta_2]$, and every
pair $u,u'\notin E_\nu(e^{-n\gamma})$,
completing the proof.
\end{proof}

 We have established a contraction estimate for directions outside a small neighborhood of the equator. To control the contribution of the remaining directions, we now derive a bound of the form $(\eta'(E(r))\le Cr^\beta)$, with $(C,\beta>0)$, for suitable maximal stationary measures. We first obtain this bound for generic measures by iterating Theorem \ref{theorem.estimateadriyaya}. We then extend it to arbitrary nearby measures by approximation. These estimates will allow us to combine contraction away from the equator with control of the mass near it. The next result is a direct consequence of Theorem \ref{theorem.estimateadriyaya}.

\begin{corollary} \label{cor:balance_measureorange}
There exists $C_7 = C_7(\nu)>0$ such that, setting
$\beta_1 = \beta_1(\nu) =\log(2/3)/\log(q)$, the following holds.
For every $r\in(0,1)$, every $\nu'\in\mathcal{N}_1$ generic satisfying $\supp\nu'\subset B_{sr}(\supp\nu)$, the unique $\eta'\in\Stat(\nu')$ satisfies
$$
    \eta'(E(r))\leq C_7 r^{\beta_1}.
$$
\end{corollary}

\begin{proof}
First, recall $\varepsilon_0 \in (0,1]$ defined in Theorem \ref{theorem.estimateadriyaya} and set
$\beta_1=\frac{\log(2/3)}{\log q}>0$ and $C_7=(q\varepsilon_0)^{-\beta_1}$. Fix $r\in(0,1)$ and $\nu'\in\mathcal{N}_1$ generic such that
$$
    \supp\nu'\subset B_{sr}(\supp\nu),
$$
and let $\eta'$ be the unique $\nu'$-stationary probability measure.

Suppose first that $0<r\leq\varepsilon_0$. Let $n\geq0$ be an integer 
such that
$$
    q^{n+1}\varepsilon_0<r\leq q^n\varepsilon_0,
$$
and define $t_j=q^j\varepsilon_0$ for $j=0,\dots,n$.
For every $j=0,\dots,n-1$, we have
$$
    r\leq t_n\leq t_j\le \varepsilon_0.
$$
Consequently, for every $j=0,\dots,n$ we get
$$
    \supp\nu'\subset B_{sr}(\supp\nu)
    \subset B_{st_j}(\supp\nu).
$$
Thus, Theorem~\ref{theorem.estimateadriyaya} applies at each
of these scales and gives
$$
    \eta'(E(t_{j+1}))\leq\frac23\eta'(E(t_j)).
$$
Iterating, we obtain
$$
    \eta'(E(t_n))
    \leq\left(\frac23\right)^n\eta'(E(\varepsilon_0))
    \leq\left(\frac23\right)^n.
$$
When $n=0$, the same bound follows directly from
$\eta'(E(\varepsilon_0))\leq1$.

Since $r\leq t_n$, monotonicity of $E(\cdot)$ yields
$$
    \eta'(E(r))
    \leq\eta'(E(t_n))
    \leq\left(\frac23\right)^n.
$$
Moreover, $q^{\beta_1}=2/3$ and $q^{n+1}\varepsilon_0<r$, so
$$
    \left(\frac23\right)^n
    =(q^n)^{\beta_1}
    \leq\left(\frac{r}{q\varepsilon_0}\right)^{\beta_1}
    =C_7 r^{\beta_1}.
$$

Finally, if $\varepsilon_0<r<1$, then
$$
    \eta'(E(r))\leq1
    \leq\left(\frac{r}{q\varepsilon_0}\right)^{\beta_1}
    =C_7 r^{\beta_1}.
$$
This proves the desired estimate.
\end{proof}

For $C_7$ and $\beta_1$ as the previous corollary, the following holds.

\begin{proposition} \label{prop:estimateforallnotonlygeneric} There exists a neighborhood $\mathcal{N}_5=\mathcal{N}_5(\nu)\subset \Prob_c(G)$ of $\nu$ such that for every $r\in(0,1)$ and every $\nu'\in \mathcal{N}_5$ satisfying $\supp\nu'\subset B_{sr}(\supp\nu)$, there exists a maximal $\nu'$-stationary measure $\eta'$
such that
$$
    \eta'(E_\nu(\ell))\leq C_7 \ell^{\beta_1}
    \qquad\text{for every }\ell\in[r,1).
$$
\end{proposition}

\begin{proof}
Define  $\mathcal{N}_5 = \mathcal{N}_1\cap\,\mathcal{N}_2$,
where $\mathcal{N}_1$ comes from
Corollary~\ref{cor:balance_measureorange} and $\mathcal{N}_2$ is defined in the beginning of
Section~\ref{s:thmA}. Fix $r\in(0,1)$ and $\nu'\in\mathcal{N}_5$ satisfying $\supp\nu'\subset B_{sr}(\supp\nu)$.

Suppose first that $\nu'$ is generic. Then its support
is Zariski-dense, and \cite{AEV2023}*{Proposition 5.2}
provides a unique $\nu'$-stationary measure $\eta'$,
which is also generic (in particular, non-atomic).
Since a maximal stationary measure exists
(see \cite{LLE}*{Proposition 6.7}), uniqueness implies
that $\eta'$ is maximal.
Moreover, for every $\ell\in[r,1)$, we have $\supp\nu'\subset B_{sr}(\supp\nu)\subset B_{s\ell}(\supp\nu)$. Corollary~\ref{cor:balance_measureorange} therefore gives
$$
    \eta'(E_\nu(\ell))\leq C_7 \ell^{\beta_1}
    \qquad\text{for every }\ell\in[r,1).
$$

Now let $\nu'$ be arbitrary.
By \cite{AEV2023}*{Proposition 5.3}, there exists a sequence
of generic measures $(\nu'_j)_{j\geq1}$ converging to
$\nu'$ in the $\fT$-topology.
Since $\mathcal{N}_5$ is open and $\nu'\in\mathcal{N}_5$, we have
$\nu'_j\in\mathcal{N}_5$ for all sufficiently large $j$.
Since convergence in the $\fT$-topology implies
Hausdorff convergence of supports, we also have $\supp\nu'_j\subset B_{sr}(\supp\nu)$ for all sufficiently large $j$. We may assume that
both conditions hold for every $j\in\NN$.

Let $\eta'_j$ be the unique $\nu'_j$-stationary measure.
By the generic case, each $\eta'_j$ is maximal and satisfies 
\begin{equation}\label{eq.balancedestimate}
    \eta'_j(E_\nu(t))\leq C_7 t^{\beta_1}
    \qquad\text{for every }j\in\NN
    \text{ and every }t\in[r,1).
\end{equation}

By compactness of $P$, after passing to a subsequence,
we may assume that
$$
    \eta'_j\overset{\ast}{\rightharpoonup}\eta'
    \qquad\text{for some }\eta'\in\Prob(P).
$$
By \cite{LLE}*{Proposition 5.9}, the measure $\eta'$
is $\nu'$-stationary.

We next show that $\eta'$ is maximal.
By the Hausdorff control of supports in the $\fT$-topology
and the choice of $\mathcal{N}_2$, the supports of $\nu'$ and of all
sufficiently large $\nu'_j$ are contained in the fixed
compact set $V\subset G$. Moreover, the function
$$
    \Phi(g,u)=\log\frac{\|gu\|}{\|u\|}
$$
is bounded and continuous on $V\times P$.
Since $\nu'_j\times\eta'_j\overset{\ast}{\rightharpoonup}\nu'\times\eta'$ and $\eta_j'$ is $\nu_j'$-maximal, we have
$$
    \int_{G\times P}\Phi\,d(\nu'\times\eta')
    =
    \lim_{j\to\infty}
    \int_{G\times P}\Phi\,d(\nu'_j\times\eta'_j)
    =
    \lim_{j\to\infty}\lambda_1(\nu'_j)
    =
    \lambda_1(\nu'),
$$ where the last equality comes from continuity of $\lambda_1$
in the $\fT$-topology (see \cite{AEV2023}*{Theorem B}). Thus $\eta'$ is maximal.

It remains to prove the estimate for this measure.
For $t\in(r,1)$, define
$$
    \tilde{E}_\nu(t)
    =\{u\in P:\ddp(u,E_\nu)<t\}.
$$
Since this set is open, the Portmanteau theorem and
\eqref{eq.balancedestimate} imply
$$
    \eta'(\tilde{E}_\nu(t))
    \leq\liminf_{j\to\infty}
         \eta'_j(\tilde{E}_\nu(t))
    \leq\liminf_{j\to\infty}
         \eta'_j(E_\nu(t))
    \leq C_7 t^{\beta_1}.
$$

Now fix any $\ell\in[r,1)$.
For every $t\in(\ell,1)$, we have
$E_\nu(\ell)\subset\tilde{E}_\nu(t)$, and hence
$$
    \eta'(E_\nu(\ell))\leq C_7 t^{\beta_1}.
$$
Letting $t\downarrow\ell$, we obtain
$$
    \eta'(E_\nu(\ell))\leq C_7 \ell^{\beta_1}.
$$
Since $\ell\in[r,1)$ was arbitrary, we conclude that
$\eta'$ satisfies the estimate for every
$\ell\in[r,1)$.
\end{proof}

As noted in Remark \ref{remark.uniquenessofmaximalstationary},
if $\nu\in\operatorname{Prob}_c(G)$ has a simple top Lyapunov
exponent, then it admits a unique maximal stationary measure,
which we denote by $\eta_\nu$.
We now combine the preceding estimates with
Lemma \ref{lemma:integralestimateofthedistanceiteratemarkovoperator}
to bound
$\dtheta(\cP_\nu^{*n}\eta_\nu,\cP_\nu^{*n}\eta')$
for a suitable maximal $\nu'$-stationary measure $\eta'$.
We split the integral into pairs of directions that both lie
outside a small neighborhood of the equator and pairs for which
at least one direction lies inside that neighborhood.
Proposition \ref{propIntegralEstimateDistance} controls the
first contribution, while
Proposition \ref{prop:estimateforallnotonlygeneric}
bounds the mass of the neighborhood under $\eta_\nu$ and $\eta'$,
thereby controlling the second contribution.

\begin{proposition}
\label{propEstimateDistanceBalancedMeasures}
    Let $\nu\in \operatorname{Prob}_c(G)$ be such that $\lambda_1(\nu)$ is simple and $\dim E_\nu=1$. Consider $\gamma \in (0, \gamma_0/3]$. There exist constants $C_8 = C_8(\nu, \gamma)>0$, $\tau_3 = \tau_3(\nu,\gamma) \in(0,1)$ and $\theta_3=\theta_3(\nu, \gamma)>0$
such that the following holds. For every integer $n\geq1$ and every
$\nu'\in \mathcal{N}_5$ satisfying $\supp\nu'\subset B_{se^{-n\gamma}}(\supp\nu)$, there exists a maximal $\nu'$-stationary measure $\eta'$
such that
$$
\dtheta(\cP_\nu^{*n}\eta_\nu,\cP_\nu^{*n}\eta')
    \leq C_8\tau_3^{n\theta}
$$ for every $\theta\in(0,\theta_3]$.
\end{proposition}

\begin{proof}
Let $\mathcal{N}_5$, $C_7>0$ and $\beta_1>0$ be as in
Proposition~\ref{prop:estimateforallnotonlygeneric}.

Since $\nu\in\mathcal{N}_5$ and
$\supp\nu\subset B_{st}(\supp\nu)$ for every $t\in(0,1)$,
that proposition and uniqueness of the maximal
$\nu$-stationary measure imply that, for every $t\in(0,1)$,
$$
    \eta_\nu(E_\nu(t))\leq C_7 t^{\beta_1}.    
$$

Fix an integer $n\geq1$ and $\nu'\in\mathcal{N}_5$
such that $\supp\nu'\subset B_{se^{-n\gamma}}(\supp\nu)$. Applying Proposition~\ref{prop:estimateforallnotonlygeneric}
with $r=e^{-n\gamma}$, choose a maximal
$\nu'$-stationary measure $\eta'$ satisfying
$$
    \eta'(E_\nu(e^{-n\gamma}))
    \leq C_7 e^{-n\gamma\beta_1}.
$$ In particular,
\begin{equation}\label{eq:equator-mass-bounds}
    \eta_\nu(E_\nu(e^{-n\gamma}))
    +\eta'(E_\nu(e^{-n\gamma}))
    \leq 2C_7 e^{-n\gamma\beta_1}.
\end{equation}

By Lemma~\ref{lemma:integralestimateofthedistanceiteratemarkovoperator},
for every $\theta\in(0,1]$,
$$
\dtheta(\cP_\nu^{*n}\eta_\nu,\cP_\nu^{*n}\eta')
    \leq
    \int_{G\times P\times P}
    \ddp(gu,gu')^\theta
    \,d\nu^{(n)}(g)\,d\eta_\nu(u)\,d\eta'(u').
$$ Set
$$
    E_n=E_\nu(e^{-n\gamma}),
    \qquad
    K_{n,\gamma}=(P\setminus E_n)\times(P\setminus E_n).
$$
We estimate this integral separately over
$G\times K_{n,\gamma}$ and $G\times K_{n,\gamma}^C$.

\begin{figure}[h]
     \centering
     \includegraphics[width=0.4\linewidth]{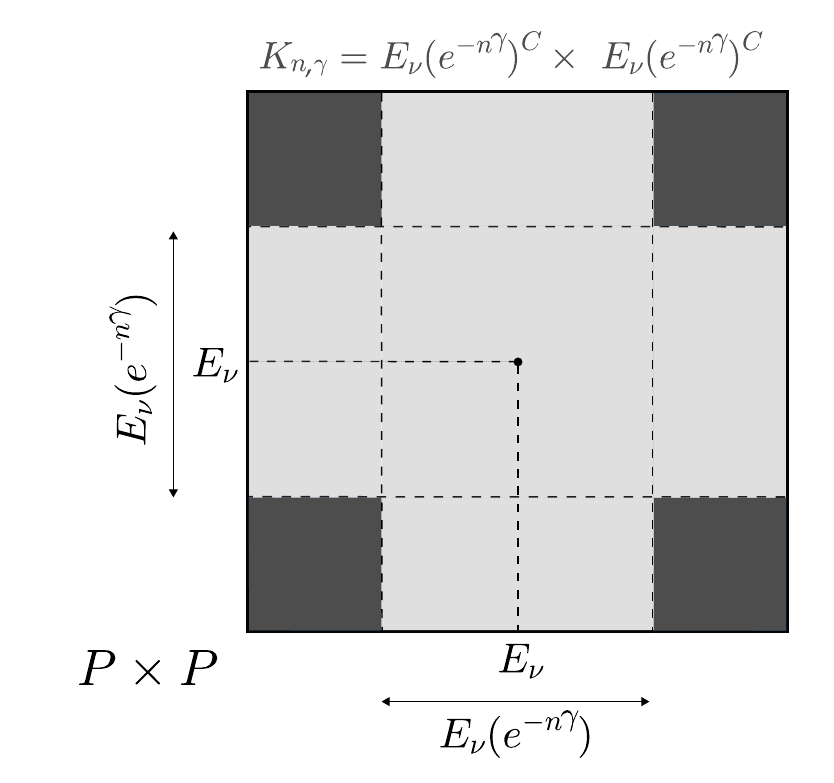}
     \caption{Schematic representation of the decomposition of
$P\times P$. The dark region represents
$K_{n,\gamma}$ and the light region represents its complement $K_{n,\gamma}^C$.}
     \label{fig:partitionKgamma}
 \end{figure}

For $(u,u')\in K_{n,\gamma}$,
Proposition~\ref{propIntegralEstimateDistance} gives
$$
    \int_G\ddp(gu,gu')^\theta\,d\nu^{(n)}(g)
    \leq C_6\tau_2^{n\theta}
$$
whenever $0<\theta\leq\theta_2$.
Integrating with respect to the product $\eta_\nu\times\eta'$,
we obtain
$$
    \int_{G\times K_{n,\gamma}}
    \ddp(gu,gu')^\theta
    \,d\nu^{(n)}(g)\,d\eta_\nu(u)\,d\eta'(u')
    \leq C_6\tau_2^{n\theta}.
$$ On the other hand,
$$
    K_{n,\gamma}^C=(E_n\times P)\cup(P\times E_n).
$$ Since $\ddp(gu,gu')^\theta\leq1$ and $\nu^{(n)}$
is a probability measure, \eqref{eq:equator-mass-bounds}
yields
$$
\begin{aligned}
    \int_{G\times K_{n,\gamma}^C}
    \ddp(gu,gu')^\theta
    \,d\nu^{(n)}(g)\,d\eta_\nu(u)\,d\eta'(u')
    &\leq
    (\eta_\nu\times\eta')(K_{n,\gamma}^C)\\
    &\leq\eta_\nu(E_n)+\eta'(E_n)\\
    &\leq2C_7e^{-n\gamma\beta_1}.
\end{aligned}
$$

Combining these estimates, we conclude that
$$
    \dtheta(\cP_\nu^{*n}\eta_\nu,\cP_\nu^{*n}\eta')
    \leq C_6\tau_2^{n\theta}
         +2C_7 e^{-n\gamma\beta_1}.
$$

Finally, set $ \theta_3=\min\{\theta_2,\gamma\beta_1\}, \tau_3=\max\{\tau_2,e^{-1}\},$ and $C_8= C_6+2C_7$. For every $\theta\in(0,\theta_3]$, we have $e^{-n\gamma\beta_1}\leq e^{-n\theta}
    \leq\tau_3^{n\theta}$ and $\tau_2^{n\theta}\leq\tau_3^{n\theta}$. Therefore,
$$
    \dtheta(\cP_\nu^{*n}\eta_\nu,\cP_\nu^{*n}\eta')
    \leq C_8\tau_3^{n\theta}
$$ for every $\theta\in(0,\theta_3]$.
\end{proof}

Our ultimate goal is to prove the Hölder continuity of the first Lyapunov exponent using Proposition \ref{prop.desmarkovprincipal}. To do so, we must first revisit Proposition \ref{prop.essentialinequalitydistancespartone} and, by applying the preceding estimate, establish an inequality of the form $d(\eta_\nu,\eta')\le C\delta_{\fT,\theta}(\nu,\nu')^\beta$. The following proposition provides such an inequality.

Let $\theta_3 = \theta_3(\nu) \in (0,1)$ be as in the previous proposition.

\begin{proposition}
\label{proposition.finalinequalitydistance}
Let $\nu\in\Prob_c(G)$ be such that $\lambda_1(\nu)$
is simple and $\dim E_\nu=1$.   For every
$\theta\in(0,\theta_3]$, there exist a neighborhood
$\mathcal{N}_6=\mathcal{N}_6(\nu,\theta)\subset \operatorname{Prob}_c(G)$ of $\nu$ and constants
$C_9=C_9(\nu,\theta)>0$ and
$\beta_2=\beta_2(\nu,\theta)\in(0,1)$ with the following property:
for every $\nu'\in\mathcal{N}_6$, there exists a maximal
$\nu'$-stationary measure $\eta'$ satisfying
\begin{equation}
\label{eq.distanceestimatewithequator}
    \dtheta(\eta_\nu,\eta')
    \leq C_9\delta_{\fT,\theta}(\nu,\nu')^{\beta_2}.
\end{equation}
\end{proposition}

\begin{proof}
Let $\mathcal{N}_5$, $C_8>0$, $\tau_3\in(0,1)$
and $\theta_3>0$ be as in
Proposition~\ref{propEstimateDistanceBalancedMeasures}.
Let $\mathcal{N}_4$, $C_2$ and
$C_3$ be as in
Proposition~\ref{prop.essentialinequalitydistancespartone}. Fix $\theta\in(0,\theta_3]$. Combining these two propositions, we obtain the following: for every integer $n\geq1$ and every $\nu'\in\mathcal{N}_4\cap \mathcal{N}_5$ satisfying $\supp\nu'\subset B_{se^{-n\gamma}}(\supp\nu),$ there exists a maximal $\nu'$-stationary measure $\eta'$
such that
\begin{equation}\label{eq:combined-stationary-estimate}
    \dtheta(\eta_\nu,\eta')
    \leq C_2 C_3^n\delta_\theta(\nu,\nu')
    +C_8\tau_3^{n\theta}.
\end{equation}

Choose $\hat\delta=\hat\delta(\theta)$ with $0<\hat\delta< \min\{1,s\}$ such that
$$
    \{\nu'\in\Prob_c(G):
      \delta_{\fT,\theta}(\nu,\nu')<\hat\delta\}
    \subset \mathcal{N}_4 \cap\, \mathcal{N}_5.
$$ Set $ \zeta=\min\{C_3^{-2},e^{-\gamma}\}\in(0,1)$ and define the neighborhood
$$
    \mathcal{N}_6=
    \{\nu'\in\Prob_c(G):
      \delta_{\fT,\theta}(\nu,\nu')<\hat\delta\zeta\}.
$$ Fix $\nu'\in\mathcal{N}_6$ and write $D=\delta_{\fT,\theta}(\nu,\nu').$ If $D=0$, then $\nu'=\nu$, and the conclusion follows
by taking $\eta'=\eta_\nu$.
From now on, assume that $D>0$.

The idea is to choose $n=n_0$ depending on $D$ so that
both terms on the right-hand side of
\eqref{eq:combined-stationary-estimate}
are bounded by positive powers of $D$.
At the same time, this choice must ensure that
$\supp\nu'\subset B_{se^{-n_0\gamma}}(\supp\nu)$,
so that the estimate applies.
Since $\delta_H(\supp\nu,\supp\nu')\leq D$, it suffices
to choose $n_0$ so that $D<se^{-n_0\gamma}$.

With this in mind, let $n_0$ be the largest natural number such that
    \begin{equation}
        \label{defnzero}\delta_{\fT,\theta}(\nu,\nu') = D\le \hat{\delta} \zeta^{n_0}.
    \end{equation} Then,
$$\delta_H(\supp\nu,\supp\nu')\leq D\leq\hat\delta\zeta^{n_0}\leq\hat\delta e^{-n_0\gamma}<se^{-n_0\gamma}.
$$
Thus $\supp\nu'\subset B_{se^{-n_0\gamma}}(\supp\nu)$. We may therefore apply
\eqref{eq:combined-stationary-estimate} with $n=n_0$
and choose a maximal $\nu'$-stationary measure $\eta'$
satisfying that estimate. First, notice that
    $$
            n_0 \le \frac{\log{\frac{\delta_{\fT,\theta}(\nu,\nu')}{\hat{\delta}}}}{\log \zeta}.
    $$ Define $x_0:=\frac{\log {(\delta_{\fT,\theta}(\nu,\nu')/\hat{\delta})}}{\log \zeta}\in \RR$. By the maximality of $n_0$, $\hat\delta\zeta^{n_0+1} <D\leq\hat\delta\zeta^{n_0}$, and hence $x_0-1<n_0\leq x_0$. Now we can bound 
    \begin{align}
        C_8\tau_3^{\theta n_0} &\le C_8\tau_3^{\theta (x_0-1)} \\
        & = \frac{C_8}{\tau_3^\theta} \exp{\left(\theta\frac{\log{\tau_3}}{\log\zeta}\log\frac{\delta_{\fT,\theta}(\nu,\nu')}{\hat{\delta}}\right)}\\
        &= C_8 \tau_3^{-\theta}\hat{\delta}^{-\theta\frac{\log \tau_3}{\log \zeta}}\delta_{\fT,\theta}(\nu,\nu')^{\theta\frac{\log \tau_3}{\log \zeta}},
    \end{align} Moreover, by the definition of $\zeta$ and equation \eqref{defnzero}, we get
    \begin{align}
        C_2 C_3^{n_0} \delta_\theta(\nu,\nu') 
        &\le C_2 \zeta^{-n_0/2} \delta_{\fT,\theta}(\nu,\nu') \\
        &\le C_2 \hat{\delta}^{1/2} \delta_{\fT,\theta} (\nu,\nu')^{1/2}.
    \end{align} To conclude, take $\beta_2\coloneqq \min \{\theta\log \tau_3/\log \zeta, 1/2\}$ and $C_9\coloneqq C_8 \tau_3^{-\theta}\hat{\delta}^{-\theta\frac{\log \tau_3}{\log \zeta}} + C_2 \hat{\delta}^{1/2}$. Note that the neighborhood $\mathcal{N}_6$ and the constants $C_9,\beta_2$ may depend on $\theta$, but are independent of $\nu'$.

\end{proof}

\subsubsection{Conclusion of the proof of the one-dimensional case}
\label{sectionholder2}

Let $\nu\in\Prob_c(G)$ be such that $\lambda_1(\nu)$
is simple and $\dim E_\nu=1$, and let $\eta_\nu$
be its unique maximal stationary measure.
Set $\theta_0=\theta_3$, where $\theta_3$
is given by Proposition~\ref{proposition.finalinequalitydistance},
and fix $\theta\in(0,\theta_0]$.

By Proposition~\ref{proposition.finalinequalitydistance},
there exist a neighborhood $\mathcal{N}_6$ of $\nu$
and constants $C_9>0$ and $\beta_2\in(0,1)$ such that,
for every $\nu'\in\mathcal{N}_6$, there exists a maximal
$\nu'$-stationary measure $\eta_{\nu'}$ satisfying
$$
    \dtheta(\eta_\nu,\eta_{\nu'})
    \leq C_9\delta_{\fT,\theta}(\nu,\nu')^{\beta_2}.
$$

Now, let  $\mathcal{N}_3$ and $C_1>0$ be as in
Proposition~\ref{prop.desmarkovprincipal}, and set $\cN_0= \mathcal{N}_3 \cap\mathcal{N}_6$. Fix $\nu'\in\cN_0$ and choose a maximal
$\nu'$-stationary measure $\eta_{\nu'}$
satisfying the estimate above.
By Proposition~\ref{prop.desmarkovprincipal},
$$
\begin{aligned}
    |\lambda_1(\nu)-\lambda_1(\nu')|
    &\leq C_1\delta_\theta(\nu,\nu')
         +C_1\dtheta(\eta_\nu,\eta_{\nu'})\\
    &\leq C_1\delta_{\fT,\theta}(\nu,\nu')
         +C_1C_9\delta_{\fT,\theta}(\nu,\nu')^{\beta_2}\\
    &\leq C_1(1+C_9)
         \delta_{\fT,\theta}(\nu,\nu')^{\beta_2},
\end{aligned}
$$
where the last inequality follows from
$0\leq\delta_{\fT,\theta}(\nu,\nu')<1$ and
$0<\beta_2<1$.

Thus, setting $C_0=C_1(1+C_9)$ and $\beta_0=\beta_2$,
we obtain
$$
    |\lambda_1(\nu)-\lambda_1(\nu')|
    \leq C_0\delta_{\fT,\theta}(\nu,\nu')^{\beta_0}
$$ for every $\nu'\in\cN_0$. Note that the neighborhood $\cN_0$ and the constants $C_0,\beta_0$
may depend on $\theta$, but are independent of $\nu'$.
This finishes the proof of
Theorem~\ref{theorem.HolderCont}.

%% file: body/4ProvaThmB.tex
\subsection{Outline of the proof of Theorem B}
The proof combines local expansion with a coupling argument.
Since $E$ is one-dimensional, the vertical angle factor equals one,
so the energy can be expressed solely in terms of projective distance.
AEV's directional expansion estimate gives local pair expansion and
annular escape, with probability bounds uniform in the transverse
direction. At each scale $\varepsilon$, we construct an adapted
localized Markov operator on $X=E(\varepsilon)$ that preserves the
normalized restriction $\mu$ of $\eta'$ to $X$.
We then couple two copies of this operator, using the same random
matrix near the equator and recoupling pairs farther away to prevent
their images from both entering a smaller neighborhood.
The resulting operator strictly decreases a cutoff logarithmic energy
 in expectation on $E(q\varepsilon)^2$, while allowing only a controlled
increase elsewhere. Non-atomicity ensures that the energy has a finite
minimum among the self-couplings of $\mu$. Since the coupled operator preserves the marginals, applying it to a minimizer cannot lower the
total energy. Integrating the drift estimates therefore bounds the
minimizer's mass on $E(q\varepsilon)^2$; its marginal identities then
give the desired factor $2/3$. The parameters depend only on $\nu$,
and all auxiliary radii are fixed multiples of $\varepsilon$, making
the estimate uniform in scale.

Let $\nu\in \Prob_c(G)$ and let $E=E_\nu$ be the $\nu$-equator. Suppose that $\dim E_\nu=1$. Write $e=\PP(E)$ for the corresponding projective point.
\subsection{Vertical angle}
For distinct projective points $x,x'$, let $\cY(x,x')\subset V$ be the
two-dimensional subspace spanned by their representing lines. Define
\begin{equation}\label{eq:vertical-angle}
 \VA_1(x,x')=
 \sup_{0\neq u\in\cY(x,x')}
 \frac{\norm{\pi_{E^\perp}u}}{\norm{u}},
\end{equation}
where $\pi_{E^\perp}$ is orthogonal projection onto $E^\perp$.
\begin{lemma}\label{lem:noStabilization}
If $\dim E=1$, then $\VA_1(x,x')=1$ for every $x\neq x'$.
\end{lemma}

\begin{proof}
The dimension inequality gives
$$
 \dim\bigl(\cY(x,x')\cap E^\perp\bigr)
 \geq 2+(d-1)-d=1.
$$
Choose $0\neq u$ in this intersection. Then
$\pi_{E^\perp}u=u$, so the quotient in
\eqref{eq:vertical-angle} equals $1$.
\end{proof}

This removes the vertical angle stabilization condition in \cite{AEV2023}*{Section~9}.
\subsection{Fixing the parameters}
Fix a compact neighborhood $W\subset G$ of $K:=\supp\nu$.
By compactness and smoothness of the projective action, choose $A = A(\nu)\geq1$,
depending only on $\nu$, such that
\begin{equation}\label{eq:compactNbhdEstimate}
-An + \log d(x, x') \leq \log d(gx, gx') \leq An + \log d(x, x')
\end{equation}
 Moreover, let $B = B(\nu) > 0$ be defined by 
 \begin{equation}\label{eq:normEstimateB}
 	B = \sup \set{\log\norm{f} + \log\norm{f^{-1}}: \, f\in W} + 2
 \end{equation}
fro every product $g=h_n\cdots h_1$, with $h_i\in W$.

Take $\kappa_0=\kappa_0(\nu)>0$ from
\cite{AEV2023}*{Proposition~5.5}. Let $\kappa_1'''(\nu)$ and $C_1'''(\nu)$ be the drift constants in \cite{AEV2023}*{Proposition~11.2}. They depend only on $\nu$, independently of $\delta$, $n$, and the radii.
Define 
\begin{equation}\label{eq:aev-constants}
 \kappa_*:=\frac12\kappa_1'''(\nu),
 \qquad
 C_*:=\max\{C_1'''(\nu),\log2,A+\kappa_*\},
 \qquad
 C_1'(\nu)=B+\gamma_1 A.
\end{equation}
Here the last identity is \cite{AEV2023}*{Equation (8.34)}.
The smaller normalization for $\kappa_*$ agrees with the proof of 
\cite{AEV2023}*{Proposition~9.3} and is recalled in \cite{AEV2023}*{Section~10}. Equations~(8.19) and~(8.26) of \cite{AEV2023} give
\begin{equation}\label{eq:kappa-comparison}
 0<\kappa_*\leq\frac{\kappa_0}{8},
\end{equation}
because $\widehat\kappa_1=\kappa_0/2$, $\gamma_1\leq1$, and
$\kappa_1'\leq\gamma_1\widehat\kappa_1/2$.

Choose
\begin{equation}\label{eq:choiceDelta}
 0<\delta<\min\left\{\frac14,\frac{\kappa_*}{20C_*}\right\},
\end{equation}
and choose an integer $n\geq N_0(\nu,\delta/2)$ large enough such that
\begin{align}
 e^{\kappa_0n/4}&\geq16,\label{eq:n-chart}\\
 6e^{-\kappa_0n/2}&<1,\label{eq:n-radii}\\
 (\kappa_*-C_*\delta)n&>9C_*(1+\delta n).
 \label{eq:n-drift}
\end{align}
The second condition is possible because it is equivalent to
$$
 (\kappa_*-10C_*\delta)n>9C_*,
$$
and \eqref{eq:choiceDelta} implies
$\kappa_*-10C_*\delta>\kappa_*/2>0$.

Set
\begin{equation}\label{eq:margulisConstant}
 \kappa_A=(\kappa_*-C_*\delta)n,
 \qquad
 \kappa_B=C_*(1+\delta n).
\end{equation}
These constants are positive and satisfy
\begin{equation}\label{eq:relationKappaAB}
 \kappa_A>9\kappa_B,
 \qquad
 \theta:=\frac{\kappa_B}{\kappa_A+\kappa_B}<\frac1{10}.
\end{equation}
Once $\delta$ and $n$ have been fixed, $\kappa_A$ and $\kappa_B$ are
independent of the working radius $\varepsilon$.

\subsection{Radii and support perturbations}
Let $L_n=e^{An}$ and choose
\begin{equation}\label{eq:radii-ratios}
 \widetilde a=\frac1{8L_n},
 \qquad
 a_1=\frac{\widetilde a}{8L_n}.
\end{equation}
Then $0<a_1<\widetilde a<1$ and
\begin{equation}\label{eq:radii-slack}
 2L_n\widetilde a<\frac12,
 \qquad
 2L_na_1<\frac{\widetilde a}{2}.
\end{equation}
For any $0<\varepsilon\leq1$, define
\begin{equation}\label{eq:radii-prime}
 \varepsilon'=a_1\varepsilon,
 \qquad
 \widetilde\varepsilon=\widetilde a\varepsilon.
\end{equation}
%
\begin{lemma}\label{lem:support-perturbation}
There exists $s = s(\nu)>0$, independent of $\varepsilon$, such
that if $\supp\nu'\subset B_{s\varepsilon}(\supp \nu)$, then
every product $g'=h_n'\cdots h_1'$ with $h_i'\in\supp\nu'$ satisfies
\begin{align}
 g'E(2\widetilde\varepsilon)
 &\subset E(\varepsilon/2),\label{eq:forward-inclusion}\\
 (g')^{-1}E(2\varepsilon')
 &\subset E(\widetilde\varepsilon/2).
 \label{eq:inverse-inclusion}
\end{align}
\end{lemma}

\begin{proof}
 First choose $s$ small enough that
$B_s(\supp\nu)\subset W$.
For each $h_i'$ choose $h_i\in K$ with
$\norm{h_i'-h_i}<s\varepsilon$ and set
$g=h_n\cdots h_1$.
Telescoping the product difference, followed by the identity
$$
 (g')^{-1}-g^{-1}=(g')^{-1}(g-g')g^{-1},
$$
gives uniform bounds on $\norm{g'-g}$ and
$\norm{(g')^{-1}-g^{-1}}$ proportional to $s\varepsilon$.
The norms and inverse norms of these products are bounded on the
relevant compact sets. Hence there are constants $C_n,C_n^->0$ such that
\begin{align*}
 \sup_{x\in\PP(V)}d(g'x,gx)&\leq C_ns\varepsilon,\\
 \sup_{x\in\PP(V)}d\bigl((g')^{-1}x,g^{-1}x\bigr)
 &\leq C_n^-s\varepsilon.
\end{align*}   
The unperturbed product $g$ fixes $e$.
Thus, for $x\in E(2\widetilde\varepsilon)$,
$$
 d(g'x,e)\leq C_ns\varepsilon+L_nd(x,e)
 \leq\bigl(C_ns +2L_n\widetilde a\bigr)\varepsilon.
$$
Similarly, for $x\in E(2\varepsilon')$,
$$
 d\bigl((g')^{-1}x,e\bigr)
 \leq\bigl(C_n^-s +2L_na_1\bigr)\varepsilon.
$$
In view of \eqref{eq:radii-slack}, we may further decrease
$s$ so that
$$
 C_ns +2L_n\widetilde a<\frac12,
 \qquad
 C_n^-s+2L_na_1<\frac{\widetilde a}{2}.
$$
The two inclusions follow.
\end{proof}

Next set
\begin{equation}\label{eq:radii-double}
 \varepsilon''=3\varepsilon'e^{-\kappa_0n/2},
 \qquad
 c_\Omega=a_1\bigl(1-6e^{-\kappa_0n/2}\bigr)>0.
\end{equation}
By \eqref{eq:n-radii}, $2\varepsilon''<\varepsilon'<\varepsilon$.
Define
\begin{equation}\label{eq:omega}
 \Omega=(\varepsilon'-2\varepsilon'')^{-1}
       =(c_\Omega\varepsilon)^{-1}
\end{equation}
and
\begin{equation}\label{eq:radii-triple}
 \varepsilon'''=
 \min\left\{\varepsilon'',\frac12\Omega^{-1}e^{-\kappa_*n}\right\}.
\end{equation}
It follows that $\varepsilon'''=q\varepsilon$, where
\begin{equation}\label{eq:q}
 q=\min\left\{
 3a_1e^{-\kappa_0n/2},
 \frac{a_1}{2}\bigl(1-6e^{-\kappa_0n/2}\bigr)e^{-\kappa_*n}
 \right\}\in(0,1).
\end{equation}
In particular, $q$ is independent of $\varepsilon$.

\subsection{Uniform expansion and annular escape}

By decreasing $s$ if necessary, it follows from \cite{AEV2023}*{Corrollary~5.7}:  
\begin{lemma}\label{lem:local-expansion}
	For the fixed $\delta$ and $n$, there exist $\rho \in (0, 1/2)$, and a $\fT$-neighborhood $\cN'$ of $\nu$, all independent of $\varepsilon$, with the following property.
If $0<\varepsilon\leq\rho$, $\nu'\in\cN'$, and
$\supp\nu'\subset B_{s\varepsilon}(K)$, then
\begin{align}
 \nu'^{(n)}
 \{g:d(gx,gx')\geq e^{\kappa_0n/2}d(x,x')\}
 &>1-\delta
 &&(x\neq x'\text{ in }E(\varepsilon)),\label{eq:pair-expansion}\\
 \nu'^{(n)}\{g:gx\notin E(\varepsilon')\}
 &>1-\delta
 &&\left(\varepsilon''\leq d(x,e)\leq\varepsilon\right).
 \label{eq:annular-escape}
\end{align}
\end{lemma}

\noindent\emph{Choice of the uniform constants.}
With $\delta$ and $n$ fixed, define the constants of Theorem~B by
\begin{equation}\label{eq:epsilon-zero}
 \varepsilon_0=\min\left\{\rho,\frac1{2c_\Omega}\right\}>0,
 \qquad
 \cN_1=\cN'.
\end{equation}
For every $0<\varepsilon\leq\varepsilon_0$, we have
$\varepsilon\leq\rho$, so the local expansion and escape estimates
apply, and $\Omega=(c_\Omega\varepsilon)^{-1}\geq2$.
The choice of $s$ makes both Lemma~\ref{lem:support-perturbation} and
Lemma~\ref{lem:local-expansion} applicable under the support condition
of Theorem~B. These constants are independent of $\varepsilon$ and will be used throughout the remainder of the proof.

\subsection{The adapted Markov operator}
Now let $X$ be a measurable subset of $P = \PP\RR^d$ and $\zeta$ be a measure on $G = \GL(d)$. The $\zeta$-core of $X$ is the set
$$
 \cX_\zeta X=\{x\in X:gx\in X\text{ for every }g\in\supp\zeta\},
$$
and we also define the strong $\zeta$-core to be
$$
\cX_\zeta^\#X=\{x\in \cX_\zeta X:g^{-1}x\in \cX_\zeta X\text{ for every }g\in\supp\zeta\}.
$$
Both sets are closed. Write $\zeta_x$ for the push-forward of $\zeta$
under $g\mapsto gx$ viewed as a measure on $\PP(\RR^d)$. Let $\cP$ be a Markov operator on $X$, with transition probabilities $\set{P(x,): x\in X}$. We say that $\cP$ is \emph{adapted to $\zeta$ on $X$} if 
\begin{enumerate}
\item $P(x,\cdot)=\zeta_x$ for every $x\in \cX_\zeta X $;
\item $P(x,\cX_\zeta^\# X)=0$ for every $x$ in the border region $\partial_\zeta X = X\setminus \cX_\zeta X$.
\end{enumerate}
Thus the operator retains the original dynamics on the core and
prevents transitions from the border into the strong $\zeta$-core $\cX_\zeta^\# X$.

Let $\varepsilon$, $\nu'$, and $\eta'$ satisfy the hypotheses of
Theorem~B, and write
$$
 X=E(\varepsilon),\qquad \zeta=\nu'^{(n)},\qquad M=\eta'(X).
$$

If $M=0$ the mass estimate is immediate; assume $M>0$ and set
$\mu=M^{-1}\eta'|_X$. The measure $\eta'$ is also $\zeta$-stationary.


We use the localization construction of \cite{AEV2023}*{Proposition~6.14, equation~(6.17)}; see also
\cite{V2020Dis}*{Appendix~A.3}.
Define the exit probability and the entrance measure by
$$
 b(x)=\zeta\{g:gx\notin X\},\qquad
 \alpha(D)=\int_{X^c}\zeta\{g:gy\in D\}\dd\eta'(y)
 \quad(D\subset X\text{ measurable}).
$$
Stationarity gives
$$
 I:=\alpha(X)=\int_X b(x)\dd\eta'(x).
$$
If $J>0$, put $R=\alpha/J$. If $J=0$, let $R$ be a point mass at any
point of $X\setminus E(2\varepsilon')$. In either case define
\begin{equation}\label{eq:localized-kernel}
 P_\varepsilon(x,D)
 =\zeta\{g:gx\in D\}+b(x)R(D),
 \qquad x\in X,\quad D\subset X.
\end{equation}

\begin{lemma}\label{lem:adapted-localization}
The kernel $P_\varepsilon$ defined in \eqref{eq:localized-kernel} is a measurable Markov kernel on $X$, preserves
$\mu=M^{-1}\eta'|_X$, and is adapted to $\zeta$ in the sense defined
above. Moreover,
\begin{equation}\label{eq:core-inclusions}
 E(2\widetilde\varepsilon)\subset \cX_\zeta X,\qquad
 P_\varepsilon(x,E(2\varepsilon'))=0\quad(x\in \partial_\zeta X).
\end{equation}
\end{lemma}

\begin{proof}
The defining formula gives a nonnegative, kernel. Since
$R(X)=1$, its total mass is
$$
 P_\varepsilon(x,X)=1-b(x)+b(x)=1.
$$
When $I>0$,
$$
 \int_X P_\varepsilon(x,D)\dd\eta'(x)
 =\int_X\zeta\{g:gx\in D\}\dd\eta'(x)+\alpha(D)
 =\eta'(D).
$$
When $I=0$, the term $b(x) R(D)$ integrates to zero and the same
identity follows from $\alpha=0$. Thus $P_\varepsilon$ leaves $\mu$
invariant in both cases.
Thus $P_\varepsilon$ leaves $\mu$ invariant.

We now verify that $P_\epsilon$ is adapted. For $x \in \cX_\zeta X$ we have $b(x) = 0$, hence $P_\epsilon(x,) = \zeta_x$. If $x\notin \cX_\zeta X$, an image $gx$ cannot belong to $\cX_\zeta^\#X$, because $gx\in \cX_\zeta^\#X$ would imply $x=g^{-1}(gx)\in \cX_\zeta X$.
This observation applies in particular to every $x\in X^c$,
so $\alpha(\cX_\zeta^\# X)=0$. Therefore $R(\cX_\zeta^\# X)=0$ when $I>0$.
For $I=0$ and $\cX_\zeta X\neq X$, the choice $z\in \partial_\zeta X$ gives the same conclusion. When $\cX_\zeta X=X$ there is no border. Thus $P_\varepsilon(x,\cX_\zeta^\# X)=0$ for every $x\in \partial_\zeta X$,
which proves adaptation.

The support inclusions \eqref{eq:forward-inclusion}--\eqref{eq:inverse-inclusion} give
\begin{equation}\label{eq:core-inclusions}
 E(2\widetilde\varepsilon)\subset \cX_\zeta X,\qquad
 P_\varepsilon(x,E(2\varepsilon'))=0\quad(x\in \partial_\zeta X).
\end{equation}
Indeed, the forward inclusion gives
$E(2\widetilde\varepsilon)\subset \cX_\zeta X$. Since $a_1<\widetilde a$, we also have $E(2\varepsilon')\subset \cX_\zeta X$. For $y\in E(2\varepsilon')$, the inverse inclusion gives $g^{-1}y\in E(\widetilde\varepsilon/2)\subset \cX_\zeta X$ for every $g\in\supp\zeta$. Consequently $E(2\varepsilon')\subset \cX_\zeta^\# X$, and the second assertion of \eqref{eq:core-inclusions} follows from adaptation.
\end{proof}

\subsection{The Margulis function}
Let $X=E(\varepsilon)$ and let
$$
 C_\varepsilon=\{x\in X:d(x,e)<2\varepsilon''\}, \qquad e = \PP(E). 
$$
Let
\begin{equation}\label{eq:psi}
 \Psi_\varepsilon(x,x')=
 \begin{cases}
 \log\bigl(\Omega+d(x,x')^{-1}\bigr),
 &x\neq x'\text{ and }(x\in C_\varepsilon\text{ or }x'\in C_\varepsilon),\\[1mm]
 +\infty,&x=x'\in C_\varepsilon,\\[1mm]
 \log\Omega,&\text{otherwise}.
 \end{cases}
\end{equation}

If $\varepsilon\leq(2c_\Omega)^{-1}$, then $\Omega\geq2$, so
$\Psi_\varepsilon$ is nonnegative. 
$\Psi_\varepsilon $ is measurable; it is bounded outside every
neighborhood of the diagonal: If $U$ is a neighborhood of the diagonal, compactness gives: $r_U = \min_{X^2\setminus U}d(x,x')>0$, then $\Psi \le \log(\Omega + r_U^{-1})$ on $X^2\setminus U$. Choose any $L >\log\Omega$. On $\set{\Psi \ge L}$, every point not in the diagonal has a coordinate in $C =E(2\varepsilon'') $ and therefore 
 $$
 \Psi_\varepsilon = \log(\Omega + d^{-1}).
 $$
this is continuous away from diagonal and tends to $+\infty$ along pairs approaching the diagonal. Hence $\Psi_{|_{\set{\Psi\ge L}}}$ is lower semicontinuous.
Two elementary calculations explain the choices of $\Omega$ and
$\varepsilon'''$. If one coordinate lies in $C_\varepsilon$ and the other
lies outside $E(\varepsilon')$, the triangle inequality gives
$$
 d(x,x')\geq\varepsilon'-2\varepsilon''=\Omega^{-1}.
$$
Consequently, for every pair outside $E(\varepsilon')^2$,
\begin{equation}\label{eq:cutoff-bound}
 \Psi_\varepsilon(x,x')\leq\log\Omega+\log2.
\end{equation}
If $x\neq x'$ are in $E(\varepsilon''')$, then
\begin{equation}\label{eq:inner-threshold}
 d(x,x')\leq2\varepsilon''',
 \qquad
 d(x,x')^{-1}\geq\Omega e^{\kappa_*n}.
\end{equation}
In particular, the elementary inequality
\begin{equation}\label{eq:log-contraction}
 \log(\Omega+cz)\leq\log(\Omega+z)+\frac12\log c,
 \qquad 0<c\leq1,\quad z\geq\frac{\Omega}{\sqrt c},
\end{equation}
applies with $c=e^{-2\kappa_*n}$ at the inner scale.
\subsection{Recoupling}
For $x,x'\in \cX_\zeta X$ with $d(x,e),d(x',e)>\varepsilon''$, put
$$
 B_x=\{g:gx\in E(\varepsilon')\},\qquad
 B_{x'}=\{g:gx'\in E(\varepsilon')\}.
$$
By \eqref{eq:annular-escape},
$p:=\zeta(B_x)<\delta$ and $p':=\zeta(B_{x'})<\delta$, so
$p+p'<1$. Thus \cite{AEV2023}*{Lemma~6.1} applies with both measures equal
to $\zeta$ on $G$ and with $A=B_x$, $A'=B_{x'}$.
Write
$$
 \zeta_b=\zeta|_{B_x},\quad \zeta_g=\zeta|_{B_x^c},\qquad
 \zeta_b'=\zeta|_{B_{x'}},\quad \zeta_g'=\zeta|_{B_{x'}^c}.
$$
The explicit coupling in \cite{AEV2023}*{equation~(6.3)} is
\begin{equation}\label{eq:directional-recoupling}
 \Gamma_{x,x'}=
 \frac{\zeta_b\times\zeta_g'}{1-p'}
 +\frac{\zeta_g\times\zeta_b'}{1-p}
 +\frac{1-p-p'}{(1-p)(1-p')}\,
       \zeta_g\times\zeta_g'
\end{equation}
It has both marginals equal to $\zeta$ and
$\Gamma_{x,x'}(B_x\times B_{x'})=0$. Following the recoupling principle in \cite{AEV2023}*{Section~11.1},
define a measurable probability kernel $\Theta_{x,x'}$ on $X^2$
by three cases:
\begin{enumerate}
\item If $x,x'\in \cX_\zeta X$ and
$\min\{d(x,e),d(x',e)\}\leq\varepsilon''$, let $\Theta_{x,x'}$ be the
push-forward of $\zeta$ under $g\mapsto(gx,gx')$.
\item If $x,x'\in \cX_\zeta X$ and both radial distances exceed $\varepsilon''$,
let $\Theta_{x,x'}$ be the push-forward of $\Gamma_{x,x'}$ under
$(g,h)\mapsto(gx,hx')$.
\item Otherwise, let
$\Theta_{x,x'}=P_\varepsilon(x,\cdot)\times
P_\varepsilon(x',\cdot)$.
\end{enumerate}
Each case has marginals $P_\varepsilon(x,\cdot)$ and
$P_\varepsilon(x',\cdot)$, and the resulting kernel is measurable.
Define the associated self-coupling operator by
$$
 \widetilde P_\varepsilon f(x,x')
 =\int_{X^2}f(y,y')\dd\Theta_{x,x'}(y,y').
$$
In cases (2) and (3), $\Theta_{x,x'}$ assigns zero mass to
$E(\varepsilon')^2$: this follows from the avoidance property of
$\Gamma_{x,x'}$ and from \eqref{eq:core-inclusions}, respectively.
\begin{proposition}\label{prop:drift}
The recoupled operator satisfies
\begin{align}
 \widetilde P_\varepsilon\Psi_\varepsilon
 &\leq\Psi_\varepsilon+\kappa_B
 &&\text{on }X^2,\label{eq:global-drift}\\
 \widetilde P_\varepsilon\Psi_\varepsilon
 &\leq\Psi_\varepsilon-\kappa_A
 &&\text{on }E(\varepsilon''')^2.
 \label{eq:inner-drift}
\end{align}
Here $\kappa_A$ and $\kappa_B$ are as in
\eqref{eq:margulisConstant}.
\end{proposition}

\begin{proof}
In cases (2) and (3), the bound \eqref{eq:cutoff-bound} gives
$$
 \widetilde P_\varepsilon\Psi_\varepsilon(x,x')
 \leq\log\Omega+\log2
 \leq\Psi_\varepsilon(x,x')+\log2
 \leq\Psi_\varepsilon(x,x')+\kappa_B.
$$

Consider case (1). At least one coordinate lies in $E(\varepsilon'')$,
so the input pair lies inside the cutoff. If $x=x'$, its energy is
infinite and the asserted upper bounds are automatic. Suppose
$x\neq x'$, and put $z=d(x,x')^{-1}$.
On the good set in \eqref{eq:pair-expansion}, the distance does not
decrease, and hence
$\Psi_\varepsilon(gx,gx')\leq\Psi_\varepsilon(x,x')$.
On the complementary set,
\eqref{eq:compactNbhdEstimate} gives
$$
 \Psi_\varepsilon(gx,gx')
 \leq\log(\Omega+e^{An}z)
 \leq\Psi_\varepsilon(x,x')+An.
$$
Therefore
$$
 \widetilde P_\varepsilon\Psi_\varepsilon(x,x')
 \leq\Psi_\varepsilon(x,x')+A\delta n
 \leq\Psi_\varepsilon(x,x')+\kappa_B,
$$
proving \eqref{eq:global-drift}.

Finally, if $x,x'\in E(\varepsilon''')$, then
$\varepsilon'''\leq\varepsilon''<\widetilde\varepsilon$, so
\eqref{eq:core-inclusions} places the pair in case (1).
For distinct such points,
\eqref{eq:kappa-comparison} implies
$e^{\kappa_0n/2}\geq e^{2\kappa_*n}$.
On the good set in \eqref{eq:pair-expansion}, if the image pair is
inside the cutoff, equations
\eqref{eq:inner-threshold}--\eqref{eq:log-contraction} give
$$
 \Psi_\varepsilon(gx,gx')
 \leq\log(\Omega+e^{-2\kappa_*n}z)
 \leq\Psi_\varepsilon(x,x')-\kappa_*n.
$$
If the image pair is outside the cutoff, the same bound follows from
$\Psi_\varepsilon(gx,gx')=\log\Omega$ and
$\Psi_\varepsilon(x,x')\geq\log\Omega+\kappa_*n$.
The complementary set has probability less than $\delta$ and contributes
a loss of at most $An$. Consequently,
$$
 \widetilde P_\varepsilon\Psi_\varepsilon(x,x')
 \leq\Psi_\varepsilon(x,x')-
       \bigl(\kappa_*-(A+\kappa_*)\delta\bigr)n
 \leq\Psi_\varepsilon(x,x')-\kappa_A.
$$
This proves \eqref{eq:inner-drift}.
\end{proof}
\begin{proof}[Proof of Theorem~B]
Fix an admissible $\varepsilon$, $\nu'$, and $\eta'$, and write
$$
 X=E(\varepsilon),\qquad
 F=E(\varepsilon'''),\qquad
 M=\eta'(X),\qquad m_0=\eta'(F).
$$
If $M=0$, then $m_0=0$ because $F\subset X$, and
\eqref{eq:balance} is immediate. We therefore assume $M>0$.

The kernel \eqref{eq:localized-kernel} leaves
$\mu=M^{-1}\eta'|_X$ invariant, and Proposition~\ref{prop:drift} gives
the drift estimates for $\widetilde P_\varepsilon$.
The measure $\mu$ is non-atomic. We use the minimum-energy argument of
\cite{V2020Dis}*{Lemma~A.5}, which requires only a measurable coupling
operator. The set of self-couplings of $\mu$ is weakly compact, and the
energy $\xi\mapsto\int\Psi_\varepsilon\dd\xi$ is lower semicontinuous.
Its minimum is finite: a non-atomic measure has a self-coupling vanishing
on a neighborhood of the diagonal, where $\Psi_\varepsilon$ is bounded
\cite{AEV2023}*{Corollary~6.11}. Let $\widetilde\mu$ be a minimizer.
Since $P_\varepsilon$ leaves $\mu$ invariant by lemma~\ref{lem:adapted-localization},
\cite{AEV2023}*{Lemma~6.18} shows that
$\widetilde P_\varepsilon^*\widetilde\mu$ is another self-coupling of
$\mu$. Minimality therefore gives
\begin{equation}\label{eq:finite-energy}
 \int_{X^2}\Psi_\varepsilon\dd\widetilde\mu<\infty,
 \qquad
 \int_{X^2}\widetilde P_\varepsilon\Psi_\varepsilon
      \dd\widetilde\mu
 \geq\int_{X^2}\Psi_\varepsilon\dd\widetilde\mu.
\end{equation}
Set $\widetilde m=M\widetilde\mu$. Both projections of $\widetilde m$
are the unnormalized restriction $\eta'|_X$, and
$\widetilde m(X^2)=M$.
Let $\cA=F^2$ and $\cB=X^2\setminus\cA$.
Integrating the drift inequalities against $\widetilde m$ gives
\begin{align*}
 \int_{X^2}\Psi_\varepsilon\dd\widetilde m
 &\leq\int_{X^2}\widetilde P_\varepsilon\Psi_\varepsilon
             \dd\widetilde m\\
 &\leq\int_{X^2}\Psi_\varepsilon\dd\widetilde m
       -\kappa_A\widetilde m(\cA)
       +\kappa_B\widetilde m(\cB).
\end{align*}
The common integral is finite by \eqref{eq:finite-energy}, so it can be
cancelled. Since $\widetilde m(\cB)=M-\widetilde m(\cA)$, we obtain
\begin{equation}\label{eq:inner-coupling-mass}
 \widetilde m(F^2)\leq
 \frac{\kappa_B}{\kappa_A+\kappa_B}M=\theta M.
\end{equation}

The projection identities give
$\widetilde m(F\times X)=\widetilde m(X\times F)=m_0$.
Therefore
\begin{equation}\label{eq:coupling-lower-bound}
 \widetilde m(F^2)\geq2m_0-M.
\end{equation}
Combining \eqref{eq:inner-coupling-mass} and
\eqref{eq:coupling-lower-bound}, and using \eqref{eq:relationKappaAB},
we find
$$
 m_0\leq\frac{1+\theta}{2}M
 <\frac{11}{20}M<\frac23M.
$$
Since $\varepsilon'''=q\varepsilon$, this proves
\eqref{eq:balance}.
\end{proof}
\begin{remark}\label{rem:nonatomic}
The argument uses only stationarity and non-atomicity of $\eta'$.
Consequently, Theorem~B holds for every non-atomic stationary probability
measure, without the genericity or top-exponent realization hypothesis.
\end{remark}